\documentclass[a4paper]{amsart}

\usepackage{amscd,amsthm,amsfonts,amssymb,esint, amsmath}
\usepackage{tikz-cd}
\usepackage[colorlinks=true]{hyperref}
\usepackage{fullpage}
\usepackage[all]{xy}
\usepackage{color}
\usepackage{comment}
\usepackage{enumitem}
\usepackage[mathscr]{euscript}
\usepackage{mathtools}
\usepackage{mathrsfs}

    \newcommand{\BA}{{\mathbb {A}}} 
    \newcommand{\BC}{{\mathbb {C}}} 
     
    \newcommand{\BG}{{\mathbb {G}}}

    \newcommand{\BQ}{{\mathbb {Q}}} \newcommand{\BR}{{\mathbb {R}}}

     \newcommand{\BZ}{{\mathbb {Z}}}

     \newcommand{\CH}{{\mathcal {H}}}
     \newcommand{\CJ}{{\mathcal {J}}}
     
     \newcommand{\CN}{{\mathcal {N}}}
    \newcommand{\CO}{{\mathcal {O}}} 
     
    \newcommand{\CS}{{\mathcal {S}}} \newcommand{\CT}{{\mathcal {T}}}

    \newcommand{\RM}{{\mathrm {M}}} \newcommand{\RN}{{\mathrm {N}}}
    \newcommand{\RO}{{\mathrm {O}}}

    \newcommand{\RU}{{\mathrm {U}}}

    \newcommand{\fg}{{\mathfrak{g}}} \newcommand{\fh}{{\mathfrak{h}}}
     
     \newcommand{\fl}{{\mathfrak{l}}}

    \newcommand{\fs}{{\mathfrak{s}}}

    \newcommand{\R}{{\mathbb {R}}} 
        \newcommand{\Q}{{\mathbb {Q}}} 
            \newcommand{\Z}{{\mathbb {Z}}} 
                    \newcommand{\C}{{\mathbb {C}}}

    \newcommand{\Ad}{{\mathrm{Ad}}}

    \newcommand{\Aut}{{\mathrm{Aut}}}

      \newcommand{\Id}{{\mathrm{Id}}}

    \newcommand{\Br}{{\mathrm{Br}}}

    \newcommand{\Cl}{{\mathrm{Cl}}}

    \newcommand{\diag}{{\mathrm{diag}}}

     \newcommand{\fin}{{\mathrm{fin}}}

    \newcommand{\Gal}{{\mathrm{Gal}}} \newcommand{\GL}{{\mathrm{GL}}}
    
    \newcommand{\Hom}{{\mathrm{Hom}}}
    
    \renewcommand{\Im}{{\mathrm{Im}}}
    \newcommand{\Ind}{{\mathrm{Ind}}}
    
    \newcommand{\inv}{{\mathrm{inv}}}

     \newcommand{\rank}{{\mathrm{rank}}}
    \newcommand{\PGL}{{\mathrm{PGL}}} \newcommand{\Pic}{\mathrm{Pic}}

    \renewcommand{\Re}{{\mathrm{Re}}}

    \newcommand{\Res}{{\mathrm{Res}}}

    \newcommand{\st}{{\mathrm{st}}}

    \newcommand{\SL}{{\mathrm{SL}}}
    \newcommand{\Spec}{{\mathrm{Spec}}} 
    \newcommand{\SO}{{\mathrm{SO}}}
    \newcommand{\St}{{\mathrm{St}}}\newcommand{\SU}{{\mathrm{SU}}}
    \newcommand{\Sym}{{\mathrm{Sym}}}\newcommand{\sgn}{{\mathrm{sgn}}}
    \newcommand{\Stab}{{\mathrm{Stab}}}

    \newcommand{\tr}{{\mathrm{tr}}}

    \newcommand{\Vol}{{\mathrm{Vol}}}

 \renewcommand{\tilde}{\widetilde}

\newcommand{\matrixx}[4]{\begin{pmatrix}
#1 & #2 \\ #3 & #4
\end{pmatrix} }        

    \font \cyr=wncyr10

    \newcommand{\Sha}{\hbox{\cyr X}}\newcommand{\wt}{\widetilde}
    \newcommand{\wh}{\widehat}

    \newcommand{\norm}[1]{\|{#1}\|} 
    \newcommand{\ov}{\overline}

    \newcommand{\lra}{\longrightarrow}
    \newcommand{\ra}{\rightarrow} 
    \newcommand{\bs}{\backslash}

    \theoremstyle{plain}
    \newtheorem{thm}{Theorem}[section] \newtheorem{cor}[thm]{Corollary}
    \newtheorem{lem}[thm]{Lemma}  \newtheorem{prop}[thm]{Proposition}

\theoremstyle{remark} \newtheorem{remark}[thm]{Remark}
\theoremstyle{remark} 
\theoremstyle{remark} \newtheorem{example}[thm]{Example}
\theoremstyle{remark} 

\theoremstyle{remark} \newtheorem{assumption}[thm]{Assumption}

    \newcommand{\etale}{\'{e}tale~}
    \newcommand{\et}{\'{e}t}
    \renewcommand{\et}{{\text{\'{e}t}}}
    
    \numberwithin{equation}{section}

    \newcommand{\Gm}{\BG_\mathrm{m}}

     \newcommand{\Tam}{{\mathrm{Tam}}}

     \newcommand{\A}{{\mathbb {A}}}
      \newcommand{\tors}{{\mathrm{tors}}}

   \newcommand{\normm}[1]{\left\lvert#1\right\rvert} 

\subjclass[2020]{Primary:11R54,11F72; Secondary: 37A17,14G05}     

\keywords{
Integral points,
asymptotic counting,
orders in number fields,
zeta functions of orders,
(twisted) orbital integrals,
endoscopy theory,  fundamental lemma,
equi-distribution of lattice orbits
}

\begin{document}

\title{Distributions of Integral Points and Dedekind Zeta Values}

\author{Li Cai} 
\address[L.C.]{Academy for Multidisciplinary Studies, Beijing National Center for Applied Mathematics, Capital Normal University, Beijing, 100048, People's Republic of China
} 
\email{caili@cnu.edu.cn}

\author{Taiwang  Deng}
\address[T.D.]{Yanqi Lake Beijing Institute of Mathematical Sciences and Applications (BIMSA), Huairou District, 100084, Beijing\\
China}
\email{dengtaiw@bimsa.cn}

\begin{abstract}
Let $\CO$ be the ring of integers for some number field $F$. Let $\chi(x)\in \CO[x]$ be a regular monic polynomial of degree $n$. We study the asymptotic count of integral $n\times n$ matrices over $\CO$ with the characteristic polynomial $\chi$ and bounded archimedean norm. Previous works establish such an asymptotic with a positive leading constant. Our main result determines this constant  in terms of the leading Laurent coefficients at $s=1$ of Dedekind zeta functions attached to orders in  $F[x]/(\chi(x))$.  
The proof combines a refinement of the equi-distribution property of  orbits with a reformulation of the counting problem in terms of generalized $\kappa$-orbital integrals. These orbital integrals are then transferred by the endoscopic fundamental lemma and related to zeta functions of orders.
\end{abstract}

\maketitle

\tableofcontents

\section{Introduction}

Let $F$ be a number field with the ring of integers $\CO$. Let $X$ be an affine variety over $\CO$. For each $\CT > 0$, consider the counting function
\[N(X,\CT) = \#\Big\{ x \in X(\CO) \Big| ||x|| \leq \CT\Big\}\]
where $||\cdot||$ is a certain norm on $X(F)$.

It is a basic problem to determine the asymptotic behavior of
$N(X,\CT)$ as $\CT \ra \infty$ (see \cite{Sar91}).

If there are sufficiently many additive terms in the equations defining $X$, then one may consider
the Hardy-Littlewood method. One obtains that (see \cite{Sch85}) 
\[\lim_{\CT \ra +\infty} \frac{N(X,\CT)}{C\CT^d} = 1\]
for some natural number $d$ and some positive number $C$. The constant $C$ is a product of local densities.

When $X$ is a homogeneous space, there are more results. 

By applying the spectral analysis on the automorphic space, the work of 
Duke-Rudnick-Sarnak \cite{DRS93} considers the counting problem for the case $X$ is a symmetric space. 

Unlike the cases where the Hardy-Littlewood circle method is applicable, the leading coefficient cannot be 
written as a product of local integrals in general. This phenomenon was first systematically studied in the work of Borovoi-Rudnick \cite{BR95}. 

In this paper, we shall focus on a special case of homogeneous spaces. Let $\chi(x) \in \CO[x]$ be a monic polynomial of degree $n$. Denote by $X$ the variety
which represents the set of $n \times n$ matrices  with characteristic polynomial  $\chi(x)$. 
In other words, for any $\CO$-algebra $R$,
\[X(R) = \left\{ A \in M_{n \times n}(R) \Big| \det(xI_n - A) = \chi(x) \right\}.\]
We consider the following norm on $M_{n \times n}(F)$
\[||x|| = \max_{v|\infty} ||x||_v, \quad ||x||_v = \left( \sum_{i,j}^n |x_{ij}|_v^{2/e_v} \right)^{1/2}.\]
Here, we normalize the absolute value $|\cdot|_v$ on $F_v$ such that if $F_v = \BR$, then it is the usual one and if $F_v = \BC$, then
 it is the square of the usual one. We denote by $e_v = 1$ if $F_v = \BR$
 and $e_v = 2$ if $F_v = \BC$. 

We call $\chi(x)$ regular if it has distinct roots in any algebraic closure of $F$. In the following, we assume $\chi(x)$ is regular.

Using the ergodic theory, especially the theory of unipotent flows on homogeneous spaces, 
the works of Eskin-Mozes-Shah\cite{EMS}, Shapira-Zheng\cite{SZ19}, and Zhang\cite{Zhang} establish the following asymptotic result.
There exists a positive real number $C(\chi)$ such that (see \cite[Theorem 1.5]{Zhang})
\[\lim_{\CT \ra +\infty}\frac{N(X,\CT)}{C(\chi)\CT^{[F:\BQ]d} (\log \CT)^{r}} = 1. \]
In the above
\begin{itemize}
  \item $d = \frac{n(n-1)}{2}$.
  \item $\chi(x) = \chi_0(x)\chi_1(x)\cdots \chi_r(x)$ where
  $\chi_i(x)$ is irreducible  with degree $n_i$. 
\end{itemize}

It is a natural question to determine the constant $C(\chi)$ in terms of arithmetic invariants.

Denote by $K = F[x]/(\chi(x))$. Then $K$ is an \etale algebra extension of $F$ of degree $n$. Denote by  $R = \CO[x]/(\chi(x))$. Then $R$ is an order in $K$. Here, an ($\CO$-)order in $K$ is a finitely generated $\CO$-algebra $R$ in $K$  
with $R \otimes_\CO F = K$.

Assume $\chi$ is irreducible (so that $K$ is a number field), $F = \BQ$ and $\CO_K = R$. Then it is a theorem of
Eskin-Mozes-Shah \cite[Theorem 1.1]{EMS} and Shah \cite[Theorem 1.2]{Sha} that
\[C(\chi) =  \Res_{s=1} \zeta_K(s)\cdot\frac{w_n}{ \prod_{k=2}^n \Lambda(k)}.\]
Here,  $w_n$ is the volume of the unit ball in $\BR^d$ and $\Lambda(s) = \pi^{-s/2}\Gamma(s/2)\zeta(s)$ is the complete Riemann zeta function.

In general, we shall relate $C(\chi)$ to the zeta values of orders.

The zeta function $\zeta_R$ for the order $R$ is defined as (see \cite{Yun})
\[\zeta_R(s) =  \sum_{M \subset R^\vee} \left( \# R^\vee/M \right)^{-s}\]
where $M$ runs over all nonzero  $R$-modules contained in $R^\vee = \Hom_\BZ(R,\BZ)$. 
When $K$ is a number field and $R = \CO_K$, $\zeta_R$ is the usual Dedekind zeta function. 
The introduction of these zeta functions enables us to determine $C(\chi)$ for any regular $\chi(x)$. 

Similar to the Dedekind zeta function, the function $\zeta_R$
admits a meromorphic continuation to the whole $s$-plane with a pole at $s=1$ of degree $r+1$ (in particular,
if $\chi$ is irreducible, it has a simple pole at $s=1$). Moreover, the leading coefficient of $\zeta_R$ at $s=1$
has a formula which is an analogue to the one for the Dedekind zeta function (see Proposition \ref{prop-residue-formula}).

When $F = \BQ$, our result is the following one. 
\begin{thm}\label{thm-main-Q}
	Assume $F = \BQ$ and $\chi(x)$ is regular. 
    Then
	\[C(\chi) =   \lim_{s\ra 1} (s-1)^{r+1} \zeta_{R}(s) \cdot \frac{v_\chi w_n}{\prod_{k=2}^n \Lambda(k)}.\]
	Here,
    \[v_\chi = n^{r-1} \prod_{i=0}^r n_i.\]
    In particular, if $\chi$ is irreducible, then $v_\chi = 1$ and if $\chi$ is completely reducible, then $v_\chi = n^{n-2}$.
\end{thm}

For general $F$, we have the following.
\begin{thm}\label{thm-main}
  Assume $\chi(x)$ is regular. Then
  \[C(\chi) =   \sum_E \phi([E:F])\delta(E) \lim_{s\ra 1} (s-1)^{r+1} \zeta_{R_E}(s) \cdot
   \frac{ v_\chi w_n}
   {|\Delta_F|^{d}\, \pi^{nr_2} \, \Res_{s=1}\Lambda_F(s) \prod_{k=2}^n \Lambda_F(k)}.\]
  Here,  
  \begin{itemize}   
    \item $E$ runs over fields $F \subset E \subset K_i$, $0 \leq i \leq r$  such that
	    $E/F$ is cyclic and unramified everywhere (including the archimedean places). Here, we write $K = \prod_{i=0}^r K_i$ 
	    with $K_i = F[x]/(\chi_i(x))$. We fix embeddings $K_i$, $0 \leq i \leq r$ to the splitting field of $\chi$. 
    \item The function $\phi$ is the Euler totient function.
    \item $\delta(E) = \prod_{v<\infty} \delta_v(E)$ with $\delta_v(E) =  [\CO_{E_v}^1: \CO_{E_v}^1 \cap \RN_{K_v/E_v} \CO_{K_v}^\times]$
	    (see Lemma \ref{lem-c_G} for notations). If $E = F$ or $K/E$ is unramified, then $\delta(E) = 1$.
    \item For each $E$, $R_E = \CO_E[x]/(\chi_E(x))$ where $\chi_E(x) \in \CO_E[x]$ is an irreducible factor of $\chi(x)$
	    in $E$. Then $R_E$ is an order in $K$.
    \item $v_\chi = [F:\BQ]^rn^{r-1} \prod_{i=0}^r n_i$ is the volume of certain polytope defined in Proposition \ref{prop: c_T} and  determined in Proposition \ref{prop:volume-polytope}. 
    \item $r_2$ is the number of complex
    places of $F$.
   \item $w_n = \prod_{v|\infty} w_{n,v}$  with $w_{n,v}$ the volume for the unit ball 
	   \[B_{n,v} = \left\{ (x_j) \in F_v^{d} \Big| \sum_j |x_j|_v^{2/e_v} \leq 1 \right\}.\]
   \item $\Lambda_F(s) = \prod_{v|\infty} \Gamma_v(s) \zeta_F(s)$ is the complete Dedekind zeta function for $F$. Here,
	   if $v$ is real, then  $\Gamma_v(s) = \pi^{-s/2}\Gamma(s/2)$ and if $v$ is complex, then $\Gamma_v(s) = 2(2\pi)^{-s}\Gamma(s)$.
   \end{itemize} 
\end{thm}

\begin{remark}
	The result is independent of the choice of $\chi_E$ for each $E$. In fact, two different such $\chi_E$ are Galois
	conjugate  and their corresponding orders $R_E$ are also Galois conjugate. In particular,
	 they share the same zeta function.
\end{remark}

\begin{remark}
Given Theorem \ref{thm-main} above, it seems quite interesting and feasible to devise an algorithm to compute the leading coefficient $C(\chi)$.
\end{remark}

Previously, there are already many important works to generalize the above result of  Eskin-Mozes-Shah.

For general $F$ and irreducible $\chi(x)$, the work of Wei-Xu \cite{WX} relates $C(\chi)$ to certain local integrals indexed
by the Brauer group of $X$. In particular, they give the
difference between $C(\chi)$ and an Euler product.  Their work 
reduces the determination of $C(\chi)$ into a local problem. 

Recently, Jeon-Lee \cite{JL24B} and Lee \cite{Lee} observe that these local integrals  are 
 $\kappa$-orbital integrals on $\SL_n$ when $\chi(0)=1$. Leaving aside zeta functions for orders,  when
    $F \not = \BQ$, Theorem \ref{thm-main} is obtained in  \cite{Lee} for $\chi$ irreducible and $K/F$ unramified everywhere 
    under the condition that 
    \begin{enumerate}
    \item $\chi(0) = 1$; 
    \item $K/F$ is a cyclic extension with prime degree. 
	\end{enumerate}
    In particular, under the second condition, the only two choices of $E$ are $F$ and $K$.
    The paper \cite{JL24B} is aimed at removing the first condition that $\chi(0) = 1$ (but also assumes the 
    second condition).  The first condition is weakened
    as $\chi(0) = 1$ or $\chi(x)$ splits over $F_v$ at those $p$-adic places $v$ of $F$
    for $p \leq n$.
In general, it is not expected to give an explicit formula for these
 $\kappa$-orbital integrals even if we can relate them to stable orbital integrals on endoscopic subgroups via the
 celebrated fundamental lemma.

Assume $\chi(x) = \chi_0(x)$ is completely reducible. When
$F = \BQ$, $C(\chi)$ is determined by Shapira-Zheng \cite{SZ19} (without an explicit formula for the  volume $v_\chi$). For general $F$ and $n=2$,
the work of Xu-Zhang \cite{XZ} relates $C(\chi)$ to certain
local integrals. Note that in this case, the
 Brauer group of $X$ is trivial.

It is the perspective of relating the constant $C(\chi)$ to zeta values of orders that allows us to eliminate all assumptions in previous work.

The proof can be sketched as  follows. Denote by $G = \SL_n$ and $\wt{G} = \GL_n$. Also denote by $\fg = \fs\fl_n$ and 
$\wt{\fg} = \fg\fl_n$ their Lie algebras. Fix $\gamma \in \wt{\fg}(\CO)$ with 
characteristic polynomial $\chi(x)$. We consider the conjugation action of $\wt{G}$ and also its subgroup $G$ on $\wt{\fg}$. 
Denote by $T$ and $\wt{T}$ the stabilizer of $\gamma$ in $G$ and $\wt{G}$ respectively. Then $X = T \bs G = \wt{T} \bs \wt{G}$. 
\begin{enumerate}
	\item Consider the action of $\Gamma = G(F) \cap \mathrm{St}(X)$ on $X(\CO)$. Here, $\mathrm{St}(X)$
		is a compact subgroup of $G(\BA)$ stabilizing $X(\CO_v)$ for each $v<\infty$. 
		For the contribution of each $\Gamma$-orbit in $N(X,\CT)$, we have the following important equi-distribution property
		(see Theorem \ref{thm: equi-distribution}): for any $\gamma' \in X(\CO)$
		\[\# \left\{y \in \gamma' \cdot \Gamma \Big| ||y|| \leq \CT  \right\} \sim
		v_\chi \frac{\Vol(\Gamma \cap T_{\gamma'}(F) \bs T_{\gamma'}(F_\infty)^1)}{\Vol(\Gamma \bs G(F_\infty))} 
		\Vol\left(\left\{ x \in \gamma' \cdot G(F_\infty) \Big| ||x|| \leq \CT \right\}  \right) (\log \CT)^r.\]
        Here, the notation $A(\CT) \sim B(\CT)$ means that $\lim_{\CT \ra +\infty} \frac{A(\CT)}{B(\CT)} = 1$. This equi-distribution
	property is a refinement of the main result in \cite{Zhang}.
	\item Using basic properties of the Brauer-Manin obstruction on $X$, the decomposition of $N(X,\CT)$ via
		the $\Gamma$-orbits of $X(\CO)$ can be expressed as a sum of products of local integrals indexed by the Brauer group
		of $X$ (see Theorem \ref{thm-wx}). This step is a  generalization of the work of Wei-Xu \cite{WX} for the elliptic 
		case. Moreover, by the Tate-Nakayama duality, 
		\[\Br(X)/\Br(F) =  H^1(F,T(\BA_{\bar{F}})/T(\bar{F}))^\vee.\]
		This suggests that these local integrals are exactly the (generalized) $\kappa$-orbital integrals for the conjugation action 
		of $G$ on $\wt{\fg}$ where $\kappa$ is a character on $H^1(F,T(\BA_{\bar{F}})/T(\bar{F}))$.
		Precisely (see Theorem \ref{thm-orbit}), we have
		\[N(X,\CT) \sim v_\chi C_0^\Tam \sum_\kappa \CO_\gamma^\kappa(f^\CT) (\log \CT)^r\]
		where $C_0^\Tam$ is a constant coming from the comparison of different measures and 
		$f^\CT = f_\infty^\CT \otimes 1_{\wt{\fg}(\wh{\CO})} \in C_c^\infty(\wt{\fg}(\BA))$ is a certain test function on $\wt{\fg}(\BA)$.
	\item   Note that the above $\kappa$-orbital integrals are in fact certain twisted orbital integrals on $\wt{\fg}$ 
		(see Proposition \ref{lem:kappa-inv}). 
		By the fundamental lemma (see Theorem \ref{fl-LW}) 
		relating the twisted orbital integrals on $\wt{G}$ to  orbital integrals
		on  its endoscopic subgroups given by Lemaire-Waldspurger \cite{LW17}, for each $\kappa$,
		\[\CO_\gamma^\kappa(1_{\wt{\fg}(\wh{\CO})}) \stackrel{.}{=} \CO_{\gamma_H}(1_{\wt{\fh}(\wh{\CO})}).\]
		Here, $H$ is the endoscopic subgroup of $G$ given by $\kappa$,
        $1_{\wt{\fh}(\wh{\CO})}$ is the characteristic function of $\wt{\fh}(\wh{\CO})$ with $\wh{\CO} = \prod_{v<\infty} \CO_v$ and $\CO_{\gamma_H}(1_{\wt{\fh}(\wh{\CO})})$ is 
		the usual  orbital integral on $\wt{H}$:  $H = \Res_{E/F}^1 \GL_m$ and $\wt{H} = \Res_{E/F} \GL_m$ 
		where $E/F$ is  unramified everywhere and $m=[K:E]$. Denote by $\wt{\fh} = \Res_{E/F} \fg\fl_m$
        the Lie algebra of $\wt{H}$.
	\item   Finally, we have the following relation (see Theorem \ref{zeta-orb}) 
		between zeta functions and orbital integrals given by Yun \cite{Yun}
		\[ \frac{\zeta_{R}(s)}{\zeta_K(s)}\Big|_{s=1} =  \frac{\CO_{\gamma}(1_{\wt{\fg}(\wh{\CO})})}{[\CO_K:R]}.\]
		By applying such relations to each endoscopic subgroups $H$ of $G$, we obtain the main result.

\end{enumerate}

The above steps (1,2,3) are inspired by and parallel to the proofs in \cite{JL24B, Lee}. 

We discuss some technical details in the proof. 

For Step (1), it is necessary to refine the equi-distribution property in \cite{Zhang}. In Proposition \ref{prop: c_T}, the asymptotic behavior of the integral for the volume of a certain truncation window is expressed in terms of the volume $v_\chi$ of a certain polytope $P_\chi$ and a term involving powers of the logarithm. Following the method in \cite{SZ19}, the proof of this proposition is given in the appendix.

One needs to study the volume $v_\chi$. By investigating the polytopes induced by submodular functions, we can reduce the problem to certain graphical zonotopes
(see Lemma \ref{lem-submodular}). Combining this with standard results in graph theory (such as Cayley’s identity), we obtain the volume of these polytopes. It is interesting that  such graphical zonotopes also appear in Arthur’s trace formula (see \cite{FL11}).

Note that in \cite{JL24B} there is the condition that $\chi(0) = 1$ or $\chi(x)$ splits over $F_v$ at those $p$-adic places $v$ of $F$
for $p \leq n$. Under such a condition, it is enough to consider the conjugation action of $G$ on $G$. 
To remove this condition, we replace this action by the conjugation action of $G$ on $\wt{\fg}$. As mentioned above, 
the corresponding (generalized) $\kappa$-orbital integrals are exactly certain twisted orbital integrals on $\wt{\fg}$ which can also be
related to orbital integrals on endoscopic groups by the fundamental lemma.  Moreover, to remove the restriction of residue characteristic
for the fundamental lemma of Lie algebras, we apply the fundamental lemma of groups for full Hecke algebras (see the proof of Theorem \ref{fl}.)

The zeta functions in step (4) have a close relation with the theory of beyond endoscopy \cite{Arthur18}, which is a strategy Langlands advanced to apply the trace formula to the general principle of functoriality.
In the strategy, certain kind of Poisson summation formula should be applied to the elliptic part of the trace formula. For this, as observed by Arthur, one may study analytic behaviors of orbital integrals by realizing them as residues of these zeta functions. Based on this observation, the second named author  and Espinosa generalize the work \cite{AliI}  to $\GL_3(\Q)$ \cite{DM26}, which isolates the contribution of the trivial representation from the regular elliptic part of the trace formula via a Poisson summation formula.
From this point of view,  the summation of  $N(X_\gamma, \CT)$ over the stable elliptic orbits $\gamma$ by applying our main Theorem \ref{thm-main}  may provide interesting connections to the theory of beyond endoscopy.

\subsection*{Acknowledgements}
\addtocontents{toc}{\protect\setcounter{tocdepth}{1}}

We thank Professor Fei Xu for bringing this question to our attention and for his encouragement during the preparation of the paper.
We thank Professor Runlin Zhang for contributing the proof of Proposition \ref{prop: c_T} in the appendix.
L. Cai is supported by the National Key R\&D Program of China No.
2023YFA1009702 and National Natural Science Foundation of
China, No. 12371012. T. Deng
 is supported by National Natural Science Foundation of
China, No. 12401013 and Beijing Natural Science Foundation, No. 1244042.

\section{Endoscopy for $\SL_n$ and twisted endoscopy for $\GL_n$}

In this section, we first recall the  endoscopy for general quasi-split groups. Then we specialize to  $\SL_n$. We also recall the relation between endoscopy for $\SL_n$ and twisted endoscopy for $(\GL_n, \varepsilon)$ with $\varepsilon$
a character on $\GL_n$.

Let $F$ be a number field with $\Gamma = \Gal(\bar{F}/F)$ the absolute Galois group.
Let $G$ be a quasi-split connected reductive group over a number field $F$. We fix a Borel pair $(T_0,B_0)$ where
$B_0$ is a Borel subgroup of $G_0$ and $T_0$ is a maximal torus in $B_0$. Denote by  $(X^*(T_0), \Delta, X_*(T_0), \Delta^\vee)$
the based root datum of $(G,B_0,T_0)$. Here, $X^*(T_0)$ and $X_*(T_0)$ are the groups of characters and co-characters on $T_0$
with $\Delta$ and $\Delta^\vee$ the sets of simple roots and simple coroots. There is an action of $\Gamma$ on $X^*(T_0)$ and 
$X_*(T_0)$ such that $\Delta$, $\Delta^\vee$ and the canonical pairing on $X^*(T_0) \times X_*(T_0)$ are all invariant.

The connected Langlands dual group $\wh{G}$ is defined to be the unique triple $(\wh{G},\wh{B_0},\wh{T_0})$, defined over $\BC$,
with the based root datum $(X_*(T_0),\Delta^\vee,X^*(T_0),\Delta)$. We fix a pinning 
$\left(\wh{T_0},\wh{B_0},\{ \wh{X}_{\alpha^\vee} \}_{\alpha^\vee} \right)$ where  $\alpha^\vee$ runs over
the set of simple coroots $\Delta^\vee$ and $\wh{X}_{\alpha^\vee} \not=0$ in the $1$-dimensional subspace of the Lie algebra of $\wh{G}$
corresponding to $\alpha^\vee$. The action of $\Gamma$ on the based root system $(X_*(T_0),\Delta^\vee,X^*(T_0),\Delta)$ lifts uniquely
to an automorphism of $\wh{G}$ keeping the pinning. 

Let $T$ be a maximal torus of $G$. Then there is a canonical $\wh{G}$-invariant class of embeddings 
\[\iota: \wh{T} \lra \wh{G}\]
which is $\Gamma$-equivariant up to inner automorphisms (see \cite[Construction 3.2.4]{Kal}). In other words, for any $\sigma \in \Gamma$, there is $g \in \wh{G}$ such that
\[\sigma_G \circ \iota \circ \sigma_T^{-1} = \Ad(g) \circ \iota.\]
Here, $\sigma_G$ and $\sigma_T$ denote the actions of $\sigma \in \Gamma$ on $\wh{G}$ and $\wh{T}$. The action of $\Gamma$ on $\wh{T}$ 
is given via the isomorphism $\wh{T} = X^*(T) \otimes \BC^\times$.

Following \cite[Section 7.4]{Kot84}, an endoscopic triple $(H,s,\eta)$ consists of
\begin{itemize}
	\item a quasi-split connected reductive group $H$;
	\item an embedding $\eta: \wh{H} \ra \wh{G}$, where $\wh{H}$ and $\wh{G}$ are the connected Langlands dual groups of $H$
		and $G$ defined over $\bar{F}$;
	\item an element $s \in Z(\wh{H})$.
\end{itemize}
The triple is required to satisfy the conditions \cite[(7.4.1,7.4.2,7.4.3)]{Kot84}.

An endoscopic datum for $G$ is a pair $(s,\rho)$. Here,
$s$ is a semisimple element in $\wh{G}/Z(\wh{G})$ and $\rho:\Gamma \ra 
\mathrm{Out}(\wh{G}_s^0)$ is a homomorphism. Here, $\wh{G}_s^0$ is 
the connected component of the centralizer of $s$ in $\wh{G}$.
The pair is required to satisfy the conditions \cite[(7.1.1,7.1.2)]{Kot84}. 

There is a bijection between the set of isomorphism classes of endoscopic triples and the set of isomorphism classes of endoscopic datum (see \cite[Section 7.6]{Kot84}).

In the following, denote by $G = \SL_n$ and $\wt{G} = \GL_n$.

For our application, we will also need  the twisted version (\cite[(2.1.1)-(2.1.4)]{KS99}) of the above definition of endoscopy data  for $(\widetilde G, \varepsilon)$, where  $\varepsilon$ is a  character on $\wt{G}(F_v)$ or $\wt{G}(\BA_F)$.
In fact, for our case, as is explained in \cite[Page 3]{HS12}, an endoscopic 
data $(\tilde{H}, \tilde{s}, \tilde{\eta})$ for $(\tilde G, \varepsilon)$ is naturally related to  
an endoscopy data $(H, s, \eta)$ for $G$ through restriction and vice versa. For this reason and to avoid further complications, in the rest of the section we restrict ourselves to consider the endoscopy for $G$.

Let $\gamma \in G(F)$ be regular semisimple. Denote by
$T = T_\gamma$ the stabilizer of $\gamma$ in $G$ and 
$\wt{T}$ the stabilizer of $\gamma$ in $\wt{G}$.
Write 
\[K=F[\gamma]\cong  K_0^{\ell}\times K_1\times K_2\times \cdots \times K_a\] 
with $K_0 = F$ and $[K_i:F]=\ell_i>1$ for $1\leq i\leq a$. Then
\begin{equation}\label{eq: decomposition-torus}
\wt{T}=\Gm^{\ell}\times \Res_{K_1/F}\Gm\times\cdots \times \Res_{K_a/F}\Gm
\end{equation}
and 
\[T=\ker(\RN:\wt{T}\to\Gm ), \quad \RN\bigl((t_j)_j,(x_i)_i\bigr)=\bigl(\prod_{j=1}^{\ell}t_j\bigr)\cdot\prod_{i=1}^a \RN_{K_i/F}(x_i).\]
In the following, we fix an embedding $\iota: \wh{T} \hookrightarrow \wh{G}$ in the canonical $\wh{G}$-conjugacy class. 

Consider the cohomology group
\[H^1(F,T(\mathbb{A}_{\overline{F}})/T(\overline{F})).\]
Here, $\BA_{\bar{F}} = \BA \otimes_F \bar{F}$ is the direct limit of $\BA_L$ where  $L/F$ are extensions of number fields. 

In the following, for a character $\kappa$  on $H^1(F,T(\mathbb{A}_{\overline{F}})/T(\overline{F}))$,
we shall associate an endoscopic triple $(H,s,\eta)$ and an element $\gamma_H \in H(F)$. When $\gamma$ is elliptic, this is \cite[Lemma 9.7] {Kott}. 

In fact, by the Tate-Nakayama duality  \cite[(2.1.1)]{Kott}, we have 
\[
  \pi_{0}\bigl(\widehat{T}^{\Gamma}\bigr)^{\vee}
    \cong
  H^1(F,T(\BA_{\bar{F}})/T(\bar{F})).
\]
We may view $\kappa$ as an element in $\pi_{0}\bigl(\widehat{T}^{\Gamma}\bigr)$. If $\gamma$ is elliptic (and $T$ is anisotropic), 
the group $\wh{T}^\Gamma$ is finite so that $\pi_0\left( \wh{T}^\Gamma \right) = \wh{T}^\Gamma$. Therefore, $\kappa$ gives
an element $s \in \wh{T}^\Gamma$. We may consider $\wh{H} = \wh{G}_s^0$. In general, we shall choose the lifting $s$ of $\kappa$ 
satisfying the condition that $\wh{G}_s^0$ is maximal (See Lemma \ref{lem:s}). 

Moreover, we shall classify these characters $\kappa$ in terms of subfields of $K$ (see Corollary \ref{cor:kappa}). 
With our choice of $s$, this determines the endoscopic subgroup $H$ and the element $\gamma_H \in H(F)$ associated to $\kappa$ (see Proposition \ref{ell-endo}).

\begin{lem}\label{lem-duality-picard}
We have $ H^1(F,T(\BA_{\bar{F}})/T(\bar{F}))^\vee \cong \Pic(T)$.
\end{lem}
\begin{proof}

Let $\pi_{1}(T)$ denote the algebraic fundamental group of $T$
(as in \cite{BR95}).  By \cite[Proposition~6.3]{CT} there is a canonical
perfect pairing
\[
  \Pic(T)\times(\pi_{1}(T)_{\Gamma})_{\tors}
  \longrightarrow\BQ/\BZ,
\]
where $\Gamma = \Gal(\overline{F}/F)$ and $(\cdot)_{\tors}$ denotes the torsion
subgroup. Hence we obtain an isomorphism
\begin{equation}\label{eq:Pic-dual-pi1}
  \Pic(T)^{\vee}
    \cong
  (\pi_{1}(T)_{\Gamma})_{\tors}.
\end{equation}
Let $\widehat{T}$ be the connected complex Langlands dual group of $T$.
Borovoi–Rudnick identify the torsion of $\pi_{1}(T)_{\Gamma}$ with the
component group of $\widehat{T}^{\Gamma}$ (cf. \cite[Section~3.4]{BR95}):
\begin{equation}\label{eq:BR}
  (\pi_{1}(T)_{\Gamma})_{\tors}
\cong\pi_{0}\bigl(\widehat{T}^{\Gamma}\bigr)^{\vee}.
\end{equation}
  Thus, combining
\eqref{eq:Pic-dual-pi1} and \eqref{eq:BR}, we have
\begin{equation}\label{eq:Pic-dual-pi0Hhat}
  \Pic(T)^{\vee}
    \cong
\pi_{0}\bigl(\widehat{T}^{\Gamma}\bigr)^{\vee}.
\end{equation}
Finally, combining the Tate-Nakayama duality \cite[(2.1.1)]{Kott} with \eqref{eq:Pic-dual-pi0Hhat},  we have 
\[
  H^1(F,T(\BA_{\bar{F}})/T(\bar{F}))\;\cong\;\Pic(T)^{\vee}.
\]
\end{proof}

We fix once and for all a splitting field $L$ for $\chi$.  Write 
$\Lambda:=\Gal(L/F)$. Moreover,
for each $i$ we fix an $F$-embedding $K_i\hookrightarrow L$
and set $\Lambda_i:=\Gal(L/K_i)$.

\begin{prop}\label{prop: Picard-general-torus}
The character lattice $X^*(T)$ of $T$ fits into an exact sequence of $\Gamma$-lattices
\begin{equation}\label{eq:char-seq}
0\longrightarrow \Z \xrightarrow{\ \Delta\ }
\Bigl(\Z^{\ell}\bigoplus\bigoplus_{i=1}^a \Z[\Lambda/\Lambda_i]\Bigr)
\longrightarrow X^*(T)\longrightarrow 0,
\end{equation}
where $\Delta(1)=(1,\dots,1;\mathbf 1_1,\dots,\mathbf 1_a)$ and
$\mathbf 1_i:=\sum_{g\Lambda_i\in \Lambda/\Lambda_i} g\Lambda_i$.

Moreover,
\[
\Pic(T)\ \cong\ H^1(\Lambda,X^*(T))
\ \cong\ \ker\!\Bigl(H^2(\Lambda,\Z)\longrightarrow
H^2(\Lambda,\Z)^{\ell}\bigoplus\bigoplus_{i=1}^a H^2(\Lambda_i,\Z)\Bigr),
\]
where the morphism on the right hand side is given by 
\[
c\ \longmapsto\ \Bigl(\underbrace{c,\dots,c}_{\ell\ \text{times}},
\ \Res_{\Lambda_1}(c),\dots,\Res_{\Lambda_a}(c)\Bigr).
\]
Here, for each $i$, $\Res_{\Lambda_i}(c)$ is the restriction of $c$
to $\Lambda_i$.

In particular, if $\ell>0$ then
\[
\Pic(T)=0.
\]
If $\ell=0$, by $H^2(\Lambda,\Z)\cong H^1(\Lambda,\Q/\Z)\cong\Hom(\Lambda,\Q/\Z)$, we have
\[
\Pic(T)\ \cong\
\left\{\chi\in\Hom(\Lambda,\Q/\Z)\Big| \chi|_{\Lambda_i}=0, \forall i\right\}.
\]
\end{prop}

\begin{proof}
We have a short exact sequence of $F$-tori
\begin{equation}\label{eq:tori-seq}
1\longrightarrow T\longrightarrow \wt{T}\xrightarrow{\ N\ } \Gm\longrightarrow 1.
\end{equation}
Let $L/F$ be the splitting field of $\chi$, and put
$\Gamma:=\Gal(L/F)$. Over $L$ all tori in \eqref{eq:tori-seq} split.
Applying the functor of characters $X^*(-)=\Hom_L(-,\Gm)$ to
\eqref{eq:tori-seq} yields an exact sequence of $\Gamma$-lattices
\begin{equation}\label{eq:Z-to-XS}
0\longrightarrow \Z \xrightarrow{\ \Delta\ } X^*(\wt{T})\longrightarrow X^*(T)\longrightarrow 0.
\end{equation}
Here $\Delta$ is induced by $N$.
One has
$X^*(\Gm^{\ell})\cong \Z^{\ell}$ with trivial $\Lambda$-action, and for each $i$,
\[
X^*(\Res_{K_i/F}\Gm)\ \cong\ \Ind_{\Lambda_i}^{\Lambda}\Z\ \cong\ \Z[\Lambda/\Lambda_i],
\]
where $\Lambda_i:=\Gal(L/K_i)$. Thus
\[
X^*(\wt{T})\ \cong\ \Z^{\ell}\bigoplus\bigoplus_{i=1}^a \Z[\Lambda/\Lambda_i].
\]
Under this identification, 
\eqref{eq:Z-to-XS} is precisely
\eqref{eq:char-seq} in the statement.

By (\ref{eq:Pic-dual-pi0Hhat}), we have 
\[
\Pic(T)\cong \pi_{0}\bigl(\widehat{T}^{\Lambda}\bigr).
\]
By \cite[Section~2.1]{Kott}, there is a canonical isomorphism
\begin{equation}\label{eq:pi0Hhat-H1Xstar}
  \pi_{0}\bigl(\widehat{T}^{\Lambda}\bigr)
    \;\xrightarrow{\ \sim\ }
  H^{1}\bigl(F,X_{\ast}(\widehat{T})\bigr),
\end{equation}
By the identification $X^{\ast}(\widehat{T}) \cong X_{\ast}(T)$,
\begin{equation}\label{eq:Pic-H1}
\Pic(T)\ \cong\ H^1(\Lambda,X^*(T)).
\end{equation}
Since $X^*(\wt{T})$ is a direct sum of  induced
$\Gamma$-modules, we have 
$
H^1(\Lambda,X^*(\wt{T}))=0.
$
Hence taking $\Lambda$-cohomology of  \eqref{eq:Z-to-XS} gives:
\[
0\rightarrow H^1(\Lambda,X^*(T))\ \rightarrow\ H^2(\Lambda,\Z)\rightarrow H^2(\Lambda,X^*(\wt{T})).
\]
 Thus
\begin{equation}\label{eq:H1-kernel}
H^1(\Lambda,X^*(T))\ \cong\ \ker\!\Bigl(H^2(\Lambda,\Z)\to H^2(\Lambda,X^*(\wt{T}))\Bigr).
\end{equation}

Now compute $H^2(\Lambda,X^*(\wt{T}))$. Since
$X^*(\wt{T})\cong \Z^{\ell}\bigoplus\bigoplus_{i=1}^a \Z[\Lambda/\Lambda_i]$,
we have
\[
H^2(\Lambda,X^*(\wt{T}))\ \cong\ H^2(\Lambda,\Z)^{\ell}\ \bigoplus\
\bigoplus_{i=1}^a H^2\bigl(\Lambda,\Z[\Lambda/\Lambda_i]\bigr).
\]
By Shapiro's lemma,
\[
H^2\bigl(\Lambda,\Z[\Lambda/\Lambda_i]\bigr)\ \cong\ H^2(\Lambda_i,\Z).
\]
Under these identifications, the map
$H^2(\Lambda,\Z)\to H^2(\Gamma,X^*(\wt{T}))$ induced by $\Delta$ is
\[
c\ \longmapsto\ \Bigl(\underbrace{c,\dots,c}_{\ell\ \text{times}},
\ \Res_{\Lambda_1}(c),\dots,\Res_{\Lambda_a}(c)\Bigr),
\]
where $\Res_{\Lambda_i}(c)$ denotes the restriction to $\Lambda_i$.

Combining \eqref{eq:Pic-H1} and \eqref{eq:H1-kernel} yields the asserted formula
for $\Pic(T)$ in terms of the kernel of
\[
H^2(\Lambda,\Z)\longrightarrow H^2(\Lambda,\Z)^{\ell}\oplus\bigoplus_{i=1}^a H^2(\Lambda_i,\Z).
\]
If $\ell>0$, then the map has a component
$c\mapsto c$ into $H^2(\Lambda,\Z)$,
so its kernel is $0$. Hence $\Pic(T)=0$.

From $0\to \Z\to \Q\to \Q/\Z\to 0$ and $H^j(\Lambda,\Q)=0$ for $j\ge1$,
we get $H^2(\Lambda,\Z)\cong H^1(\Lambda,\Q/\Z)\cong \Hom(\Lambda,\Q/\Z)$.
Under this identification, the restriction map corresponds to restricting
characters to $\Lambda_i$, giving the final description in the statement.
\end{proof}

\begin{cor}\label{cor:kappa}
There is a one-to-one map between the characters $H^1(F,T(\mathbb{A}_{\overline{F}})/T(\overline{F}))^\vee$ and
the pairs $(E,\nu)$ where
\begin{itemize}
	\item $E$ is an intermediate field such that $F \subset E \subset K_i\subseteq L$ for all $0 \leq i \leq a$ with $E/F$ a cyclic  ext		ension.
	\item $\nu: \Gal(E/F) \hookrightarrow \BQ/\BZ$ is a faithful character on $\Gal(E/F)$.
\end{itemize}
\end{cor}
\begin{proof}
By Lemma \ref{lem-duality-picard} and Proposition \ref{prop: Picard-general-torus},
\[
H^1(F,T(\mathbb{A}_{\overline{F}})/T(\overline{F}))^\vee\cong \Pic(T)\cong\
\left\{\chi\in\Hom(\Lambda,\Q/\Z)\Big| \chi|_{\Lambda_i}=0, \forall i\right\}.
\]

In particular,
we can identify an element
$\kappa\in H^1(F,T_\gamma(\mathbb{A}_{\overline{F}})/T_\gamma(\overline{F}))^\vee$ with 
a character 
\[\chi_{\kappa}\in \Hom(\Gal(L/F), \BQ/\BZ), \quad \chi_{\kappa}\mid_{\Gal(L/K_i)} = 0, \quad 1 \leq i \leq a.\] 
Therefore, by the Galois theory,  $\ker(\chi_{\kappa})$ determines a cyclic extension $E$ of $F$ and we write
$\nu$ for the induced (faithful) character on $\Gal(E/F)$.
The fact $\chi_{\kappa}\mid_{\Gal(L/K_i)}=1$ implies that $E\subseteq K_i$ for $1 \leq i \leq a$.
\end{proof}

Let $\kappa \in H^1(F,T(\mathbb{A}_{\overline{F}})/T(\overline{F}))^\vee$. 
By the Tate-Nakayama duality, we view $\kappa$ as an element in $\pi_0\left(\wh{T}^\Gamma\right)$. 

\begin{lem}\label{lem:s}
   Let $s \in \wh{T}^\Gamma$ be a lifting of $\kappa \in \pi_0\left(\wh{T}^\Gamma\right)$ such that  
   the stabilizer $\wh{H} = \wh{G}_s$ of $s$ is maximal. Then such an element $s$ is unique. Moreover, assume $\kappa$ corresponds to $(E,\nu)$ in Corollary 
   \ref{cor:kappa}, say $[E:F] = d$ with $dm = n$. We fix a generator $\theta$ of $\Gal(E/F)$ and denote by $u = \nu(\theta)$ where
   $u$ is a primitive $d$-th root of unity by fixing an embedding $\BQ/\BZ \hookrightarrow \BC^\times$.  Then, we have 
   \[s=\diag(\Id_m, u \Id_m, \ldots, u^{d-1}\Id_m)\in \wh{G} = \PGL_n(\BC).\]
\end{lem}
\begin{proof}

Let $\Sigma_K=\Hom_F(K,\overline F)$.
The restriction map $\Sigma_K\to \Hom_F(E,\overline F)$ partitions $\Sigma_K$ into $d$ fibers of equal size $m=n/d$.
Equivalently, after choosing one fiber $\Sigma$ and using the cyclicity of $\Gal(E/F)=\langle\theta\rangle$, we may write
\[
\Sigma_K \;=\; \bigsqcup_{i=0}^{d-1} \theta^i(\Sigma),
\qquad |\Sigma|=m.
\]
Using the standard identification
\[
\widehat T \;\cong\; (\C^\times)^{\Sigma_K}/\Delta(\C^\times),
\]
define an element $s_0\in \widehat T$ by declaring it to take the constant value $u^i$ on the block $\theta^i(\Sigma)$.
Up to conjugation by Weyl group of $\wh G$,  we have
\[
s_0 \;=\; \diag\bigl(\Id_m,\; u \Id_m,\; u^2 \Id_m,\;\dots,\; u^{d-1}\Id_m\bigr)\in \PGL_n(\C).
\]
By construction, $s_0\in \widehat T^{\,\Gamma}$,
and its image in $\pi_0(\widehat T^{\,\Gamma})$ is exactly $\kappa$
via the Tate--Nakayama identification.

Its stabilizer in $\widehat G=\PGL_n(\C)$ is the block-diagonal Levi
\begin{equation}\label{eqn: Levi-max}
\widehat G_{\,s_0}
\;\cong\;
\bigl(\GL_m(\C)\times\cdots\times \GL_m(\C)\bigr)\big/\Delta(\C^\times).
\end{equation}

Note that under the above identification, the element $\theta\in \Gamma$ acts on $\wh T$ via conjugating by
\[
w_\theta=\begin{pmatrix}
&   &\Id_m           &        &     & \\
&   &            & \ddots &    &    \\
&   &  &       & \Id_m & \\
& \Id_m &              &                &       &
\end{pmatrix}.
\]
In particular we have any $s\in \wh T^{\Gamma}$ whose projection to $\pi_0(T^{\,\Gamma})$ is $\kappa$ has the form
\[
s=\diag(z, zu, zu^2, \ldots, zu^{d-1}),  \qquad z\in (\C^\times)^m,
\]
where we regard $(\C^\times)^m$ as the diagonal torus of  $\GL_m(\C)$ in \eqref{eqn: Levi-max}.
In particular, for any such $s$, the centralizer $\wh G_s$ will be a Levi in $\wh G_{s_0}$.
This shows that uniqueness of $s_0$ such that $\wh G_{s_0}$ is maximal.

\end{proof}

We choose $s \in \wh{T}^\Gamma$ as in Lemma \ref{lem:s} and write $\wh{H} = \wh{G}_s$.
The Galois group $\Gamma$ acts on $\wh{T}$. It then acts on the root system of $\wh{H}$ so that it induces a homomorphism
\[\rho: \Gamma \lra \mathrm{Out}(\wh{H}).\]
In particular, from the character $\kappa$, we  obtain an endoscopic datum $(s,\rho)$ of $G$, or equivalently, an endoscopic 
triple $(H,s,\eta)$ of $G$. There exists a unique
maximal torus $T_H$ of $H$ with an admissible isomorphism $T_H \ra T$. Denote by $\gamma_H \in T_H(F)$ the image of $\gamma$ 
in this admissible isomorphism.

\begin{prop}\label{ell-endo}
	Let $\kappa \in H^1(F,T(\mathbb{A}_{\overline{F}})/T(\overline{F}))^\vee$ correspond to $(E,\nu)$ in Corollary \ref{cor:kappa}.
	Then
	\begin{itemize}
	\item $H = \mathrm{Res}^1_{E/F}(\GL_m)$ with $m = [K:E]$.  In other words, $H$ consists of those
    $x \in \mathrm{Res}_{E/F}(\GL_m)$ such that
    $\RN_{E/F}(\det(x)) = 1$.
	\item We view $K$ as a vector space over $E$ and then identify $\gamma$ as an element in $\GL_m(E)$, then
	      in fact $\gamma \in H(F)$ and this is $\gamma_H$.
	\end{itemize}
\end{prop}
\begin{proof}
   We shall use notations in Lemma \ref{lem:s} and its proof so that
   \[s=\diag(\Id_m, u \Id_m, \ldots, u^{d-1}\Id_m)\in \wh{G}\]
   and its centralizer
   \[\wh{H} =  \GL_m(\BC) \times \cdots \times \GL_m(\BC)/\Delta(\BC^\times)\]
   where 
   \[\Delta: \BC^\times \ra \GL_m(\BC) \times \cdots \times \GL_m(\BC)\]
   is the diagonal embedding.

   Let $R(\widehat{T}, \widehat{G})$ be the set of roots 
of $\widehat{T}$ in $\widehat{G}$ over $\BC$.
Consider the set 
\[
R^\vee(T, H):=\left\{\alpha\in R(\widehat{T}, \widehat{G})\mid \langle s, \alpha\rangle=1\right\}.
\]
Note that since $s$ is fixed by $\Gamma$,  $R^\vee(T, H)$ is stable under $\Gamma$.  In fact, note  that 
the set $R^\vee(T, H)$ can be identified with the set of roots in $\wh{H}$.

Now we observe that the Galois group 
$\Gal(E/F)=\langle \theta\rangle$ permutes 
the set of roots from simple factors. 

Moreover, under the above identification, 
we have 
\[
R^\vee(T, H)=\coprod_{i=0}^{d-1}\theta^i(\Phi) 
\]
where $\Phi$ is the set of roots of the first copy of $\GL_m(\BC)$
with respect to the torus consisting of diagonal matrices.
Note that the action of $\theta$ sends the simple roots to simple roots (with respect to the Borel consisting of upper triangular matrices 
in $\GL_n(\BC)$).  The Weyl group element $w_{\theta}$ acts trivially on 
$R^\vee(T, H)$.

In particular, we have 
\[
\eta: \widehat{H}\rtimes \langle \theta\rangle \rightarrow \wh{G} \times \langle \theta\rangle
\]
where 
\[
\eta(\theta)=(w_{\theta}, \theta).
\]

Given the above, the corresponding endoscopic group for $G$ is
\[
H = \mathrm{Res}^1_{E/F}(\GL_m),
\qquad m = [K:E],
\]
with the standard block-diagonal $L$-embedding 
${}^{L}H \to {}^{L}G$ determined by $\eta $.
This completes the proof.

\end{proof}

Let $\kappa \in H^1(F,T(\BA_{\bar{F}})/T(\bar{F}))^\vee$ correspond to a pair $(E,\nu)$ as in 
Proposition \ref{ell-endo}. We give the following three equivalent characterizations for a real component of $\kappa$.

\begin{prop}\label{prop-arch-kappa}
Let $v$ be a real place of $F$. Then the following conditions are equivalent:
\begin{enumerate}
    \item  the degree $[E:F] = 2d$ is even and $E_v\cong \BC^d $;
    \item  the homomorphism $\Gamma_v \hookrightarrow \Gamma\twoheadrightarrow\Gal(E/F)$ is nontrivial;
    \item the character $\kappa_v\neq 1$ with $\kappa_v$ the composition
    $
    \Gamma_v\rightarrow \Gal(E/F)\stackrel{\kappa}{\hookrightarrow} \BC^\times.
    $
\end{enumerate}
\end{prop}
\begin{proof}
The equivalence between (2) and (3) follows directly from the definition.
The implication from
(1) to (2) follows from the fact that the nontrivial element $\tau\in \Gamma_v$ acts by complex conjugation on the simple factor $\BC$.
From (2) to (1) we note that the non-triviality of $\tau$
in $\Gal(E/F)$ implies that $E_v$ contains at least one simple factor isomorphic to $\BC$.
Since the Galois group $\Gal(E/F)$ acts transitively on the set $\Hom_{F_v}(E_v, \overline{F}_v)$, (1) holds.
\end{proof}

In the following, for a real place $v$, we study $G(F_v)$-conjugacy classes in the stable conjugacy class of a regular semisimple element $\gamma$ in $\wt{G}(F_v)$.

\begin{prop}\label{prop: archi-kappa}
 Let $K_v=F_v[\gamma]$. Then the stable conjugacy class of $\gamma$ contains one or two $G(F_v)$-conjugacy classes. Moreover, it contains two  conjugacy classes if and only if $n=2k$ is even  and $K_v\cong \BC^k$.
\end{prop}

\begin{proof}
Note that in general we have 
$K_v\cong \BR^{r_1}\times \BC^{r_2}$ with $r_1+2r_2=n$. In particular, the centralizer of $\gamma$ in $G(F_v)$ is given by
\[
T=K_v^1 = \{ x \in K_v^\times \Big| \RN_{K_v/F_v} x = 1 \}.\]

Denote by  $\Gamma_v= \{1, \tau\}$
with $\tau$ the complex conjugation. We have 
\[
\widehat{T}=(\BC^{\times})^n/\Delta(\BC^{\times})
\]
where $\Delta: \BC^\times \rightarrow (\BC^{\times})^n$ is the diagonal embedding.

Assume $r_1>0$. We can assume that
$\Gamma_v$ acts trivially on the first copy of $\BC^\times$. Then we have a $\Gamma_v$-equivariant isomorphism 
\[
\widehat{T}\cong (\BC^\times)^{n-1}
\]
by identifying $(\BC^\times)^{n-1}$ with the last $n-1$ copies of $\BC^\times$. This shows that 
\[
H^1(F_v,T)\cong\pi_0\left(\widehat{T}^{\Gamma_v}\right)\cong \pi_0\left(((\BC^\times)^{n-1})^{\Gamma_v}\right)=\{1\}.
\]

Assume $r_1=0$.
We arrange the action of $\Gamma_v$ on 
$\widehat{T}$ which
permutes the $i$-th and the $(i+1)$-th copies of $\BC^\times $ for all odd $i$. In particular, for 
\[
\begin{aligned}
\tau((1, z_2, z_3, \ldots, z_{n})) &=(z_2, 1, z_4, z_3, \cdots, z_{n}, z_{n-1})  \\
&\sim (1, z_2^{-1}, z_4z_2^{-1}, z_3z_2^{-1}, \cdots, z_{n}z_2^{-1}, z_{n-1}z_2^{-1})\in \widehat{T}^{\Gamma_v}.
\end{aligned}
\]
We have 
\[
z_2=z_2^{-1}, z_3=z_4z_2^{-1}, z_4=z_3z_2^{-1}, \ldots .
\]
This implies 
\[
\pi_0\left(\widehat{T}^{\Gamma_v}\right)\simeq \{\pm 1\}.
\]

Finally, it follows from \cite[\S 1.5.1]{NGO10}
the conjugacy classes within a stable conjugacy class are parametrized by 
\[
\ker(H^1(F_v,T)\rightarrow H^1(F_v,G)).
\]
In our case by simply-connectedness, $H^1(F_v,G)=1$ and by the Tate-Nakayama duality
\[
H^1(F_v,T)\cong \pi_0\left(\widehat{T}^{\Gamma_v}\right)\simeq \{\pm 1\}.
\]
\end{proof}

\section{Haar measures}\label{Haar}

It is important for us to normalize the Haar measures.

Let $F$ be a number field. For each place $v$, we normalize the absolute value $|\cdot|_v$ on $F_v$ as follows. If $F_v = \BR$ 
(resp. $F_v = \BC$), then it is (resp. square of) the usual one. If $F_v$ is nonarchimedean, it maps the uniformizer to $q_v^{-1}$ with
$q_v$ the cardinality of the residue field. 

We normalize the Haar measure $dx_v$ on $F_v$ as
following:
\begin{itemize}
  \item if $F_v = \BR$, then $dx_v$ is the usual
  Lebesgue measure;
  \item if $F_v = \BC$, then $dx_v$ is twice of
  the usual Lebesgue measure;
  \item if $v$ is nonarchimedean, then
  $dx_v$ is normalized so that the integral ring
  $\CO_v$ of $F_v$ has volume one.
\end{itemize}
We have $d(ax_v) = |a|_v dx_v$ for any $a \in F_v^\times$.

We denote by $\BA = \BA_F$ the ring of adeles. Denote by $\BA_f = \prod_{v <\infty}' F_v$  and  $\wh{\CO} = \prod_{v < \infty} \CO_v$. 
By tensor products, we obtain a global absolute value $|\cdot|$ on 
$\BA$ and a Haar measure $dx$ on $\BA$.

Let $G$ be a connected reductive group over a number field $F$. We shall consider two kinds of Haar measures
on $G(\BA)$.

The first Haar measure on $G(\BA)$ we shall consider is the Tamagawa measure. 
Consider the representation $\rho_G$ of $\Gamma_F = \Gal(\bar{F}/F)$ on $X^*(G)_\BQ$, the group of $F$-characters of $G$. Denote by 
\[L(s,\rho_G) = \prod_{v < \infty} L_v(s,\rho_G)\] 
the Artin $L$-function associated to $\rho_G$ and
\[r_G = \lim_{s \ra 1} (s-1)^{\rank X^*(G)} L(s,\rho_G).\]

We fix an invariant  gauge form $\omega_G$ on $G$. In other words, $\omega_G$ is a nowhere-zero and regular differential form
in $\Gamma(G, \bigwedge^{\dim G} \Omega_{G/F})$ 
which is invariant under the translation of $G$ (note that $G$ is unimodular). For each place $v$ of $F$, 
denote by $|\omega_G|_v$ the Haar measure on $G(F_v)$ induced from $\omega_G$ and the Haar measure $dx_v$ on $F_v$. 

The Tamagawa measure $d^\Tam g$  on $G(\BA)$ is defined as follows
\[d^\Tam g = \frac{1}{r_G |\Delta_F|^{\frac{\dim G}{2}}} d^\Tam g_\infty d^\Tam g_\fin, \quad d^\Tam g_\infty = \prod_{v|\infty} |\omega_G|_v, 
\quad  d^\Tam g_\fin = \prod_{v < \infty} L_v(1,\rho_G) |\omega_G|_v.\]
Then $d^\Tam g$ is independent of the choice of the gauge form $\omega_G$.

We shall also consider the Tamagawa measure on a homogeneous space. Let $H$ be a reductive subgroup of $G$ and $X = H \bs G$. Let
$\omega_G$ be an invariant gauge form on $G$, $\omega_H$ be an invariant gauge form on $H$ and $\omega_X$ be a $G$-invariant gauge
form on $X$. We assume that the three forms $\omega_G$, $\omega_H$ and $\omega_X$ match together algebraically in the sense
of \cite[Section 2.4]{Wei82}. In particular, for any place $v$,
\[\int_{G(F_v)} f(g)|\omega_G|_v = \int_{H(F_v) \bs G(F_v)}\int_{H(F_v)} f(hg)|\omega_H|_v |\omega_X|_v, \quad f \in C_c(G(F_v)).\]
Here, we identify $H(F_v) \bs G(F_v)$ as an open subset of $X(F_v)$.

The Tamagawa measure $d^\Tam \mu$  on $X(\BA)$ is defined as follows
\[d^\Tam \mu = \frac{r_H |\Delta_F|^{\frac{\dim H}{2}}}{r_G |\Delta_F|^{\frac{\dim G}{2}}} 
d^\Tam \mu_\infty d^\Tam \mu_\fin, \quad  d^\Tam \mu_\infty = \prod_{v|\infty} |\omega_X|_v, 
\quad  d^\Tam \mu_\fin = \prod_{v < \infty} \frac{L_v(1,\rho_G)}{L_v(1,\rho_H)} |\omega_X|_v.\]

In this paper, we shall also consider another Haar measure on $G(\BA)$ which is convenient for
the study of local orbital integrals. 

For each nonarchimedean place $v$, we fix a maximal open compact subgroup 
$U_v$ of $G(F_v)$. Consider the following measure on $G(\BA)$
\[d^0 g = \prod_v d^0 g_v.\]
Here, 
\begin{itemize}
\item if $v|\infty$, $d^0g_v$ is the local Haar measure on $G(F_v)$ given in \cite[Section 11]{Gr97}. It depends only on $G_v$ and the Haar measure $dx_v$ on $F_v$.
\item if $v <\infty$, $d^0 g_v$ is normalized so that the volume of $U_v$ is $1$.
\end{itemize}

If $X = H \bs G$ is an affine homogeneous space with $X(\BA) = H(\BA) \bs G(\BA)$, 
then we denote by $d^0\mu$ the quotient measure on $X(\BA)$ given by $d^0g$ on $G(\BA)$ and $d^0 h$ on $H(\BA)$.

Now, we compare the Tamagawa measure $d^\Tam g$ and another measure $d^0 g$ for a special case. 
Let $\gamma \in \GL_n(F)$ be regular semisimple. Then the stabilizer of $\gamma$ in $\GL_n$ is
isomorphic to $K^\times$ where 
\[K = F[\gamma] = \prod_{i=1}^r K_i\] 
is a separable extension of $F$ and $K_i/F$ are extensions of number fields.

Denote by $d^\Tam \mu, d^0\mu$ the quotient measures on $(K^\times \bs \GL_n)(\BA) = K_\BA^\times \bs \GL_n(\BA)$
given by the above two measures on $\GL_n(\BA)$ and $K_\BA^\times$ respectively. Note that in this special case, for each nonarchimedean $v$, 
the maximal open compact subgroups of $\GL_n(F_v)$ and $K_v^\times$ are unique up to conjugation.
In particular, the definition of $d^0\mu$ is independent of the choice of maximal open compact subgroups. Denote by
\[C^\Tam_0 = \frac{d^\Tam \mu}{d^0\mu}.\]

\begin{prop}\label{meas-comp}
	We have 
	\[C^\Tam_0 = \frac{1}{|\Delta_F|^{n(n-1)/2}} \cdot \frac{\lim_{s\ra 1} (s-1)^r\zeta_K(s)}{\Res_{s=1}\zeta_F(s) \prod_{i=2}^n \zeta_F(i)}
	\cdot \frac{1}{\sqrt{|\Delta_{K/F}|_\fin}}.\]
	Here, $\Delta_{K/F}$ is the relative discriminant of $K/F$ and we denote by $|\Delta_{K/F}|_\fin = \prod_{v < \infty}
	|\Delta_{K/F}|_v$.
\end{prop}

The proof of this proposition is  given in \cite[Appendix A]{Lee}. As there is the restriction that $K$ is totally real, we repeat the proof here for the convenience of readers.

To compare the two Haar measures $d^\Tam \mu$ and $d^0 \mu$, we need to consider
a third Haar measure. It is the one  constructed by Gross \cite{Gr97} and Gross-Gan \cite{GG99}. Let $G$ be a connected reductive group defined over $F$. Denote by
\[d g = \prod_v d g_v\]
where $d g_v$ is the Haar measure given in \cite[Section 5]{GG99} for $v$ nonarchimedean and \cite[Section 11]{Gr97} 
for $v$ archimedean respectively. In particular, for $v|\infty$, $d g_v = d^0 g_v$. Note that for each $v$, $dg_v$ depends only on $G_v$ and the Haar measure $dx_v$ on $F_v$.
In particular, it is independent of any choice of invariant gauge form $\omega_G$ on $G$. However, we can compare $dg$ with $\prod_v |\omega_G|_v$ as follows.

\begin{prop}[Proposition 9.3 in \cite{GG99}]\label{prop-GG} For almost all $v$, 
	$d g_v = |\omega_G|_v$. Moreover, 
	\[\prod_v \frac{d g_v}{|\omega_G|_v} = f(M_G)^{1/2}\]
	where $f(M_G)$ is the global conductor of the motive $M_G$ of $G$.
\end{prop}

\begin{proof}[Proof of Proposition \ref{meas-comp}]

   By the definition of the Tamagawa measures
   \[C^\Tam_0 = \frac{\lim_{s \ra 1} (s-1)^r\zeta_K(s)}
   {\Res_{s=1} \zeta_F(s)} \cdot |\Delta_F|^{-\frac{n(n-1)}{2}} \prod_v \frac{d^\Tam \mu_v}{d^0 \mu_v}.\]
   For each $v$, we have
   \[\frac{d^\Tam \mu_v}{d^0 \mu_v} = \frac{d^\Tam \mu_v}{d \mu_v} \cdot \frac{d \mu_v}{d^0 \mu_v}.\]
   Here, for each $v$, $d\mu_v$ is the quotient measure
   on $K_v^\times \bs \GL_n(F_v)$ given by the local measure
   $dg_v$ on $\GL_n(F_v)$ and $dt_v$ on $K_v^\times$ ($dg_v$
   and $dt_v$ both are the third Haar measures).

   If $v < \infty$, then
   \[\frac{d^\Tam \mu_v}{d \mu_v} = \frac{\zeta_v(1)}{\zeta_{K_v}(1)} \cdot \frac{|\omega_{\GL_n}|_v/|\omega_{K^\times}|_v}{dg_v/dt_v}, \quad 
   \frac{d\mu_v}{d^0 \mu_v} = \frac{\zeta_{K_v}(1)}{\prod_{k=1}^n \zeta_v(k)}.\]
   Therefore, for $v < \infty$, 
   \[\frac{d^\Tam \mu_v}{d^0 \mu_v} = \frac{1}{\prod_{k=2}^n \zeta_v(k)} \frac{|\omega_{\GL_n}|_v/|\omega_{K^\times}|_v}{dg_v/dt_v}.\]
   On the other hand, if $v|\infty$, then
   \[\frac{d^\Tam \mu_v}{d^0 \mu_v} =  \frac{|\omega_{\GL_n}|_v/|\omega_{K^\times}|_v}{dg_v/dt_v}.\]
   Therefore, by Proposition \ref{prop-GG}
   \[
   \begin{aligned}
   \prod_v \frac{d^\Tam \mu_v}{d^0 \mu_v} &= 
   \frac{1}{\prod_{k=2}^n \zeta_F(k)} \cdot \prod_v \frac{|\omega_{\GL_n}|_v}{dg_v} \prod_v \frac{dt_v}{|\omega_{K^\times}|_v}\\
   &= \frac{1}{\prod_{k=2}^n \zeta_F(k)} \cdot \frac{f(M_{K^\times})^{1/2}}{f(M_{\GL_n})^{1/2}} \\
   &= \frac{1}{\prod_{k=2}^n \zeta_F(k)} \cdot \frac{1}{\sqrt{|\Delta_{K/F}|_\fin}}.
   \end{aligned}\]
   We are done.

\end{proof}

Finally, we focus on the volume of maximal compact subgroups of $\GL_n(F_v)$ with $v|\infty$. Denote by 
$U_\infty = \prod_{v|\infty} U_v$ with $U_v$ a maximal compact subgroup of 
$\GL_n(F_v)$. If $v$ is real, then $U_v = \RO_n(\BR)$ and if $v$ is complex, then $U_v = \RU(n)$ is the compact unitary group of rank $n$. 
The Haar measure on $U_v$ is normalized as follows.

Consider the case $F_v = \BR$ and we shall write simply $U_v$ by $U$. Assume
\[\chi(x) = (x-\lambda_1)(x-\ov{\lambda_1}) \cdots (x - \lambda_{r_2})(x-\ov{\lambda_{r_2}})(x-\mu_1)\cdots (x-\mu_{r_1}) \in \BR[x]\]
with $\lambda_k = x_k + iy_k$, $1 \leq k \leq r_2$ (and their complex conjugation) the complex roots and $\mu_j$, $1 \leq j \leq r_1$
the real roots and $r_1 + 2r_2 = n$. Up to the conjugation by $\GL_n(\BR)$, we may assume
\[\gamma = \diag\left[ 
\begin{pmatrix}
	x_1 & y_1 \\
	-y_1 & x_1
\end{pmatrix}, \cdots, 
\begin{pmatrix}
	x_{r_2} & y_{r_2} \\
	-y_{r_2} & x_{r_2}
\end{pmatrix}, 
\mu_1,\ldots,\mu_{r_1}
\right].\]
Then the stabilizer $S$ of $\gamma$ in $\GL_n(\BR)$ consists of
\[\diag\left[ t_1k(\theta_1),\ldots,t_{r_2}k(\theta_{r_2}),t_{1+r_2},\ldots,t_{r_1+r_2} \right], \quad t_i \in \BR^\times, \quad k(\theta) = 
\begin{pmatrix}
	\cos(\theta) &\sin(\theta) \\
	-\sin(\theta) &\cos(\theta)
\end{pmatrix}.\]
Then we have the following decomposition
\[\GL_n(\BR) = S A_0 N U.\]
Here, $A_0$ consists of
\[a = \diag\left[ a_1,a_1^{-1},\ldots,a_{r_2},a_{r_2}^{-1},1,\ldots,1 \right], \quad a_i > 0.\]
The unipotent group $N$ consists of $n = (n_{ij}) \in \GL_n(\BR)$ with
\[n_{ij} \in 
\begin{cases}
M_{2 \times 2}(\BR), \quad \text{if } i \leq r_2, j \leq r_2, \\
M_{1 \times 2}(\BR), \quad \text{if } i > r_2, j \leq r_2, \\
M_{2 \times 1}(\BR), \quad \text{if } i \leq r_2, j > r_2, \\
M_{1 \times 1}(\BR), \quad \text{if } i > r_2, j > r_2, \\
\end{cases}
\]
$n_{ij} = 0$ if $i < j$ and $n_{ii}$ is the identity matrix
for any $i$.

The Haar measure
on $U$ is normalized via the above decomposition as follows (see \cite[Equation (28)]{EMS})
\[\int_{S \bs \GL_n(\BR)} f(g)dg = \int_{A_0} \int_N \int_U f(ank)\left(\prod_{i=1}^{r_2} a_i\right)da_1\cdots da_{r_2} dn dk, \quad 
f \in C_c(S \bs \GL_n(\BR)).\]
Here, the measure on $S \bs \GL_n(\BR)$ is the above quotient measure $d^0\mu$ (the one before Proposition \ref{meas-comp}), $da_i$ are the Haar measure on $\BR$ and 
the isomorphism $N \cong \BR^{n(n-1)/2-r_2}$ gives the Haar measure on $N$. 

For the case $F_v = \BC$, our normalization of the measure on $U_v$ is exactly the same as the real case with $r_2=0$. We omit the details.

If $v$ is real, denote by  $\Gamma_v(s) = \pi^{-s/2}\Gamma(s/2)$ and if $v$ is complex, denote by $\Gamma_v(s) = 2(2\pi)^{-s}\Gamma(s)$.

\begin{prop}\label{prop-vol-arch}
	For any archimedean place $v$, we have 
	\[\Vol(U_v) = \Gamma_v(1)^n \prod_{i=1}^n \frac{1}{\Gamma_v(i)}.\]
\end{prop}
\begin{proof}
	When $v$ is real, we refer to \cite{Sha}  for the proof.

	In the following, we shall focus on the case $v$ is complex. The proof is classical and is reduced to the study of the so-called
	integrals of Gindikin-Karpelevich. We shall follow \cite[Section 12]{Vos}.

	We denote by $G = \GL_n(\BC)$, $A$ the diagonal torus of $G$ and $N$ (resp. $\bar{N}$) the subgroup of unipotent upper (resp. lower) 
	triangular matrices in $G$, $U = \RU(n)$ the maximal compact subgroup of $G$. 
	Denote by $\rho$ the half sum of positive roots of $A$ in $G$. Let $f \in C_c(G)$ be an arbitrary function on $G$ with compact
	support such that  $f(kg) = f(g)$ for any $k \in U$ and $g \in G$. We have
	\[\int_G f(g)dg = \int_{\bar{N}} \int_A \int_N f(\bar{n}an)e^{2\rho \log a} d\bar{n}dadn\]
	where
	\[e^{2\rho \log a} = |a_1^{n-1} \cdots a_n^{1-n}|, \quad a = \diag\left[ a_1,\ldots,a_n \right].\]
	For each $\bar{n} \in \ov{N}$, write $\bar{n} = k(\bar{n}) a(\bar{n})n(\bar{n})$ with $k(\bar{n}) \in U$, $a(\bar{n}) \in A$
	and $n(\bar{n}) \in N$. we have
	\[
	\begin{aligned}
	\int_G f(g)dg &= \int_{\ov{N}} \int_A \int_N f(an)e^{-2\rho \log a(\bar{n})} e^{2\rho \log a} d\bar{n} da dn \\
	&= \Vol(U)^{-1} \int_{\ov{N}} e^{-2\rho \log a(\bar{n})}d\bar{n} \cdot \int_U \int_A \int_N f(kan)e^{2\rho \log a} dkda dn \\
	&= \Vol(U)^{-1}\int_{\ov{N}} e^{-2\rho \log a(\bar{n})}d\bar{n} \cdot \int_G f(g)dg.
	\end{aligned}\]
	In particular, the volume of $U$ equals  the following integral of Gindikin-Karpelevich
	\[\Vol(U) = \int_{\ov{N}} e^{-2\rho \log a(\bar{n})}d\bar{n}.\]

	In the following, we shall use notations in \cite[Section 12]{Vos}. 
	Let $\alpha$ be a root of $A$ in $\ov{N}$. Consider the canonical morphism $\SL_2(\BC) \ra G$ which maps
	the group $\matrixx{1}{0}{z}{1}$, $z \in \BC$ onto $\ov{N}_\alpha$. For $\bar{n} = \matrixx{1}{0}{z}{1}$, 
	if we write
	\[\bar{n} = \matrixx{1}{n'}{0}{1}\matrixx{a}{0}{0}{b} k_2, \quad k_2 \in \RU(2),\]
	we have
	\[|ab| = 1, \quad |a| = (|z|+1)^{-1}.\]
	For any character $\lambda$ on $A$, denote by $\Phi_\lambda$ the function on $G$ characterized by
	\[\Phi_\lambda(nak) = a^{\lambda+\rho}, \quad n \in N, a \in A, k \in U.\]
	Then
	\[\int_{\ov{N}_\alpha} \Phi_\lambda(\bar{n})d\bar{n} = \int_\BC (|z|+1)^{-(\lambda(H_\alpha) + 1)}dz
	= \frac{\Gamma_\BC(\lambda(H_\alpha))}{\Gamma_\BC(\lambda(H_\alpha)+1)}.\]
	In the above, note that our measure $dz$ is twice of the usual Lebesgue measure. Following the argument in the proof of 
	\cite[Theorem 12.5]{Vos}, we obtain
	\[\int_{\ov{N}} e^{-2\rho \log a(\bar{n})}d\bar{n} = \prod_{\alpha > 0} \frac{\Gamma_\BC(\rho(H_\alpha))}{\Gamma_\BC(\rho(H_\alpha)+1)}.\]

	Finally, note that
	\[\prod_{\alpha > 0} \frac{\Gamma_\BC(\rho(H_\alpha))}{\Gamma_\BC(\rho(H_\alpha)+1)} = \prod_{k=1}^{n-1}
	\left(\frac{\Gamma_\BC(k)}{\Gamma_\BC(k+1)}\right)^{n-k}.\]
	By induction on $n$, we prove that
	\[\prod_{k=1}^{n-1}
	\left(\frac{\Gamma_\BC(k)}{\Gamma_\BC(k+1)}\right)^{n-k} = \Gamma_\BC(1)^n \prod_{k=1}^n \frac{1}{\Gamma_\BC(k)}.\]

\end{proof}

\section{A counting formula}\label{sec: counting-formula}

In this section, we generalize  \cite[Theorem 4.3]{WX} for elliptic orbits to arbitrary regular semisimple orbits.

Let $G$ be a connected reductive group over $F$. 
Let $X$ be a separated scheme of finite type over $\CO$ such that its generic fiber
is the homogeneous space $H \bs G$. Here, $H$ is the stabilizer of a fixed point $P \in X(F)$. 
As we shall count integral points on $X$, we assume that $X(\CO) \not= \emptyset$.

We shall focus on the case that $X$ is affine, or equivalently, $H$ is reductive. We fix an embedding 
\[X\hookrightarrow \Spec(F[x_1,\ldots, x_n]).\] 
For each $x=(x_1,\ldots, x_n)\in X(F_\infty)$, one defines
\[ ||x|| =\max_{v|\infty} ||x||_v, \quad ||x||_v = \left(\sum_{1\leq i\leq n} |x_i|_v^{2/e_v}\right)^{1/2}.\]
For $\CT>0$, one sets
\[
N(X, \CT)=\#\{x\in X(\CO): ||x|| \leq \CT\}, \quad  X(F_{\infty},\CT)=\{x\in X(F_\infty): ||x|| \leq \CT\}.\]

In the following, we shall focus on the case: $G$ is semisimple and simply connected and $H = T$ is a maximal torus in $G$. Moreover,
we assume that for any simple factor $G'$ of $G$, $G'(F_\infty)$ is not compact. 

For a rational point $x \in X(F)$, denote by $T_x$ the stabilizer of $x$ in $G$. 
We fix an invariant gauge form $\omega_G$ on $G$, an invariant gauge form $\omega_{T_x}$ on $T_x$ and a $G$-invariant gauge form
on $X$ match together algebraically (see Section \ref{Haar}). Denote by $\mu_{G,\infty} = |\omega_G|_\infty$, $\mu_{T_x,\infty} = 
|\omega_T|_\infty$ and $\mu_{X,\infty} = |\omega_X|_\infty$ the induced Tamagawa measures on $G(F_\infty)$, $T_x(F_\infty)$ and $X(F_\infty)$.
Similarly,  denote by $\mu_{G,f}$, $\mu_{T_x,f}$ and 
$\mu_{X,f}$ the induced Tamagawa measures on $G(\BA_f)$, $T_x(\BA_f)$ and $X(\BA_f)$. Write $\mu_G = \mu_{G,\infty} \otimes \mu_{G,f}$,
$\mu_{T_x} = \mu_{T_x,\infty} \otimes \mu_{T_x,f}$ and $\mu_X = \mu_{X,\infty} \otimes \mu_{X,f}$ the Haar measures on $G(\BA)$,
$T_x(\BA)$ and $X(\BA)$.

Denote by $r$ the split rank of $T_x$. It is independent of the choice of $F$-rational point $x$.
Choose a basis of characters $\chi_1,\dots,\chi_r\in X^\ast(T_x)$ and consider the map
\[
\log:T_x(\BA) \lra \BR^{r},\qquad t\mapsto \bigl(\log|\chi_1(t)|_\BA,\dots,\log|\chi_{r}(t)|_\BA \bigr).
\]
Denote by $T_x(\BA)^1$ the kernel of the above $\log$ map. We choose the Haar measure $\mu^1_{T_x}$ on $T_x(\BA)^1$ such that
the quotient of $\mu_{T_x}$ by $\mu^1_{T_x}$ is the Lebesgue measure on $\BR^r$ along the above $\log$ map.

Denote by $T_x(F_\infty)^1 = T_x(\BA)^1 \cap T_x(F_\infty)$. Similar to the above, we choose the Haar measure 
$\mu^1_{T_x,\infty}$ on $T_x(F_\infty)^1$ such that
the quotient of $\mu_{T_x,\infty}$ by $\mu^1_{T_x,\infty}$ is the Lebesgue measure on $\BR^r$ along the above $\log$ map.

We assume that the following equi-distribution property holds for $X = T \bs G$:

\begin{assumption}\label{ass: equi-property}
There is a constant $v_T > 0$ such that for any arithmetic subgroup $\Gamma \subset G(F_\infty)$ and any $x \in X(F)$,
\[
\#\{y\in x\cdot\Gamma: ||y|| \leq \CT\}\sim v_T \frac{ \mu^1_{T_x,\infty}((T_x(F)\cap \Gamma)\bs T_x(F_\infty)^1)}
{\mu_{G,\infty}(\Gamma\bs G(F_\infty))}
\mu_{X, \infty}(x\cdot G(F_\infty)\cap X(F_{\infty},\CT))(\log \CT)^r. 
\]
Here, $T_x$ is the stabilizer of $x$ in $G$. We also simply write $\mu_{G,\infty}(\Gamma \bs G(F_\infty))$ for the volume
of $\Gamma \bs G(F_\infty)$ with respect to the measure $\mu_{G,\infty}$, which is always finite when $G$ is semisimple (cf. \cite[Corollary 5]{SC07}). The same convention will be adopted for other cases.
\end{assumption}

When $T$ is elliptic, the equi-distribution property holds by \cite{EMS}. We shall prove that this holds for $G = \SL_n$
in Theorem \ref{thm: equi-distribution} as a refinement of a result of Zhang.

The goal of this section is to prove the following result.

\begin{thm}\label{thm-wx-general}
We have
\[N(X,\CT) \sim  v_T \#\Pic(T) \int_{X(\BA)^{\Br(X)}} f_{X}^\CT(x) dx (\log \CT)^r.\]
Here,
\begin{itemize}
	\item $\Br(X) = H^2_\et(X,\BG_m)$ is the Brauer group of $X$.
	\item Denote by
	     \[X(\BA)^{\Br(X)} = \bigcap_{\xi \in \Br(X)} X(\BA)^\xi, \quad X(\BA)^\xi = \left\{ x \in X(\BA) \Big| \sum_v \xi(x_v) = 0 \right\}.\]
	    Here, for each $\xi \in \Br(X)$ and $x_v \in X(F_v)$, the Brauer evaluation $\xi(x_v)$  is the image of $\xi$ under 
	    the following morphism
	    \[\Br(X) \stackrel{\Br(x_v)}{\lra} \Br(F_v) \stackrel{\inv_v}{\lra} \BQ/\BZ\]
         where $\inv_v: \Br(F_v)\rightarrow \BQ/\BZ$ is defined by \cite[Theorem 1.5.34]{Poo08}.
	\item The function $f_X^\CT$ on $X(\BA)$ is the characteristic function on $X(F_\infty,\CT) \times X(\wh{\CO})$.
	\item The measure $dx$ is given by the Tamagawa measure on $X(\BA)$ (see Section \ref{Haar}).
\end{itemize}
\end{thm}

We now give the proof of Theorem \ref{thm-wx-general}. It is 
similar to the elliptic case.

Let $S$ be a finite set of places of $F$ containing the archimedean places. Denote by $\CO_S$ the $S$-integers of $\CO$, that is,
those elements in $F$ which are integral outside $S$. We assume that $S$ is large enough such that 
\begin{itemize}
	\item the groups $G,T$ admit integral models over $\CO_S$ (that is, a group scheme of finite type over $\CO_S$ whose generic
		fibre is the given group over $F$) such that
		\[X \otimes_\CO \CO_S \cong T \bs G.\]
	\item the fixed point $P \in X(\CO_S)$. 
\end{itemize}
For any finite $v \in S$, we also fix an integral model of $G_v$ over $\CO_v$. 

Similar to the elliptic case (see \cite[Section 2]{WX}), we introduce the following  subgroup of $G(\BA)$. 
For any place $v$, 
\[\St_v = 
\begin{cases}
    G(\CO_v), \quad &v \not\in S; \\
    \left\{ g \in G(\CO_v): X(\CO_v) = X(\CO_v) \cdot g \right\}, \quad &v \in S \text{ and } v \nmid \infty; \\
    G(F_v), \quad &v | \infty.
\end{cases}\]
Denote by 
\[\St = \prod_v \St_v.\]
Then $\St$ is an open subgroup of $G(\BA)$. By the strong approximation property for $G$ (see \cite[Theorem 7.12]{PR98}), 
we have $G(\BA) = G(F)\St$.

For any $x,y \in X(\CO)$, consider the following equivalence relation
\[ x \sim_{\St} y \text{ if and only if } x = y\cdot u\]
for some $u \in \St$. The set of equivalent classes in $X(\CO)$ is denoted by $X(\CO)/\sim_{\St}$.

\begin{prop}[Proposition 2.9 in \cite{WX}]\label{prop: identify-to-Br-set}
The diagonal map
\[X(\CO)/\sim_{\St} \stackrel{\sim}{\lra} 
X(\wh{\CO})^{\Br(X)}/\St. \]
is bijective.
\end{prop}

Let $x \in X(\CO)$. For any $y,z \in x\St \cap X(\CO)$, consider the following equivalence relation
\[ y \sim_{G(F)} z \text{ if and only if } y = z\cdot \gamma\]
for some $\gamma \in G(F)$. The set of equivalent classes in $x\St \cap X(\CO)$ is denoted by $x\St \cap X(\CO)/\sim_{G(F)}$.

\begin{prop}[Proposition 2.11 in \cite{WX}]\label{prop:Shafarevich}
We have 
\[\#\left(x\St \cap X(\CO)/\sim_{G(F)}\right) = \#\Sha(T).\]
Here, the Shafarevich group $\Sha(T)$ of $T$ is defined as
\[\Sha(T) = \ker\left( H^1(F,T) \lra \prod_v H^1(F_v,T) \right).\]
\end{prop}

Denote by
\[\Gamma = \St \cap G(F).\]
Then $\Gamma$ is an arithmetic subgroup of $G(F_\infty)$.

For any $x \in X(\CO)$ and $y \in x\St \cap X(\CO)$. Consider the action of $\Gamma$ on 
\[x\St \cap yG(F) = y\St \cap yG(F).\]

\begin{prop}[Proposition 2.12 in \cite{WX}]\label{lem:decompose-single-rational-orbit}
	There is a bijection
	\[(x\St \cap yG(F))/\Gamma = T_y(F) \bs T_y(\BA)/(T_y(\BA)\cap \St)\]
	where $T_y$ is the stabilizer of $y$ in $G$.
\end{prop}

By Proposition \ref{prop: identify-to-Br-set}, for the equivalent classes of $X(\CO)$ under the equivalence relation $\sim_{\St}$,
we can choose representatives $x^{(a)} \in X(\CO)$ with the index $a \in X(\wh{\CO})^{\Br(X)}/\St$ so that
\[X(\CO) = \bigsqcup_{a} x^{(a)}\St \cap X(\CO).\]

For each $x^{(a)}$, by Proposition \ref{prop:Shafarevich}, for the equivalent classes of $x^{(a)}\St \cap X(\CO)$ under the 
equivalence relation $\sim_{G(F)}$, we can choose representatives $x_j^{(a)} \in x^{(a)}\St \cap X(\CO)$ with the index $j \in \Sha(T)$ 
so that
\[x^{(a)}\St \cap X(\CO) = \bigsqcup_j x^{(a)}\St \cap x_j^{(a)}G(F).\]

For each $x^{(a)}$ and $x_j^{(a)}$, by Proposition \ref{lem:decompose-single-rational-orbit}, for the equivalent classes of
$x^{(a)}\St \cap x_j^{(a)}G(F)$ under the action of $\Gamma$, we can choose representatives $x_{j,k}^{(a)} \in x^{(a)}\St \cap x_j^{(a)}G(F)$
with the index
\[k \in T_{x_j^{(a)}}(\BA)/T_{x_j^{(a)}}(F)\prod_v \left(T_{x_j^{(a)}}(F_v)\cap \St_v\right).\]
Here, $T_{x_j^{(a)}}$ is the stabilizer of $x_j^{(a)}$ in $G$.

We obtain
\[X(\CO) = \bigsqcup_{a,j,k} x_{j,k}^{(a)} \cdot \Gamma\]
so that for each $\CT > 0$
\[N(X,\CT) = \sum_{a,j,k} N(x_{j,k}^{(a)},\CT), \quad N(y,\CT) = \#\left\{ z \in y\cdot \Gamma: ||z|| \leq \CT \right\}.\]

Note that for each $\CT$, as $N(X,\CT)$ is finite, the summation over $a,j,k$ is in fact  finite.

Applying Assumption \ref{ass: equi-property} for each $x_{j,k}^{(a)}$
and $\Gamma$, we obtain
\[N(X,\CT) \sim \frac{v_T(\log \CT)^r}{\mu_{G,\infty}(\Gamma \bs G(F_\infty))} \sum_{a,j,k} 
\mu^1_{T_{x_{j,k}^{(a)}},\infty}\left( (T_{x_{j,k}^{(a)}}(F)\cap \Gamma)\bs T_{x_{j,k}^{(a)}}(F_\infty)^1 \right)
\mu_{X,\infty}\left( x^{(a)}\cdot G(F_\infty) \cap X(F_\infty,\CT) \right).\]
In the above, we note that
\[x_{j,k}^{(a)}\cdot G(F_\infty) \cap X(F_\infty,\CT) = x^{(a)} \cdot G(F_\infty) \cap X(F_\infty,\CT).\]

Let $x \in X(\CO)$. For each $v$, denote by $\St_{x,v} = \St_v \cap T_x(F_v)$ and $\St_x = \prod_v \St_{x,v}$. Denote by 
\[r_x = \#\left(T_{x}(\BA)/T_{x}(F)\St_x \right).\]

\begin{lem}\label{lem:vol-T-1}
	We have
	\[\mu_{T_x,\infty}^1\left( (T_x(F) \cap \Gamma) \bs T_x(F_\infty)^1 \right) = \frac{r_T |\Delta_F|^{\dim T/2}\#\Pic(T)}{r_x\mu_{T_x,f}(\St_{x,f})\#\Sha(T)}.\]
	Here, $r_T = \lim_{s \ra 1} (s-1)^{\rank X^*(T)} L(s,\rho_T)$ is the normalization factor of the global Tamagawa
	measure on $T(\BA)$ (see Section \ref{Haar}).
\end{lem}
\begin{proof}
	Denote by $\St_{x,f} = \prod_{v < \infty} \St_{x,v}$. We have the exact sequence
	\[1 \lra T_x(F)\bs T_x(F)T_x(F_\infty)^1\St_{x,f} \lra T_x(F) \bs T_x(\BA)^1 \lra T_x(F)T_x(F_\infty)^1\St_{x,f} \bs T_x(\BA)^1 \lra 1.\]
	Note that
	\[T_x(F)T_x(F_\infty)^1\St_{x,f} \bs T_x(\BA)^1 = T_x(F)T_x(F_\infty)\St_{x,f} \bs T_x(\BA)\]
	so that it has the cardinality $r_x$. On the other hand, we have
	\[T_x(F)\bs T_x(F)T_x(F_\infty)^1\St_{x,f}  = (T_x(F) \cap T_x(F_\infty)^1\St_{x,f}) \bs T_x(F_\infty)^1\St_{x,f} 
	= (T_x(F)\cap \Gamma) \bs T_x(F_\infty)^1\St_{x,f}\]
	and
	\[1 \lra \St_{x,f} \lra (T_x(F)\cap \Gamma) \bs T_x(F_\infty)^1\St_{x,f} \lra (T_x(F)\cap \Gamma) \bs T_x(F_\infty)^1 \lra 1.\]
    In the above, the first map sends any element in $\St_{x,f}$
    to $T_x(F_\infty)^1\St_{x,f}$ whose archimedean component is
    equal to $1$.
    
	Therefore,
	\[\mu_{T_x}^1(T_x(F) \bs T_x(\BA)^1) = r_x \mu_{T_x,\infty}^1\left( T_x(F) \cap \Gamma \bs T_x(F_\infty)^1 \right)\mu_{T_x,f}(\St_{x,f}).\]
	Denote by $\tau(T_x)$ the Tamagawa number of $T_x$. By the definition of Tamagawa measures,
	\[\mu_{T_x}^1(T_x(F) \bs T_x(\BA)^1) = r_{T_x}|\Delta_F|^{\dim T_x/2}\tau(T_x).\]
	Note that $T_x$ is isomorphic to $T$. By \cite[(5.1.1)]{Kot84}
	\[\mu_{T_x}^1(T_x(F) \bs T_x(\BA)^1) = r_{T}|\Delta_F|^{\dim T/2}\tau(T).\]
        By a formula of Sansuc  \cite[\S 5.1]{Kot84}, we have 
	\[\tau(T) = \frac{\#\Pic(T)}{\#\Sha(T)}.\]
	The Lemma now holds.
\end{proof}

\begin{lem}
   Let $x \in X(\CO)$. For any $y \in x\St \cap xG(F)$, the volume
   \[\mu_{T_y,\infty}^1\left( T_y \cap \Gamma \bs T_y(F_\infty)^1 \right) = \mu_{T_x,\infty}^1\left( T_x \cap \Gamma \bs T_x(F_\infty)^1 \right).\]
\end{lem}
\begin{proof}

	By Lemma \ref{lem:vol-T-1}, it is enough to show that if $y \in x\St \cap xG(F)$, then 
	$r_y = r_x$ and $\mu_{T_y,f}(\St_{y,f}) = \mu_{T_x,f}(\St_{x,f})$.

	As $y = x\cdot \gamma$ for some $\gamma \in G(F)$, there is an isomorphism
	\[T_x \stackrel{\sim}{\lra} T_y, \quad t \mapsto t \cdot \gamma.\]
	As we assume that gauge forms on $G$, $T_x$ (and also $T_y$) and $X$ match together algebraically, 
	the gauge form $\omega_{T_x}$ on $T_x$ and the gauge form $\omega_{T_y}$ on $T_y$ are compatible with the above isomorphism.  

	We may write $y = x\cdot u = x \cdot \gamma$ with $u \in \St$ and $\gamma \in G(F)$. In particular, $t = \gamma u^{-1} \in T_x(\BA)$. 
	
	We claim that the conjugation of $\gamma$ induces a bijection
	\[T_x(F)\St_x \bs T_x(\BA) \stackrel{\sim}{\lra} T_y(F)\St_y \bs T_y(\BA), \quad t \mapsto t\cdot \gamma.\]
	In fact, it is clear that $T_x \cdot \gamma = T_y$. Moreover
	\[ (T_x(F)\St_x)\cdot \gamma = T_y(F) (\St_x \cdot \gamma) = T_y(F) (\St_x \cdot tu) = T_y(F) (\St_x \cdot u) = T_y(F)\St_y.\]

	In the above, we have already noted that there is a bijection
	\[\St_{x,v} \stackrel{\sim}{\lra} \St_{y,v}, \quad t \mapsto t\cdot \gamma.\]
	Therefore, $\mu_{T_y,f}(\St_{y,f}) = \mu_{T_x,f}(\St_{x,f})$.
	We are done.
\end{proof}

Hence
\[N(X,\CT) \sim  v_T(\log \CT)^r \sum_{a,j} 
v(x_j^{(a)}) \mu_{X,\infty}\left( x_j^{(a)} \cdot G(F_\infty) \cap X(F_\infty,\CT) \right)
\]
where for each $x = x_j^{(a)}$,
\[v(x) = \frac{r_x \mu_{T_x,\infty}^1\left( (T_x(F) \cap \Gamma) \bs T_x(F_\infty)^1 \right)}{\mu_{G,\infty}(\Gamma \bs G(F_\infty))}.\]

In the following, for each $x = x_j^{(a)}$, we shall study $v(x)$.

\begin{lem}\label{lem:vol-G}
	We have
	\[\mu_{G,\infty}(\Gamma \bs G(F_\infty))\mu_{G,f}(\St) = |\Delta_F|^{\dim G/2}.\]
\end{lem}
\begin{proof}
	By the strong approximation property for $G$,
	\[\mu_{G,\infty}(\Gamma \bs G(F_\infty))\mu_{G,f}(\St) = \mu_{G,\BA}(G(F) \bs G(F)\St) = \mu_{G,\BA}(G(F) \bs G(\BA)) = 
	|\Delta_F|^{\dim G/2} \tau(G) = |\Delta_F|^{\dim G/2}.\]
	Here, as $G$ is simply connected, $\tau(G) = 1$. 
\end{proof}

\begin{lem}
	We have
	\[v(x) = \frac{r_T |\Delta_F|^{\dim T/2}}{|\Delta_F|^{\dim G/2}}\frac{\#\Pic(T)}{\#\Sha(T)}\mu_{X,f}(x\cdot \St).\]
\end{lem}
\begin{proof}
	By Lemma \ref{lem:vol-T-1} and Lemma \ref{lem:vol-G}, 
	\[v(x) = \frac{r_T |\Delta_F|^{\dim T/2}}{|\Delta_F|^{\dim G/2}}\frac{\#\Pic(T)}{\#\Sha(T)}\frac{\mu_{G,f}(\St)}{\mu_{T_x,f}(\St_x)}.\]
	Note that for each $v<\infty$, we have
	\[\mu_{X,v}(x\cdot \St_v) = \frac{\mu_{G,v}(\St_v)}{\mu_{T_x,v}(\St_{x,v})}.\]
	The Lemma now holds.
\end{proof}

Therefore, by Proposition \ref{prop:Shafarevich}

\[
\begin{aligned}
N(X,\CT) &\sim  v_T(\log \CT)^r \sum_{a,j} 
v(x_j^{(a)}) \mu_{X,\infty}\left( x_j^{(a)} \cdot G(F_\infty) \cap X(F_\infty,\CT) \right) \\
&= v_T(\log \CT)^r \frac{r_T |\Delta_F|^{\dim T/2}}{|\Delta_F|^{\dim G/2}}\#\Pic(T) \sum_{a} \mu_{X,f}(x^{(a)}\cdot \St)
\mu_{X,\infty}\left( x^{(a)} \cdot G(F_\infty) \cap X(F_\infty,\CT) \right). 
\end{aligned}
\]

Finally, by Proposition \ref{prop: identify-to-Br-set}, we obtain
\[N(X,\CT) \sim v_T (\log \CT)^r \#\Pic(T) \int_{X(\BA)^{\Br}} f_{X}^\CT(x) dx.\]
Here, we note that $\frac{r_T |\Delta_F|^{\dim T/2}}{|\Delta_F|^{\dim G/2}}$ is the normalization factor in the definition
of the global Tamagawa measure on $X(\BA)$. We are done.

\section{An equi-distribution property}\label{sec: equi-distribution-refinement}

Denote by $G = \SL_n$. Let $\chi(x) \in \CO[x]$ be a monic and regular polynomial of degree $n$. Denote by $K = F[x]/(\chi(x))$. As $\chi$ is regular, 
we may write
\[K = K_0 \times \cdots \times K_r\]
where $K_j$, $0 \leq j \leq r$, are field extensions of $F$. Denote by $n_j = [K_j:F]$, $0 \leq j \leq r$. Then
\[\chi(x) = \chi_0(x) \cdots \chi_r(x)\]
with irreducible $\chi_j(x) \in F[x]$, $0 \leq j \leq r$. 

Let $\gamma_0 \in \GL_n(F)$ be a matrix with integral
entries such that the characteristic polynomial of $\gamma_0$ is $\chi$. Denote by $T_{\gamma_0}$ the stabilizer of $\gamma_0$ in $G$. 

The purpose of this section is to prove the following result. 

\begin{thm}\label{thm: equi-distribution}
	The equi-distribution property (Assumption \ref{ass: equi-property}) holds for $X = T_{\gamma_0} \bs G$: 
	there is a constant $v_\chi > 0$ such that for any arithmetic subgroup $\Gamma \subset G(F_\infty)$ and any $\gamma \in X(F)$,
\[
\#\{y\in \gamma\cdot\Gamma: ||y|| \leq \CT\}\sim v_\chi \frac{ \mu^1_{T_\gamma,\infty}((T_\gamma(F)\cap \Gamma)\bs T_\gamma(F_\infty)^1)}
{\mu_{G,\infty}(\Gamma\bs G(F_\infty))}
\mu_{X, \infty}(B(\gamma,\CT))(\log \CT)^r. 
\]
Here, we denote by $B(\gamma,\CT) = \gamma \cdot G(F_\infty) \cap X(F_\infty,\CT)$. In fact, we have
\[v_\chi = [F:\BQ]^rn^{r-1}\prod_{j=0}^r n_j.\]
\end{thm}

This is a refinement of  Theorem \ref{thm: Zhang2021}. To state this theorem, we need to introduce some notations.

Fix $\gamma \in X(F)$ and denote by $T = T_\gamma$ the stabilizer of $\gamma$ in $G$. 

Denote by $V = F^n$. 
For each $0 \leq j \leq r$, denote by $V_j \subset V$ the kernel of $\chi_j(\gamma)$ (under the usual action of $n \times n$ matrices on $V$). 
Then 
\[V = \bigoplus_{j=0}^r V_j, \quad \dim_F V_j = n_j.\]
The subspace $V_j$ is preserved by $T$. For each $0 \leq j \leq r$, denote by
\[\lambda_j(t) = \det(t|_{V_j}), \quad t \in T.\]
Then $\lambda_j$, $0 \leq j \leq r$ generate the group $X^*(T)_\BQ$.

Denote by $H$ the hyperplane in $\BR^{r+1}$ consisting of $y = (y_j)_j \in \BR^{r+1}$ and $\sum_j y_j = 0$. The measure on $H$
is endowed by the projection to the first $r$-th coordinates
\[H \stackrel{\sim}{\lra} \BR^r, \quad y = (y_j)_{0 \leq j \leq r} \mapsto (y_j)_{0 \leq j \leq r-1}.\]
Consider the following surjective map
\begin{equation}\label{eqn: log-map-def}
T(F_\infty) \stackrel{\log}{\lra} H, \quad t \mapsto \left( \log|\lambda_j(t)| \right)_j.
\end{equation}
Denote by $T(F_\infty)^1$ the kernel of this $\log$ map. 
We endow the measures on $T(F_\infty)$ and $T(F_\infty)^1$ such that the quotient measure on $T(F_\infty)^1 \bs T(F_\infty)$ is given
by the above measure on $H$.

For each $0 \leq j \leq r$, we fix a $F$-basis of $V_j$, say $e_1^j, \ldots, e_{n_j}^j$. Denote by $\CJ$ the set of  proper 
nonempty subsets of $\{0,\ldots,r\}$. For each $J \in \CJ$, denote by
\[w_J = \wedge_{j \in J} w^j \in \bigwedge^{\sum_{j \in J} n_j} V, \quad w^j = e_1^j \wedge \cdots \wedge e_{n_j}^j.\]
Then 
\[w_J \cdot t = \prod_{j \in J} \lambda_j(t) w_J, \quad t \in T.\]
In particular, $w_J$ is invariant under the action of $T(F_\infty)^1$. 

For each archimedean place $v$, we consider the following function $||\cdot||_v$ on $V_v$: for $y = \sum_{i,j} y_{ij} e_i^j \in V_v$,
\[||y||_v = \left(\sum_{i,j} |y_{ij}|_v^{2/e_v}\right)^{e_v/2} \]
where $e_v = 1$ if $v$ is real and $e_v=2$ if $v$ is complex. Note that
\begin{itemize}
	\item For any $\alpha \in F_v$ and $y \in V_v$, $||\alpha y||_v = |\alpha|_v ||y||_v$;
	\item Denote by $U_v$ the standard maximal compact subgroup of $G(F_v)$: if $v$ is real, then $U_v = \SO(n)$ and 
		if $v$ is complex, then $U_v = \SU(n)$. Then $||y\cdot k||_v = ||y||_v$ for any $y \in V_v$ and 
		$k \in U_v$.
\end{itemize}
Moreover, we extend the definition of $||\cdot||_v$ to $\wedge^{k}V_v$, $1 \leq k \leq n$
with respect to the basis induced by $\{e_i^j\}_{i,j}$. Denote by $||\cdot|| = \prod_{v|\infty} ||\cdot||_v$.

For each $J \in \CJ$, consider the function
\[d_J(g) = ||w_J \cdot g||, \quad g \in T(F_\infty)^1 \bs G(F_\infty).\]

For any $\varepsilon>0$ and $g\in G(F_\infty)$, consider the  following truncation window given by
\[\Omega_{g,\varepsilon} = \left\{ t\in T(F_\infty)^1 \bs T(F_\infty)\Big|  d_J(tg) \ge \varepsilon ,\ \ \forall  J\in \CJ \right\}.\]

\begin{lem}\label{lem: modified-def-Omega}
   For each $\varepsilon$ and $g$, $\Omega_{g,\varepsilon}$ has finite volume. 
   There is an isomorphism
   \[\Omega_{g,\varepsilon} \stackrel{\log}{\lra} \left\{ y= (y_j)_j \in H \Big|  \sum_{j \in J} y_j 
    \geq \log(\varepsilon) - \log d_J(g), \quad \forall J \in \CJ\right\}.\]
\end{lem}
\begin{proof}
	For each $t  \in \Omega_{g,\varepsilon}$, denote by $y = (y_j)_j \in \BR^{r+1}$ its image under the $\log$ map. 
	Then $y_j = \log|\chi_j(t)|$. For each $J \in \CJ$, we have
	\[d_J(tg) = ||w_J \cdot tg|| = ||\prod_{j \in J} \chi_j(t) w_J g || = \prod_{j \in J} |\chi_j(t)| d_J(g).\]
	This gives the above isomorphism.

	Moreover, note that for each $J$
	\[\sum_{j \in J} y_j = - \sum_{j \in J^c} y_j.\]
	This implies that $\Omega_{g,\varepsilon}$ is bounded so that it has finite volume.
\end{proof}

\begin{lem}\label{lem-invariant-Omega}
	Denote by $\mu(\Omega_{g,\varepsilon})$ the volume of $\Omega_{g,\varepsilon}$ under the above quotient measure on 
	$T(F_\infty)^1 \bs T(F_\infty)$. Then for any $\varepsilon$ and $g$,
	\[\mu(\Omega_{tgk,\varepsilon}) = \mu(\Omega_{g,\varepsilon}), \quad t \in T(F_\infty), k \in \prod_{v|\infty} U_v.\]
\end{lem}
\begin{proof}
	It is invariant under the right multiplication of $U_v$ since the function $||\cdot||_v$ is invariant under $U_v$.
	Let $t \in T(F_\infty)$ and denote by $y^0 = \log(t) \in \BR^{r+1}$ under \eqref{eqn: log-map-def}. Then under the $\log$ map, the map $y \mapsto y+y_0$ gives
	a bijection between $\Omega_{tg,\varepsilon}$ and $\Omega_{g,\varepsilon}$. In particular, they have the same volume.
\end{proof}

In particular, for each $\varepsilon$, the function $g \mapsto \mu(\Omega_{g,\varepsilon})$ descends to a function on $\gamma \cdot G(F_\infty)$.

\begin{thm}[Theorem 1.2 in \cite{Zhang}]\label{thm: Zhang2021}
 Let $\Gamma<G(F)$ be an arithmetic subgroup. For any $\varepsilon>0$ and $\gamma \in X(F)$, we have
\begin{equation}\label{eq:eqred}
\#\{y\in \gamma \cdot\Gamma: \|y\|\leq \CT\}\ \sim
\frac{ \mu^1_{T_\gamma,\infty}((T_\gamma(F)\cap \Gamma)\bs T_\gamma(F_\infty)^1)}{\mu_{G,\infty}(\Gamma\bs G(F_\infty))}
\int_{B(\gamma,\CT)}\mu(\Omega_{g,\varepsilon})d\mu_{X, \infty},
\quad \CT\to\infty.
\end{equation}
\end{thm}

In the following, we shall study the asymptotic behavior
\[\int_{B(\gamma,\CT)}\mu(\Omega_{g,\varepsilon}) d\mu_{X,\infty}, \quad \CT \to\infty.\]
It is enough to consider the case $\gamma = \gamma_0$.

We rewrite the basis $\{e_i^j\}_{i,j}$ as
\[(f_1,\ldots,f_n) = \left( e_1^0,\ldots,e_{n_0}^0,\ldots,e_1^r,\ldots,e_{n_r}^r \right).\]
In particular, for each $J \in \CJ$, we can write  
\[w_J = f_{i_1} \wedge \cdots \wedge f_{i_\ell}\]
for some indices $i_1,\ldots,i_\ell$.  Denote by
\begin{equation}\label{eq:degree}
\deg(J) = \sum_{t=1}^\ell (i_t - t).
\end{equation}

\begin{prop}\label{prop: c_T} For any fixed $\varepsilon > 0$, as $\CT \ra +\infty$, we have
\begin{equation}\label{eq:key24}
\int_{B(\gamma_0,\CT)}\mu(\Omega_{g,\varepsilon}) d\mu_{X,\infty} \sim v_{\chi}(\log \CT)^r \mu_{X,\infty}(B(\gamma_0,\CT)),
\qquad \CT \to\infty.
\end{equation}
where $v_\chi$ is the volume of the following polytope 
\[P_\chi = \left\{ y \in H \Big| \sum_{j \in J} y_j \geq -[F:\BQ]\deg(J), \quad \forall J \in \CJ \right\}.\]
\end{prop}

\begin{remark}
	It is true that the definition of $\deg(J)$ with $J \in \CJ$  depends on the ordering of $(V_j)_{j=0}^r$. 
	However, as the proposition implies, the volume $v_{\chi}$ does not.
\end{remark}

We shall prove Proposition \ref{prop: c_T} in the appendix of this paper.

\begin{lem}\label{lem:deg_v(J)=m_J}
   Let $J \in \CJ$. We have
   \[\deg(J) = m_J:= \sum_{p<q, p \not\in J, q \in J} n_pn_q\]
   so that
   \[P_\chi = \left\{ y \in H \Big| \sum_{j \in J} y_j \geq -[F:\BQ] m_J, \quad \forall J \in \CJ \right\}.\]
\end{lem}
\begin{proof}
   By Proposition \ref{prop: c_T}, it suffices to show that
   \[\sum_{t=1}^\ell (i_t-t) = m_J.\]

   Let $A:=\{i_1,\dots,i_\ell\}\subset\{1,2,\dots,n\}$ be the set of indices of $\{f_1,\ldots,f_n\}$ occurring in 
   $\prod_{j \in J} V_{j}$.  For each $1 \leq t \leq \ell$, since $i_t$ is the $t$-th smallest element of $A$, among the integers 
   $\{1,2,\dots,i_t\}$ there are exactly $t$ elements of $A$. Hence
	\[\#\big(\{1,\dots,i_t\}\setminus A\big)=i_t-t\]
    and

\[\sum_{t=1}^{\ell}(i_t-t) =\sum_{t=1}^{\ell}\#\big(\{1,\dots,i_t\}\setminus A\big) =\#\Big\{(k,i)\;:\; i\in A,\ k\notin A,\ k<i\Big\}.\]

By construction of the ordered basis, an index in the block $V_{p}$ of $V$ is less than
an index in the block $V_{q}$ if and only if $p<q$.  Moreover, $A$ consists
all indices belonging to blocks $V_{q}$ with $q\in J$, while $A^c$
consists all indices belonging to blocks $V_{p}$ with $p\notin J$.
Therefore a pair $(k,i)$ with $k<i$ contributes if and only if $k$ lies in some
block $V_{p}$ with $p\notin J$, $i$ lies in some block $V_{q}$ with $q\in J$, and
$p<q$.

Fix such a pair of blocks $(p,q)$ with $p<q$, $p\notin J$, $q\in J$.  There are
$n_p$ choices for $k$ in $V_{p}$ and $n_q$ choices for $i$ in $V_{q}$. Hence exactly
$n_p n_q$  pairs $(k,i)$ coming from these blocks.  Summing over all
such $(p,q)$ yields
\[
\#\Big\{(k,i)\;:\; i\in A,\ k\notin A,\ k<i\Big\}
=
\sum_{\substack{p<q\\ p\notin J,\ q\in J}} n_p n_q = m_J.
\]
The Lemma now holds.
\end{proof}

\begin{prop}\label{prop:volume-polytope}
    We have
    \[v_\chi = [F:\BQ]^rn^{r-1}\prod_{j=0}^r n_j.\]
\end{prop}

We shall prove Proposition \ref{prop:volume-polytope} in the next section.

Now, Theorem \ref{thm: equi-distribution} holds.

\section{Volumes of polytopes}

In this section, we give a proof of Proposition \ref{prop:volume-polytope}.

Denote by $d_0 = [F:\BQ]$. Let $\chi(x) \in \CO[x]$ be a monic and regular polynomial. 
Denote by $K = F[x]/(\chi(x))$. We write
\[K = K_0 \times \cdots \times K_r\]
where $K_j$, $0 \leq j \leq r$, are field extensions of $F$. Denote by $n_j = [K_j:F]$, $0 \leq j \leq r$. 

Denote by $\CJ$ the set of proper nonempty subsets $J$ of $\{0,\ldots,r\}$. For any
$J \in \CJ$, denote by
\[m_J = \sum_{p<q, p \not\in J, q \in J} n_pn_q.\]
Then
\[m_{J^c} = \sum_{p < q, p \in J, q \not\in J} n_pn_q = 
\sum_{p > q, p \not \in J, q \in J} n_pn_q, \quad
m_J + m_{J^c} = \sum_{p \in J, q \not\in J} n_pn_q.\]

 By Lemma \ref{lem:deg_v(J)=m_J}, we have
\[P_\chi = \left\{ y \in H \Big| \sum_{j \in J} y_j \geq -d_0m_J, \quad \forall J \in \CJ \right\}.\]

\begin{lem}
	The polytope $P_\chi$ is a translate of the centrally symmetric polytope
	\begin{equation}\label{eqn: cetral-symmetric-polytope}
    C_\chi = \left\{y \in H \Bigg| \Big|\sum_{j \in J} y_j\Big|  \leq \frac{d_0}{2}\sum_{p \in J, q \not\in J} n_pn_q , 
	\quad \forall J \in \CJ \right\} .
    \end{equation}
\end{lem}
\begin{proof}
     For $y \in P_\chi$ and any $J \in \CJ$, we have
     \[ \sum_{j \in J} y_j \geq -d_0m_J, \quad -\sum_{j \in J} y_j  = \sum_{j \in J^c} y_j \geq -d_0m_{J^c}.\]
     In particular,
     \[ d_0m_{J^c} \geq \sum_{j \in J} y_j \geq -d_0m_J.\]
     Denote by
     \[b_J = \frac{d_0(m_{J^c} - m_J)}{2}, \quad c_J = \frac{d_0(m_{J^c} + m_J)}{2}.\]
     Then for any $J$,
     \[ \Big| \sum_{j \in J} y_j  - b_J \Big| \leq c_J.\]
     We claim that there is $y_0 \in C_\chi$ such that $\sum_{j \in J} y_{0,j} = b_J$. Therefore,
     \[\Big| \sum_{j \in J} (y-y_0)_j \Big| \leq c_J.\]
     In other words, $P_\chi$ is the translate of $C_\chi$ by $y_0$.

     For the claim, as for each $J$, $b_J + b_{J^c} = 0$, it suffices to prove that 
     $b_{J_1} + b_{J_2} = b_{J_1 \cup J_2}$ for any $J_1 \cap J_2 = \emptyset$.  For this, it is enough to prove that
     $b_{J} + b_{\{j_0\}} = b_{J \cup \{j_0\}}$ for any proper subset $J$ and $j_0 \not\in J$. In fact,
     \[
     \begin{aligned}
     m_{(J \cup \{j_0\})^c} &= \sum_{p > q, p \not\in J \cup \{j_0\}, q \in J \cup \{j_0\}} n_pn_q \\
			&= \sum_{p > q, p \not\in J \cup \{j_0\}, q \in J} n_pn_q + \sum_{p > j_0, p \not\in J} n_pn_{j_0} \\
			&= m_{J^c} + m_{\{j_0\}^c} - \sum_{j_0 > q, q \in J} n_{j_0}n_{q}- \sum_{p > j_0, p \in J} n_{p}n_{j_0} \\
			& = m_{J^c} + m_{\{j_0\}^c} - \sum_{q \in J} n_{j_0}n_{q}.
    \end{aligned}\]
    Similarly, we have
    \[ m_{J \cup \{j_0\}} = m_{J} + m_{\{j_0\}} - \sum_{q \in J} n_{j_0}n_{q}.\]
    Thus
    \[ m_{(J \cup \{j_0\})^c} - m_{J \cup \{j_0\}} = (m_{J^c} - m_J) + (m_{\{j_0\}^c} - m_{\{j_0\}}).\]
    The claim now holds.
\end{proof}

Recall that the Minkowski sum of several subsets $A,\ldots,B$ in the linear space $\BR^{r+1}$ is
the locus of sums of vectors that belong to these subsets 
\[A+ \cdots +B = \left\{ a+\cdots+b \Big| a \in A,\ldots,b \in B \right\} \subset \BR^{r+1}.\]
A zonotope is a Minkowski sum of several line intervals. Let $\Gamma$ be a connected graph with
the vertex set $\left\{ 0,\ldots,r \right\}$. Denote by $E(\Gamma)$  the set of edges for $\Gamma$. For each edge $(p,q) \in E(\Gamma)$, consider a weight $w_{pq} \geq 1$ and write 
$w = (w_{pq})_{(p,q) \in E(\Gamma)}$. For a connected weighted graph $(\Gamma,w)$, 
consider the weighted graphical zonotope
\[Z(\Gamma,w) = \sum_{(p,q) \in E(\Gamma), p<q} [0,w_{pq}(e_p-e_q)] \subset H.\]
Here, $e_0,\ldots, e_r$ are the coordinate vectors in $\BR^{r+1}$.

\begin{lem}\label{lem-submodular}
    Consider the following weighted complete graph $(\Gamma,w)$ with
	\[V(\Gamma) = \{0,\ldots,r\}, \quad E(\Gamma) = \left\{ (p,q), 0 \leq p < q \leq r \right\}, \quad w = (w_{pq}), \quad w_{pq} = d_0n_pn_q.\]
    Then the weighted graphical zonotope $Z(\Gamma,w)$ is a translate of $C_\chi$ \eqref{eqn: cetral-symmetric-polytope}. 
\end{lem}

To prove this lemma, we introduce the following notations. Denote by $\CJ$ the set of proper non-empty subsets of $\{0,\ldots,r\}$. 
For a function $\phi: \CJ \ra \BR_{\geq 0}$, consider its base polytope
\[B(\phi) = \left\{ y \in H \Bigg| \sum_{j \in \CJ} y_j \leq \phi(J), \quad \forall J \in \CJ \right\}
=\left\{ y \in H \Bigg| -\phi(J^c) \leq \sum_{j \in \CJ} y_j \leq \phi(J), \quad \forall J \in \CJ \right\}.\]

We have the following result for base polytopes given by submodular functions.
\begin{prop}[\cite{Fu05}, Chapter II, Section 3.1, (3.33)]\label{prop-submodular}
   Let $\phi_i: \CJ \ra \BR_{\geq 0}$, $i=1,2$ be two functions on $\CJ$ with non-negative values. Assume that they are
   submodular, that is, 
   \[\phi_i(J_1) +\phi_i(J_2) \geq \phi_i(J_1 \cap J_2) + \phi_i(J_1 \cup J_2), \quad J_1, J_2 \in \CJ, \quad i=1,2.\]
   Then
   \[B(\phi_1 + \phi_2) = B(\phi_1) + B(\phi_2).\]
\end{prop}

\begin{proof}[Proof of Lemma \ref{lem-submodular}]
Denote by 
\[Z^0(\Gamma,w) = Z(\Gamma,w) - \frac{1}{2} \sum_{p,q} w_{pq}(e_p-e_q) = \sum_{p,q} \left[ -\frac{w_{pq}}{2}(e_p-e_q),\frac{w_{pq}}{2}(e_p-e_q)
\right].\]
We shall prove that
\[C_\chi = Z^0(\Gamma,w).\]

In fact, consider the function 
\[f: \CJ \lra \BR_{\geq 0}, \quad J \mapsto \frac{1}{2} \sum_{p \in J, q \not\in J} w_{pq}.\]
Then $B(f) = C_\chi$ as $f(J^c) = f(J)$ for any $J \in \CJ$.

On the other hand, for each $0 \leq p < q \leq r$, consider
\[f_{pq}: \CJ \lra \BR_{\geq 0}, \quad J \mapsto 
\begin{cases}
	\frac{w_{pq}}{2}, \quad &\text{if $|\{p,q\} \cap J| = 1$}; \\
	0, \quad &\text{otherwise}.
\end{cases}\]
Then $f_{pq}$ are submodular, that is, for any $J_1,J_2 \in \CJ$
\[f_{pq}(J_1) + f_{pq}(J_2) \geq f_{pq}(J_1 \cap J_2) + f_{pq}(J_1 \cup J_2).\]
In fact, equality holds in the above inequality except $p,q \not\in J_1 \cap J_2$ and $|\{p,q\} \cap J_i| = 1$ for $i=1,2$. Moreover,
we have
\[B(f_{pq}) = \left[ -\frac{w_{pq}}{2}(e_p-e_q),\frac{w_{pq}}{2}(e_p-e_q)\right].\]

Note that $f = \sum_{0 \leq p<q \leq r} f_{pq}$. By Proposition \ref{prop-submodular}, we have
\[C_\chi = B(f) = \sum_{0 \leq p < q \leq r} B(f_{pq}) = Z^0(\Gamma,w).\]

\end{proof}

\begin{lem}
   Let $(\Gamma,w)$ be a connected weighted graph. We have
\[\Vol(Z(\Gamma,w)) = \sum_T \prod_{(p,q) \in E(T)} w_{pq}.\]
Here, $T$ runs over spanning trees of $\Gamma$.
\end{lem}
\begin{proof}
	First, we consider the case $w_{pq} = 1$ for any $(p,q) \in E(\Gamma)$. Then by \cite[Proposition 2.4]{Postnikov2009}, 
	$\Vol(Z(\Gamma,w))$ equals  the number of spanning trees of $\Gamma$.  In fact, 
	the zonotope considered in \cite[Definition 2.2]{Postnikov2009} is
	\[Z_\Gamma = \sum_{(p,q) \in E(\Gamma)} [e_p,e_q].\]
	Note that for each $(p,q) \in E(\Gamma)$ with $p < q$
	\[ t_{pq}e_p + (1-t_{pq})e_q = t_{pq}(e_p-e_q) + e_q, \quad t_{pq} \in [0,1]\] 
	so that 
	\[ [e_p,e_q] = [0,e_p-e_q] + e_q.\]
	In particular, $Z_\Gamma$ is a translate of $Z(\Gamma,w)$ and they have the same volume.

     In the proof of \cite[Proposition 2.4]{Postnikov2009}, $Z_\Gamma$ is dissected into half-open
     parallelepipeds indexed by the spanning trees of $\Gamma$, each of unit volume.

The general case follows by scaling. Indeed, replacing the segment
$[0,e_p-e_q]$ by $[0,w_{pq}(e_p-e_q)]$ scales the corresponding edge vector by $w_{pq}$.
For a fixed spanning tree $T$, the associated parallelepiped has its volume scaled by
$\prod_{(p,q) \in E(T)} w_{pq}$. Summing over all spanning trees, we obtain
\[
\Vol(Z(\Gamma,w)) = \sum_{T} \prod_{(p,q) \in E(T)} w_{pq}.
\]
\end{proof}

\begin{lem}[Weighted Cayley identity]\label{lem-cayley}
	Let $m_j>0$, $0 \leq j \leq r$ be positive real numbers and $m = \sum_{j=0}^r m_j$. 
	Let $(\Gamma,w)$ be the following weighted complete graph
\[V(\Gamma) = \{0,\ldots,r\}, \quad E(\Gamma) = \left\{ (p,q), 0 \leq p < q \leq r \right\}, \quad w = (w_{pq}), \quad w_{pq} = m_pm_q.\]
Then
\[\sum_T \prod_{(p,q) \in E(T)} w_{pq} = m^{r-1} \prod_{j=0}^r m_j.\]
Here, $T$ runs over spanning trees of $\Gamma$.	
\end{lem}
\begin{proof}
We apply Kirchhoff's matrix-tree theorem \cite[Theorem 5.6.8]{StanleyEC2} to $(\Gamma,w)$.
The weighted Laplacian $L=(\ell_{pq})_{0\leq p,q \leq r}$ for $(\Gamma,w)$ is given by
\[
\ell_{pp}=\sum_{p\not= q} w_{pq}=m_pm - m_p^2, \quad \ell_{pq}=-m_p m_q, \quad p \not= q.
\]
Kirchhoff's theorem asserts that any cofactor $\det(L^{(k)})$ (obtained by deleting the $k$-th row and column of $L$)
equals the weighted sum $\sum_T\prod_{ (p,q) \in E(T)}w_{pq}$ over spanning trees $T$ .

It therefore suffices to compute $\det(L^{(r+1)})$.
For this, write $D=\diag(m_0,\dots,m_{r-1})$, $u=(m_0,\dots,m_{r-1})^\top$ and $v =(\sqrt{m_0},\dots,\sqrt{m_{r-1}})^\top$. Then
\[
L^{(r+1)}= mD-uu^\top  = D^{1/2}\left(mI_r-vv^\top \right)D^{1/2}, \quad 
\det(L^{(r+1)})=\det(D)\,\det\left(mI_r-vv^\top\right).
\]
Using the matrix-determinant lemma, we obtain
\[
\det(mI_r-vv^\top) = m^r\det\left( I_r - \frac{1}{m} vv^\top \right) = m^r\left( 1 - \frac{1}{m}v^\top v \right) = m^r\left( 1 - 
\frac{1}{m} \sum_{j=0}^{r-1} m_j\right) = m^{r-1}m_r.\]
Thus,
\[\sum_T\prod_{ (p,q) \in E(T)}w_{pq} = \det(L^{(r+1)}) = m^{r-1}\prod_{j=0}^r m_j.\]
\end{proof}

\begin{proof}[The proof of Proposition \ref{prop:volume-polytope}]
	For each $0 \leq j \leq r$, write $m_j = \sqrt{d_0} n_j$. Consider the weighted complete graph $(\Gamma,w)$ in Lemma \ref{lem-cayley}
	with $w_{pq} = m_pm_q$. Then
	\[v_\chi = \Vol(P_\chi) = \Vol(C_\chi) = \Vol(Z(\Gamma,w)) = (\sqrt{d_0}n)^{r-1}\prod_{j=0}^r (\sqrt{d_0}n_j)
	= d_0^r n^{r-1} \prod_{j=0}^{r} n_j.\]
	We are done.
\end{proof}

We illustrate the above results by the following example.

\begin{example}
Consider the complete graph $\Gamma=K_3$ with the edge set $\{12,13,23\}$ and unweighted
edge weights $w_{ij}\equiv 1$. Let 
\[
H:=\{(x_1,x_2,x_3)\in\R^3:\ x_1+x_2+x_3=0\}.
\]
We endow $H$ with the measure  by the following map
\[
\varphi:\R^2 \xrightarrow{\sim} H,\qquad (x_1,x_2)\longmapsto (x_1,x_2,-x_1-x_2).
\] 

The associated graphical zonotope to $\Gamma$ is
\[
Z(\Gamma,1)=[0,e_1-e_2]+[0,e_1-e_3]+[0,e_2-e_3]\ \subset\ H.
\]

Under the map $\varphi$, we have
\[
e_1-e_3 \leftrightarrow v_{13} = (1,0),\qquad
e_2-e_3 \leftrightarrow v_{23} = (0,1),\qquad
e_1-e_2 \leftrightarrow v_{12} = (1,-1).
\]
Hence
\[
Z(\Gamma,1)=[0,v_{13}]+[0,v_{23}]+[0,v_{12}]\ \subset\ \R^2.
\]

One checks that $Z(\Gamma, 1)$ is the convex hexagon with
vertices
\[
(0,0),\ (0,1),\ (1,1),\ (2,0),\ (2,-1),\ (1,-1)
\]
in cyclic order. Thus, 
\[
\mathrm{Area}(Z(\Gamma, 1))=3.
\]

The graph $\Gamma$ has exactly $3$ spanning trees. Correspondingly, $Z(\Gamma, 1)$ decomposes
(up to boundaries) into $3$ parallelograms of lattice area $1$:
\begin{align*}
P_{(13,23)} &:= [0,v_{13}]+[0,v_{23}],\\
P_{(12,13)} &:= [0,v_{12}]+[0,v_{13}],\\
P_{(12,23)} &:= v_{13}+[0,v_{12}]+[0,v_{23}].
\end{align*}
Thus $\mathrm{Vol}_{dx}(Z(\Gamma, 1))=3=\#\{\text{spanning trees of }K_3\}$.

\begin{center}
\begin{tikzpicture}[scale=1.2, line join=round, line cap=round]
  \draw[->] (-0.4,0) -- (2.6,0) node[below] {$x_1$};
  \draw[->] (0,-1.4) -- (0,1.6) node[left] {$x_2$};

  \coordinate (A) at (0,0);
  \coordinate (B) at (0,1);
  \coordinate (C) at (1,1);
  \coordinate (D) at (2,0);
  \coordinate (E) at (2,-1);
  \coordinate (F) at (1,-1);

  \fill[gray!15] (A) -- (1,0) -- (C) -- (B) -- cycle;

  \fill[gray!25] (B) -- (C) -- (D) -- (1,0) -- cycle;

  \fill[gray!35] (1,0) -- (D) -- (E) -- (F) -- cycle;

  \draw[thick] (A) -- (B) -- (C) -- (D) -- (E) -- (F) -- cycle;

  \fill (A) circle (1.2pt) node[below left] {$(0,0)$};
  \fill (B) circle (1.2pt) node[left] {$(0,1)$};
  \fill (C) circle (1.2pt) node[above] {$(1,1)$};
  \fill (E) circle (1.2pt) node[right] {$(2,-1)$};
  \fill (F) circle (1.2pt) node[below] {$(1,-1)$};

  \draw[->,thick] (A) -- (1,0) node[midway,below] {$v_{13}$};
  \draw[->,thick] (A) -- (0,1) node[midway,left] {$v_{23}$};
  \draw[->,thick] (A) -- (1,-1) node[midway,below right] {$v_{12}$};

\end{tikzpicture}
\end{center}
\end{example}

\section{The distributions of integral points and orbital integrals}

Let $\chi(x) \in \CO[x]$ be a regular polynomial. We shall
assume that $\chi(0) \not=0$.  This assumption is harmless for
our main result (see Lemma \ref{lem:chi(0)} below).

Denote by $G = \SL_n$ and $\wt{G} = \GL_n$ over a number field $F$. Denote
by $\fg = \fs\fl_n$ and $\wt{\fg} = \fg\fl_n$ their Lie algebras. 
Let $\gamma \in \wt{\fg}(\CO)$ 
be regular semisimple with the characteristic polynomial $\chi(x)$. As we assume $\chi(0) \not=0$,  $\gamma \in \wt{G}(F)$. Denote by $K = F[x]/(\chi(x))$ and $R = \CO[x]/(\chi(x))$. 
We consider the conjugation action of $\wt{G}$ and $G$ on $\wt{\fg}$.
Denote by $\wt{T} \cong K^\times$ and $T \cong K^1$ the stabilizers of $\gamma$ in $\wt{G}$ and $G$ respectively.

For each $\kappa \in H^1(F,T(\BA_{\bar{F}})/T(\bar{F}))^\vee$, we have the $\kappa$-orbital integral
$\CO_\gamma^\kappa(f)$ for any $f \in C_c^\infty(\wt{\fg}(\BA))$. If $f = \otimes_v f_v$, $f_v \in C_c^\infty(\fg(F_v))$ is
decomposable, 
\[\CO_\gamma^\kappa(f) = \prod_v \CO_\gamma^{\kappa_v}(f_v).\]
In the above, for each $v$, the local character $\kappa_v \in H^1(F_v,T)^\vee$ is the $v$-component of $\kappa$ with
\[\CO_\gamma^{\kappa_v}(f_v) = \sum_{\gamma' \sim_\st \gamma} \kappa_v(\inv_v(\gamma',\gamma))\CO_{\gamma'}(f_v),
\quad \CO_{\gamma'}(f_v) = \int_{T_{\gamma'}(F_v) \bs G(F_v)} f_v( \gamma' \cdot g)dg, \quad \gamma' \cdot g = g^{-1} \gamma' g.\]
Here, 
\begin{itemize}
\item $\gamma' \in \wt{\fg}(F_v)$ runs over $G(F_v)$-conjugacy classes which are stable conjugate to $\gamma$ (or equivalently, 
	$\wt{G}(F_v)$-conjugate to $\gamma$).
\item $\inv_v(\gamma',\gamma) \in H^1(F_v,T)$ is represented by the following cocycle
\[\sigma \in \Gamma_v \mapsto g\sigma(g)^{-1}, \quad \gamma' = \gamma \cdot g, \quad g \in G(\ov{F_v}).\]
\end{itemize}

The measure defining the $\kappa$-orbital integrals is normalized as follows. For each $v$, we have
\[\wt{T}(F_v) \bs \wt{G}(F_v) = X(F_v)  = \bigsqcup_{\gamma' \sim_\st \gamma} T_{\gamma'}(F_v) \bs G(F_v)\]
For each conjugacy class $\gamma' \in  \wt{\fg}(F_v)$ which is
stable conjugate to $\gamma$, $T_{\gamma'}$ is the stabilizer of $\gamma'$ in $G$. For each $\gamma'$, 
$T_{\gamma'}(F_v) \bs G(F_v)$ is open in $X(F_v)$. The quotient measure on 
$T_{\gamma'}(F_v) \bs G(F_v)$ is given by the quotient measure $d^0\mu$ on $\wt{T}(F_v) \bs \wt{G}(F_v) = K_v^\times \bs \GL_n(F_v)$ (see Section
\ref{Haar}). In particular, for each $v<\infty$, the quotient measure  is given by the measures on $\GL_n(F_v)$ and
$K_v^\times$ such that both $\GL_n(\CO_v)$ and $\CO_{K,v}^\times$ have volume one.

\begin{thm}\label{thm-orbit}
   For the homogeneous space $X = T \bs G$, we have
   \[N(X,\CT) \sim v_\chi C^\Tam_0 \sum_\kappa \CO_\gamma^\kappa(f^\CT) (\log \CT)^r.\]
   Here, 
   \begin{itemize}
     \item  $v_\chi$
    is the volume of certain polytope defined in Proposition \ref{prop: c_T} and  determined in Proposition \ref{prop:volume-polytope}. 
	 \item $C^\Tam_0$ is the constant in Proposition \ref{meas-comp}.
	 \item $\kappa$ runs over characters on $H^1(F,T(\BA_{\bar{F}})/T(\bar{F}))$. 
	 \item  For each $\CT$, the test function $f^\CT = \otimes_{v}  f_v^\CT \in C_c^\infty(\wt{\fg}(\BA))$ is taken as follows.
		 If $v|\infty$, then $f_v^\CT \in C_c^\infty(\wt{\fg}(F_v))$ is chosen such that for each conjugacy class $\gamma'$ in the stable
		 conjugacy class of $\gamma$, the local orbital integral
		 \[\CO_{\gamma'}(f_v^\CT) = \Vol( (\gamma' \cdot G(F_v)) \cap X(F_v,\CT)), \quad
		 X(F_v,\CT) = \left\{ x \in X(F_v)\Big| ||x||_v \leq \CT \right\}.\]
		 If $v < \infty$, then $f_v^\CT = f_v = 1_{\wt{\fg}(\CO_v)}$ is the characteristic function of $\wt{\fg}(\CO_v)$.
   \end{itemize}
\end{thm}

\begin{remark}
This result is obtained in \cite{Lee} for  $X = G_\gamma \bs G$ where $G$ is an arbitrary semisimple and simply connected group and $\gamma \in G(F)$ is elliptic under the assumption of the equi-distribution property. In fact, if we assume the equi-distribution property, then 
the above theorem can be also generalized using our main result
in Section \ref{sec: counting-formula}.
\end{remark}

We now give the proof of Theorem \ref{thm-orbit}. The proof is
completely parallel to the proof of \cite[Theorem 1.1]{Lee} 
which is based on the result of Wei-Xu \cite[Theorem 4.3]{WX}
for elliptic orbits.  Note that we have the following
generalization for regular semisimple orbits.

\begin{thm}\label{thm-wx}
We have
\[N(X,\CT) \sim  v_\chi  C_0^\Tam \#\Pic(T) (\log \CT)^r\int_{X(\BA)^{\Br(X)}} f_{X}^\CT(x) dx .\]
We refer to Theorem \ref{thm-wx-general} for notations. The measure $dx$ is given by the measure $d^0 \mu$ on $\wt{T}(\BA) \bs \wt{G}(\BA)$ (see Section \ref{Haar}).
\end{thm}
\begin{proof}
This follows from Theorem \ref{thm-wx-general}.
The result requires that the homogeneous space $X$ satisfies the equi-distribution property
in Theorem \ref{thm: equi-distribution}. Moreover, note that  the measure in Theorem 
\ref{thm-wx-general} is the Tamagawa measure $d^\Tam \mu$ on $X(\BA)$.
Here, we consider the measure $d^0 \mu$ so that there is the constant $C = \frac{d^\Tam \mu}{d^0 \mu}$.
\end{proof}

Consider the following exact sequence of $F$-varieties
\[1 \longrightarrow T \longrightarrow G \longrightarrow X \longrightarrow 1.\]
This gives rise to the following commutative diagram of pointed sets
\[
\begin{tikzcd}
G(F)  \ar[r]  &  X(F)   \ar[r]\ar[d] &H^1(F, T) \ar[r]\ar[d] & H^1(F,G)\ar[d]\\
G(\BA) \ar[r] & X(\BA) \ar[r] &  H^1(F,T(\BA_{\bar{F}})) \ar[r]&  H^1(F,G(\BA_{\bar{F}}))
\end{tikzcd}
\]
where the connecting map $X(\BA) \ra H^1(F,T(\BA_{\bar{F}}))$ is given by $\inv(\cdot,\gamma)=\oplus_{v}\inv_v(\cdot,\gamma) $.

Denote by 
\[
  \ker_F(X)=\ker( H^1(F,T)\rightarrow H^1(F,G)), \qquad
  \ker_\A(X)=\ker(H^1\bigl(F,T(\A_{\overline F})\bigr)\rightarrow H^1\bigl(F,G(\A_{\overline F})\bigr)).
\]
Then the above morphism $H^1(F,T) \ra H^1(F,T(\BA_{\bar{F}}))$ induces the following one
\[
  \alpha : \ker_F(X) \longrightarrow \ker_\A(X).
\]

\begin{lem}
	We have
	\[X(\BA)^{\Br(X)} = \left\{ x \in X(\BA) \Big| \inv(x,\gamma) \in \Im(\alpha) \right\}.\]
         In particular, we have
\[
\xymatrix@C=3pc@R=3pc{
X(\BA)/G(\BA)  \ar[r]^\sim & \ker_\BA(X) \\
X(\BA)^\Br/G(\BA) \ar[r]^\sim \ar@{^{(}->}[u] & \Im(\alpha) \ar@{^{(}->}[u]
}
\]
\end{lem}
\begin{proof}
	This is from the study of the commutative diagram (3.1) in \cite{CTX}. This commutative diagram in \cite{CTX} holds 
    for very general
    homogeneous spaces. For more details, see the proof
	of \cite[Theorem 1.1]{Lee}.
\end{proof}

By Theorem \ref{thm-wx} and the above Lemma, we obtain
\[N(X,\CT) \sim  v_\chi C^\Tam_0 \#\Pic(T) (\log \CT)^r\sum_{a \in \Im(\alpha)}\int_{x_a \cdot G(\BA)} f_{X}^\CT(x) dx.\]
Here, for each $a \in \ker_\BA(X)$, $x_a \in X(\BA)$ represents the $G(\BA)$-orbit with $\inv(x_a,\gamma) = a$.

\begin{lem}\label{lem:exercise}
    There is an exact sequence of pointed sets
\[
  \ker_F(X) \stackrel{\alpha}{\lra} \ker_\BA(X) \stackrel{\pi}{\lra} H^1(F,T(\BA_{\bar{F}})/T(\bar{F})).
\]
Here, for each $a\in \ker_{\BA}(X)$ corresponding to 
the $G(\BA)$-orbit $x_a\cdot G(\BA)$ in $ X(\BA)$, $\pi(a)$
is the image of $\inv(x_a,\gamma)$ under the morphism
$H^1(F,T(\BA_{\bar{F}})) \ra H^1(F,T(\BA_{\bar{F}})/T(\bar{F}))$.
\end{lem}
\begin{proof}
     This is \cite[Exercise 5.4]{Kal}. As $\gamma$ is assumed
     to be elliptic there, we give a proof. 

     We have the following commutative diagram
     \[
\begin{tikzcd}
1  \ar[r]  &  \ker_F(X)   \ar[r]\ar[d,"\alpha"]  &H^1(F, T) \ar[r]\ar[d] & H^1(F,G)\ar[d,hook]\\
1 \ar[r]   &  \ker_\BA(X)    \ar[r]\ar[dr,"\pi"] &  H^1(F,T(\BA_{\bar{F}})) \ar[r]\ar[d] &  H^1(F,G(\BA_{\bar{F}})) \\
&   &   H^1(F,T(\BA_{\bar{F}})/T(\bar{F}))   & 
\end{tikzcd}
\]
Here, as $G$ is simply connected, $G$  satisfies the Hasse principle,
in other words, we have an injective map
$H^1(F,G) \hookrightarrow H^1(F,G(\BA_{\bar{F}}))$ (in fact, in our case $H^1(F,G) = 0$ as $G$ is  semisimple and simply connected). 
Moreover, note that the sequence
\[H^1(F,T) \lra H^1(F,T(\BA_{\bar{F}})) \lra H^1(F,T(\BA_{\bar{F}})/T(\bar{F}))\]
is exact.

Now, the result follows by a standard diagram chase.

\end{proof}

\begin{proof}[The proof of Theorem \ref{thm-orbit}]

Consider the following function on $H^1(F,T(\BA_{\bar{F}})/T(\bar{F}))$
\[\frac{1}{\# H^1(F,T(\BA_{\bar{F}})/T(\bar{F}))} \sum_{\kappa \in H^1(F,T(\BA_{\bar{F}})/T(\bar{F}))^\vee} \kappa.\]
By Lemma \ref{lem:exercise}, its pullback to $\ker_\BA(X)$ is the characteristic function on $\Im(\alpha)$. Therefore,
\[N(X,\CT) \sim v_\chi C_0^\Tam\sum_{a \in \ker_\BA(X)} \sum_{\kappa} \kappa(\pi(a)) \int_{x_a \cdot G(\BA)} f_{X}^\CT(x) dx (\log \CT)^r.\]
Here, we use Lemma \ref{lem-duality-picard}. Moreover, note that 
\begin{itemize}
	\item when $a$ runs over $\ker_\BA(X)$, $x_a$ runs over $G(\BA)$-conjugacy classes which are stable conjugate to $\gamma$.
	\item for each $a$, $\kappa(\pi(a)) = \kappa(\inv(x_a,\gamma))$.
\end{itemize}
Therefore,
\[N(X,\CT) \sim v_\chi C_0^\Tam\sum_\kappa \CO_\gamma^\kappa(f_X^\CT)(\log \CT)^r\]
where $\CO_\gamma^\kappa(f_X^\CT) = \prod_v \CO_\gamma^{\kappa_v}(f_{X,v}^\CT)$ with
\[\CO_\gamma^{\kappa_v}(f_{X,v}^\CT) = \sum_{\gamma' \sim_\st \gamma} \kappa_v(\inv(\gamma',\gamma)) \CO_{\gamma'}(f_{X,v}^\CT),
\quad \CO_{\gamma'}(f_{X,v}^\CT) = \int_{T_{\gamma'}(F_v) \bs G(F_v)} f_{X,v}^\CT(\gamma' \cdot g)dg.\]

If $v$ is archimedean, then $f_{X,v}^\CT$ is the characteristic function of $X(F_v,\CT)$ so that
\[\CO_{\gamma'}(f_{X,v}^\CT) = \Vol((\gamma' \cdot G(F_v)) \cap X(F_v,\CT)).\]
If $v$ is nonarchimedean, then $f_{X,v}^\CT = f_{X,v}$ is the characteristic function of $X(\CO_v)$ so that
\[\CO_{\gamma'}(f_{X,v}) = \int_{T_{\gamma'}(F_v) \bs G(F_v)} 1_{X(\CO_v)}(\gamma'\cdot g)dg = \int_{T_{\gamma'}(F_v) \bs G(F_v)} 1_{\wt{\fg}(\CO_v)}(\gamma'\cdot g)dg
 = \CO_{\gamma'}(1_{\wt{\fg}(\CO_v)}).\]
 We are done.
\end{proof}

\section{Archimedean local orbital integrals}

Let $\kappa \in H^1(F,T(\BA_{\bar{F}})/T(\bar{F}))^\vee$. By Corollary \ref{cor:kappa}, the pair $(\gamma,\kappa)$ determines
a subfield $E$ of $K$ containing $F$ with $E/F$ cyclic. By Proposition \ref{ell-endo}, the associated endoscopic subgroup   is 
$H = \Res_{E/F}^1 \GL_m$ with $[E:F]m = n$. Denote by $\wt{H} = \Res_{E/F} \GL_m$. 

In the next two sections, we study the $\kappa$-orbital integral
\[\CO_\gamma^\kappa(f^\CT) = \CO_\gamma^{\kappa_\infty}(f_\infty^\CT)\CO_\gamma^{\kappa_\fin}(f_\fin)\]
where $f^\CT \in C_c^\infty(\wt{g}(\BA))$ is the
test function in Theorem \ref{thm-orbit}.

First, we consider the archimedean case.

\begin{prop}\label{arch}
   The archimedean $\kappa$-orbital integral $\CO_\gamma^{\kappa_\infty}(f_\infty^\CT) = 0$ unless $\kappa_\infty = 1$, or
   equivalently, $E/F$ is unramified at all the archimedean places. If $E/F$ is unramified at archimedean places,
   \[\CO_\gamma^{\kappa_\infty}(f_\infty^\CT) \sim \frac{w_n \Vol(U_\infty)}{\sqrt{|\Delta_\chi|_\infty}} \cdot \CT^{[F:\BQ] d}.\]
   Here, 
   \begin{itemize}
   \item  $d = n(n-1)/2$.
   \item $w_n = \prod_{v|\infty} w_{n,v}$  with $w_{n,v}$ the volume for the unit ball 
	   \[B_{n,v} = \left\{ (x_j) \in F_v^{d} \Big| \sum_j |x_j|_v^{2/e_v} \leq 1 \right\}.\]
       Here, if $v$ is real, then $e_v = 1$ and if $v$ is complex, then
       $e_v = 2$.
   \item $U_\infty = \prod_{v|\infty} U_v$ is a maximal compact subgroup of $\GL_n(F_\infty)$ and the
	   measure defining $\Vol(U_\infty)$ is normalized at the end of  Section \ref{Haar}.
    \item $\Delta_\chi \in F^\times$ is the discriminant
   of $\chi(x)$.
   
   \end{itemize}
\end{prop}
\begin{proof}
	The proof is reduced to the local case. Let $v$ be an archimedean place of $F$.

	By our choice of $f_v^\CT$, we have
	\[\CO_\gamma^{\kappa_v}(f_v^\CT) = 
	\sum_{\gamma' \sim_\st \gamma} \kappa_v(\inv_v(\gamma',\gamma)) \Vol( (\gamma'\cdot G(F_v)) \cap X(F_v,\CT)). \]
    
    We have studied the cohomology group $H^1(F_v,T)$ in Proposition \ref{prop: archi-kappa}. 
    The cohomology group $H^1(F_v,T)$ is trivial unless $v$ is real, $n$ is even and $T_v \cong (\BC^\times)^{n/2}$. When
    $H^1(F_v,T)$ is nontrivial, $H^1(F_v,T) \cong \BZ/2\BZ$. In this case, if $\gamma'$ is  another conjugacy class  which
    is stable conjugate to $\gamma$, we have
    \[\Vol( (\gamma' \cdot G(F_v)) \cap X(F_v,\CT)) = \Vol( (\gamma \cdot G(F_v)) \cap X(F_v,\CT)).\]
    In fact, we note that the two orbits $\gamma' \cdot G(F_v)$ and $\gamma \cdot G(F_v)$ are conjugate by some $g_0 \in \GL_n(F_v)$
    with $\det g_0 = -1$ and the set $X(F_v,\CT)$ is stable under the conjugation of $g_0$. 

    Therefore, for all the cases,
    \[\CO_{\gamma}^{\kappa_v}(f_v^\CT) = \Vol( (\gamma \cdot G(F_v)) \cap X(F_v,\CT)) \sum_{\gamma' \sim_\st \gamma} \kappa_v(\inv_v(\gamma',\gamma)).\]
    In particular, $\CO_{\gamma}^{\kappa_v}(f_v^\CT) = 0$ unless $\kappa_v = 1$. By Proposition \ref{prop-arch-kappa}, for a real place $v$, $\kappa_v \not= 1$ if and only if
    $[E:F] = 2d$ is even with $E_v \cong \BC^d$. As
    $E/F$ is Galois, $\kappa_v \not=1$ if and only if $E/F$ is ramified at $v$. If $\kappa_v = 1$, then
    \[\CO_{\gamma}^{\kappa_v}(f_v^\CT) = \Vol( X(F_v,\CT)).\]
    We shall study the volume $\Vol(X(F_v,\CT))$ in the next lemma.

\end{proof}

\begin{lem}\label{lem: volumn-estimation}
	We have
	\[\Vol(X(F_v,\CT)) \sim \frac{w_{n,v} \Vol(U_v)}{\sqrt{|\Delta_\chi|_v}} \cdot \CT^{e_v d}.\]
\end{lem}
\begin{proof}
   The volume $\Vol( X(F_v,\CT))$ is studied well in \cite[page 280-282]{EMS} when $v$ is real. 
   Here, we just repeat the proof for the case $T$ is split which includes the case $F_v = \BC$.
  
   Assume
\[\chi(x) = (x-\mu_1)\cdots (x-\mu_n) \in F_v[x].\] 
   Up to conjugation by $\GL_n(F_v)$, we may assume
\[\gamma = \diag\left[ \mu_1,\ldots,\mu_n \right] \in \GL_n(F_v).\]
Then the stabilizer of $\gamma$ in $\GL_n(F_v)$ is the diagonal torus $A$. Denote by $N$ the unipotent subgroup of the 
Borel subgroup consisting of upper-triangular matrices. Denote by $U_v$ the maximal compact subgroup of $\GL_n(F_v)$, which is
$\RO(n)(\BR)$ when $v$ is real or $\RU(n)$ when $v$ is complex. As in the end of Section \ref{Haar}, the measure on $U_v$ is normalized as
\[\int_{A(F_v) \bs \GL_n(F_v)} f(g)dg = \int_{N(F_v)} \int_{U_v} f(nk) dn dk, \quad  f \in C_c(A(F_v) \bs \GL_n(F_v)).\]

Note that the norm $||\cdot||_v$ is invariant under the conjugate
action of $U_v$. We have
\[
\begin{aligned}
	X(F_v,\CT) &= \left\{ g \in A(F_v) \bs \GL_n(F_v) \Big| ||\gamma \cdot g||_v < \CT \right\} \\
	&= \left\{ (n,k) \in N(F_v) \times U_v \Big| ||n^{-1} \gamma n||_v < \CT \right\} \\
	&= \left\{ (n,k) \in N(F_v) \times U_v \Bigg| \sum_{1 \leq i \leq n} |\mu_i|_v^{2/e_v} + \sum_{1 \leq i < j \leq n} |w_{ij}|_v^{2/e_v} < \CT^2 \right\}.
\end{aligned}\]
Here, we denote by $w = n^{-1} \gamma n$. In fact, we note that for $i<j$
\[w_{ij} = (\mu_i-\mu_j)n_{ij} + w_{ij}'\]
where the number $w_{ij}'$ depends only on $\gamma$ and those $n_{kl}$ with $k < \ell$ and $\ell -k < j-i$. In particular, the Jacobian 
of the transformation from $(n_{ij})_{i < j}$ to $(w_{ij})_{i<j}$ is 
\[J = \prod_{1 \leq i<j \leq n} |\mu_i - \mu_j|_v = |\Delta_\chi|_v^{1/2}.\]
Therefore, we obtain that
\[\Vol(X(F_v,\CT)) \sim \frac{w_{n,v}\Vol(U_v)}{\sqrt{|\Delta_\chi|_v}} \cdot \CT^{e_v d}.\]

\end{proof}

\section{Nonarchimedean local orbital integrals}

Next, we consider the nonarchimedean case. We keep the same notations as in the previous two sections. 
In particular, we fix a finite place $v$ and 
a character $\kappa_v$ on $H^1(F_v,T)$. Let $\gamma \in \wt{\fg}(\CO_v) \cap \wt{G}(F_v)$ be regular semisimple with
the characteristic polynomial $\chi(x)$.

Firstly, we relate the $\kappa_v$-orbital integrals to certain
twisted orbital integrals.

\begin{prop}\label{lem:kappa-inv}
Let $\gamma \in \widetilde G(F_v)$ be regular  semisimple.
Let $K_v = F_v[\gamma]$, which is then an \'etale algebra of degree $n$ over $F_v$.
Let $\wt{T_\gamma}$ (resp.\ $T_\gamma$) be the centralizer of
$\gamma$ in $\widetilde G$ (resp.\ in $G$). Let $\gamma' \in \wt{G}(F_v)$ be stably conjugate to $\gamma$ (under the action of $G$), and let
$g_{\gamma'} \in \widetilde G(F_v)$ be such that
\[
  g_{\gamma'}^{-1}\,\gamma\,g_{\gamma'} \;=\; \gamma'.
\]
Let $\varepsilon_v: F_v^\times \to \BC^\times$ be a character which is
trivial on $\RN_{K_v/F_v}K_v^\times$, and let
\[
  \kappa_v : H^1(F_v,T_\gamma) \longrightarrow \BC^\times
\]
be the character attached to $\varepsilon_v$ via the canonical isomorphism
$H^1(F_v,T_\gamma) \simeq F_v^\times / \RN_{K_v/F_v}K_v^\times$.
Then
\[
  \kappa_v\bigl(\inv(\gamma',\gamma)\bigr)
  \;=\;
  \varepsilon_v\bigl(\det(g_{\gamma'})\bigr).
\]
In particular, for any $f \in C_c^\infty(\wt{\fg}(F_v))$,
\[ \CO_\gamma^{\kappa_v}(f) = \CO_{\gamma}^{\varepsilon_v}(f) = \int_{\wt{T_\gamma}(F_v) \bs \wt{G}(F_v)} 
f(\gamma \cdot g)\varepsilon_v(\det g)dg.\]
Here, the measure on $\wt{T_\gamma}(F_v) \bs \wt{G}(F_v) = K_v^\times \bs \GL_n(F_v)$ is the one given in Section \ref{Haar}.
\end{prop}

\begin{proof}
In the proof, we omit the subscript $v$. 
Let $\wt{T_\gamma}$ (resp.\ $T_\gamma$) be the centralizer of $\gamma$ in
$\widetilde G = \GL_n$ (resp.\ $G = \SL_n$).
Then $\wt{T_\gamma} \simeq \mathrm{Res}_{K/F}\mathbb G_m$, and we have
a short exact sequence of $F$–tori
\[
  1 \longrightarrow T_\gamma
    \longrightarrow \widetilde T_\gamma
    \xrightarrow{\det} \mathbb G_m
    \longrightarrow 1.
\]
By Hilbert~90, $H^1(F,\wt{T_\gamma}) = 0$. Passing to $F$–points
and the Galois cohomology yields an exact sequence
\[
  1 \longrightarrow T_\gamma(F)
    \longrightarrow K^\times
    \xrightarrow{\ \RN_{K/F}\ } F^\times
    \longrightarrow H^1(F,T_\gamma)
    \longrightarrow 1.
\]
In particular, there is a canonical isomorphism
\[
  H^1(F,T_\gamma) \cong F^\times / \RN_{K/F}K^\times,
\]

Let $\gamma' \in \wt{G}(F)$ be stably conjugate to $\gamma$ (under the action of $G$). There exists $g \in G(\overline{F})$ with
$g^{-1}\gamma g = \gamma'$, and the element
\[
  \mathrm{inv}(\gamma',\gamma) \in H^1(F,T_\gamma)
\]
is represented by the $1$–cocycle
\[
  c_\sigma := g\, \sigma(g)^{-1} \in T_\gamma(\overline{F}),
  \qquad \sigma \in \Gamma_F.
\]

On the other hand, by assumption we have fixed
$g_{\gamma'} \in \widetilde{G}(F)$ such that
$g_{\gamma'}^{-1} \gamma g_{\gamma'} = \gamma'$.  Both $g$ and $g_{\gamma'}$
conjugate $\gamma$ to $\gamma'$, hence
\[
  z := gg_{\gamma'}^{-1} \in \wt{T_\gamma}(\overline{F}),
\]
and we write
\begin{equation}\label{eq:g-z-ggamma}
  g = z\,g_{\gamma'}.
\end{equation}
Taking determinant in \eqref{eq:g-z-ggamma}
gives
\[
  1 = \det(g) = \det(z)det(g_{\gamma'}),
  \qquad\text{so}\quad \det(z) = \det(g_{\gamma'})^{-1}.
\]

Using \eqref{eq:g-z-ggamma} and the $F$–rationality of $g_{\gamma'}$, we compute
\[
  c_\sigma
  = g\sigma(g)^{-1}
  = zg_{\gamma'}\bigl(\sigma(z)\sigma(g_{\gamma'})\bigr)^{-1}
  = zg_{\gamma'}\sigma(g_{\gamma'})^{-1}\sigma(z)^{-1}
  = z\sigma(z)^{-1}.
\]
Thus $\mathrm{inv}(\gamma',\gamma)$ is represented by the cocycle
$\sigma \mapsto z\,\sigma(z)^{-1}$.

Now consider the connecting homomorphism associated with the short exact
sequence of tori,
\[
  \delta : F^\times  \longrightarrow H^1(F,T_\gamma).
\]
For $a \in F^\times$, choose $b \in \wt{T_\gamma}(\overline{F})$ with
$\det(b) = a$; then $\delta(a)$ is represented by the cocycle
$\sigma \mapsto b^{-1}\sigma(b)$.
Applying this to $a = \det(g_{\gamma'})^{-1}$ and choosing $b = z$ gives
\[
  \mathrm{inv}(\gamma',\gamma)
  = \delta\bigl(\det(g_{\gamma'})^{-1}\bigr)^{-1}=[\det(g_{\gamma'})] \in F^\times/\RN_{K/F} K^\times
\]
via the isomorphism
$H^1(F,T_\gamma) \simeq F^\times / \RN_{K/F}K^\times$. By definition of $\kappa$, obtained from $\varepsilon$ via the 
identification in the statement of the proposition, we conclude
\[
  \kappa\bigl(\mathrm{inv}(\gamma',\gamma)\bigr)
  = \varepsilon\bigl(\det(g_{\gamma'})\bigr).
\]

Now, consider the twisted orbital integral $\CO_\gamma^\varepsilon(f)$
for any $f \in C_c^\infty(\wt{\fg}(F))$. By our normalization of Haar
measures, 
\[\CO_\gamma^\varepsilon(f) = \sum_{g \in \wt{T}(F) \bs \wt{G}(F)/G(F)}
\varepsilon(\det g)\int_{g^{-1} T(F) g \bs G(F)} f(\gamma \cdot gh)dh.\]
Note that there is a bijection
   \[\wt{T}(F) \bs \wt{G}(F)/ G(F) \stackrel{\sim}{\lra} \left\{ \gamma' \Big| \gamma' \sim_\st \gamma \right\}/\sim,
   \quad g \mapsto \gamma \cdot g.\]
As $\kappa(\inv(\gamma',\gamma)) = \varepsilon(\det(g))$ if $\gamma' = \gamma \cdot g$, we
obtain
\[\CO_\gamma^\varepsilon(f) = \sum_{\gamma' \sim_\st \gamma}
\kappa(\inv(\gamma',\gamma))\int_{T_{\gamma'}(F)\bs G(F)} f(\gamma' \cdot h)dh = \CO_\gamma^\kappa(f).\]
We are done.

\end{proof}

Consider the $\kappa$-orbital integral
\[\CO_\gamma^{\kappa_v}(f_v) = \CO_\gamma^{\kappa_v}(1_{\wt{\fg}(\CO_v)})
= \CO_\gamma^{\varepsilon_v}(1_{\wt{\fg}(\CO_v)}).\]

\begin{prop}\label{prop-vanishing}
Let $U_v = \wt{G}(\mathcal{O}_v)$ be the standard hyperspecial subgroup of $\wt{G}(F_v)$  and let $\mathcal{H} = \CH(\wt{\fg}(F_v),U_v)$ be the spherical Hecke algebra of
compactly supported, bi-$U_v$–invariant functions on $\wt{\fg}(F_v)$.
For $f \in \mathcal{H}$ and a regular semisimple $\gamma \in \wt{G}(F_v)$, consider the
twisted orbital integral $\CO_\gamma^{\varepsilon_v}(f)$.

Assume that $\varepsilon_v$ is ramified, or equivalently, $E_v/F_v$ is
ramified. Then, for any $f \in \mathcal{H}$, one has
\[
\CO^{\varepsilon_v}_\gamma(f) \;=\; 0.
\]

\end{prop}

\begin{proof}
       Note that the character $\varepsilon_v$ satisfies the commutative diagram 
\[
\begin{tikzcd}
F_v^\times/\RN_{K_v/F_v}K_v^\times \arrow[r] \arrow[dr, "\varepsilon_v"] & F_v^\times/\RN_{E_v/F_v}E_v^\times\ar[d,hook, "\eta_{E_v/F_v}"]\\
& \BC^\times
\end{tikzcd}
\]
where the horizontal arrow is the natural projection and $\eta_{E_v/F_v}$ 
is the character associated to the extension $E_v/F_v$ via the local 
class field theory.  Therefore, $\varepsilon_v$ is ramified if and only if $E_v/F_v$ is ramified.

    Consider the change of variables $g \mapsto gk$ for an arbitrary element $k \in U_v$:
    \begin{align*}
        \CO^{\varepsilon_v}_\gamma(f) &= \int_{\wt{T_\gamma}(F_v)\backslash \wt{G}(F_v)} f((gk)^{-1}\gamma (gk)) \varepsilon_v(\det(gk)) \, d(gk) \\
          &= \int_{\wt{T_\gamma}(F_v)\backslash \wt{G}(F_v)} f(k^{-1} g^{-1} \gamma g k) \epsilon(\det g) \varepsilon_v(\det k) \, dg\\
          &= \varepsilon_v(\det k) \int_{\wt{T_\gamma}(F_v)\backslash \wt{G}(F_v)} f(g^{-1} \gamma g ) \varepsilon_v(\det g) \, dg\\
          &=\varepsilon_v(\det k) \CO^{\varepsilon_v}_\gamma(f) .
    \end{align*} 
    Since $\varepsilon_v$ is ramified, there exists a unit $u \in \CO_v^\times$ such that $\varepsilon_v(u) \neq 1$. Since the determinant map is surjective on $\CO_v^\times$, there exists an element $k_0 \in U_v$ such that $\det(k_0) = u$.
    
    Performing the change of variables $g \mapsto g k_0$ gives:
    \[
         \CO^{\varepsilon_v}_\gamma(f)(1 - \varepsilon_v(u)) = 0
    \]
   which implies that $ \CO^{\varepsilon_v}_\gamma(f)= 0$.
\end{proof}

\begin{remark}
    In the non-twisted case, Proposition \ref{prop-vanishing} is a special case of \cite[Proposition 7.5]{Kott}.
\end{remark}

We now consider the case  $E_v/F_v$
is unramified.

To state the fundamental lemma, we need to introduce several notations. Denote by
\[\Delta(\gamma) = \prod_{\alpha \in R(\wt{T},\wt{G})} (\alpha(\gamma)-1)\] 
where $R(\wt{T},\wt{G})$ is the set of roots of $\wt{G}$ with respect to $\wt{T}$. We identify $\gamma$ as an element $\gamma_E$ in $\wt{H}(F) = \GL_m(E)$ 
by viewing $K$ as a vector space over $E$. Denote by
\[\Delta(\gamma_E) = \prod_{\alpha \in R(\wt{T_H},\wt{H})} (\alpha(\gamma_E)-1)\]
where $\wt{T_H}$ is the stabilizer of $\gamma_E$ in $\wt{H}$.

We shall compare the local Haar measure on $\wt{T}(F_v) \bs \wt{G}(F_v)$ and $T(F_v) \bs G(F_v)$. There is a constant $c_{G,v}$
\[d^0 \mu|_{T(F_v) \bs G(F_v)} = c_{G,v} \cdot \frac{d^0g}{d^0 t}\]
where $d^0g$ is the measure on $G(F_v)$ with $\Vol(G(\CO_v),d^0g) = 1$ and $d^0 t$ is the measure on $T(F_v)$) with $\Vol(T(\CO_v),d^0 t) = 1$.
Similarly, there is a  constant $c_{H,v}$ such that
\[d^0 \mu_H |_{T_H(F_v) \bs H(F_v)} = c_{H,v} \cdot \frac{d^0h}{d^0 t}.\]
Here, if we denote by $\wt{H} = \Res_{E/F} \GL_m$ and $\wt{T_H}$ the stabilizer of $\gamma_E$ in $\wt{H}$, then
the measure $d^0\mu_H$ on $\wt{T_H}(F_v) \bs \wt{H}(F_v)$ is the quotient measure $d^0 \mu$ on $K_v^\times \bs \GL_m(E_v)$ (see Section
\ref{Haar}). 

Consider the stable orbital integral
\[\CS\CO_{\gamma_E}(1_{\wt{\fh}(\CO_v)}) = \sum_{\gamma_E' \sim_\st \gamma_E} \CO_{\gamma_E'}(1_{\wt{\fh}(\CO_v)}).\]
Here, $\gamma_E'$ runs over $H(F_v)$-conjugacy classes  which are stable conjugate to $\gamma_E$. The measure defining this
stable orbital integral is given by $d^0\mu_H$.

Denote by $\CH(\wt{G}(F_v),U_v)$ the spherical Hecke algebra of $\wt{G}_v$,
consisting of compactly supported, bi-$U_v$-invariant functions on 
$\wt{G}(F_v)$. Similarly, denote by $\CH(\wt{H}(F_v),U_v\cap \wt{H}(F_v))$
the spherical Hecke algebra of $\wt{H}_v$. 

Consider the following morphism on these two spherical Hecke algebras given by
Waldspurger in \cite[II 3.]{Wald91}
\[
b:\mathcal{H}(\wt{G}(F_v),U_v)\lra \mathcal{H}(\wt{H}(F_v),U_v\cap \wt{H}(F_v))
\]
such that under the Satake isomorphism
(we do not distinguish between $b$ and $\wh b$ in loc.cit.)
\begin{align*}
\mathcal{H}(\wt{G}(F_v),U_v)&\cong\BC[X_1,\ldots, X_n, X_1^{-1},\ldots, X_n^{-1}]^{S_n}\\
\mathcal{H}(\wt{H}(F_v),U_v\cap \wt{H}(F_v))&\cong\BC[Y_1,\ldots, Y_m, Y_1^{-1},\ldots, Y_m^{-1}]^{S_m}\\
X_{id+j}&\mapsto \varepsilon_v(\varpi_v)^jY_{i+1}^{1/d}.
\end{align*}

\begin{thm}[Fundamental lemma \cite{LW17}] \label{fl-LW} Assume $E_v/F_v$ is unramified. Then for any $f \in \CH(\wt{G}(F_v),U_v)$
\begin{equation}
|\Delta(\gamma)|_v^{1/2} c_{G,v}^{-1} \CO_\gamma^{\varepsilon_v}(f) =  |\Delta(\gamma_E)|_v^{1/2} c_{H,v}^{-1} 
	\CS\CO_{\gamma_E}(b(f)).
\end{equation}	
\end{thm}

\begin{remark}
Endoscopy was introduced by Langlands \cite{Lan79} for stabilizing the trace formula.
Langlands--Shelstad \cite{LS87} subsequently developed the formalism of transfer factors and formulated precise matching statements for orbital integrals, of which the fundamental lemma (FL) is the central local identity. Important progress was made by Hales, Kottwitz, Waldspurger and others. Laumon–Ng\^o proved the FL for unitary groups and Ng\^o \cite{NGO10} proved the general Lie algebra FL
under minor restrictions on characteristics. We refer to \cite{NGO11} for a detailed exposition on FL.
In \cite{LW17}, Lemaire and Waldspurger showed that the twisted FL holds for the whole spherical Hecke algebra for unramified groups over any $p$-adic fields.
\end{remark}

\begin{cor}\label{fl} Assume $E_v/F_v$ is unramified. Then 
\begin{equation}\label{eqn: fundamental-Lie-algebra}
|\Delta(\gamma)|_v^{1/2} c_{G,v}^{-1} \CO_\gamma^{\varepsilon_v}(1_{\wt{\fg}(\CO_v)}) =  |\Delta(\gamma_E)|_v^{1/2} c_{H,v}^{-1} 
	\CS\CO_{\gamma_E}(1_{\wt{\fh}(\CO_v)}).
\end{equation}	
\end{cor}
\begin{proof}

We first consider the case that $E_v$ is an unramified field extension of $F_v$ (the elliptic case). Denote by $d=[E_v:F_v]$.

As our $\gamma$ belongs to $\wt{G}(F_v)$, it is enough to consider the
characteristic function $1_{\wt{\fg}(\CO_v) \cap \wt{G}(F_v)}$ on 
the set $\wt{\fg}(\CO_v)\cap \wt{G}(F_v)$. We now break this characteristic function  into a sum of spherical functions on
$\wt{G}(F_v)$ and then apply the above Theorem \ref{fl-LW} to yield (\ref{eqn: fundamental-Lie-algebra}).

Note that we have 
\[
1_{\wt{\fg}(\CO_v) \cap \wt{G}(F_v)}=\sum_{j=0}^{\infty}\Phi_j\]
where $\Phi_j$ is the characteristic function on the set of $x \in \wt{\fg}(F_v)$ with $|\det(x)|_v = q_v^{-j}$.

On the Lie algebra $\wt{\fh}$ of $\wt{H}$, for $\wt{\fh}(\CO_v) = 
\fg\fl_m(\CO_{E_v})$ with $md = n$, we also have
\[
1_{\wt{\frak{h}}_v(\CO_v) \cap \wt{H}(F_v)}=\sum_{j=0}^{\infty}\Phi_j^H
\]
where $\Phi_j^H$ is the characteristic function on the set of
$x \in \wt{\fh}(F_v)$ with $|\det(x)|_{E_v} = q_v^{-dj}$.

In particular, it suffices to show that 
\[ 
|\Delta(\gamma)|_v^{1/2} c_{G,v}^{-1} \CO_\gamma^{\kappa_v}(\Phi_j)=
\begin{cases}
|\Delta(\gamma_E)|_v^{1/2} c_{H,v}^{-1} 
	\CS\CO_{\gamma_E}(\Phi_{j/d}^{H}),&\text{ \ if\ }d\mid j\\
    0, &\text{ \ otherwise.\ }
\end{cases}  
\]

Denote the Satake transform of $f$ by $\widehat f$. By Theorem \ref{fl-LW}, it is enough to show that
\begin{equation}\label{eqn: image-characteristic}
b(\wh \Phi_{j})=\begin{cases}
\wh \Phi_{j/d}^H, &\text{ \ if \ }d\mid j,\\
0, &\text{ \ otherwise\  }.
\end{cases}
\end{equation}
A simple calculation shows the following (see  \cite[(3.12)]{Gross98})
\begin{align*}
\wh \Phi_{j}=q_v^{\frac{(n-1)j}{2}}\tr(\Sym^{j}\BC^n)=q_v^{\frac{(n-1)j}{2}}\sum_{i_1\leq i_2\leq \cdots \leq i_{j}} X_{i_1}X_{i_2}\cdots X_{i_{j}},\\
\wh \Phi_{j}^H=q_v^{\frac{(n-1)dj}{2}}\tr(\Sym^{j}\BC^m)=q_v^{\frac{(n-1)dj}{2}}\sum_{i_1\leq i_2\leq \cdots \leq i_{j}} Y_{i_1}Y_{i_2}\cdots Y_{i_{j}}.
\end{align*}
 Now we obtain (\ref{eqn: image-characteristic}).

More generally, $E\cong \prod_{i}E_{i,v}$ is a product of unramified extensions. In this case, $\wt{H}_v$ is elliptic in a Levi subgroup of $\wt{G}$, which is a product of several $\GL_m$. The same argument shows that (\ref{eqn: fundamental-Lie-algebra}) holds.
\end{proof}

\begin{lem}\label{lem-Delta}
 We have
\[|\Delta_\chi|_v = |\Delta_{K/F}|_v[\CO_{K,v}:R_v]^{-2}.\]
Similarly, if we denote by $\chi_E(x) \in \CO_E[x]$
the characteristic polynomial of $\gamma_E$, then
\[|\RN_{E/F}\Delta_{\chi_E}|_v = |\RN_{E/F}(\Delta_{K/E})|_v [\CO_{K,v}: R_{E,v}]^{-2}.\]
  In particular, under our condition that $E/F$ is unramified everywhere
  \[ \frac{|\Delta(\gamma_E)|_v^{1/2}}{|\Delta(\gamma)|_v^{1/2}} = \frac{|\RN_{E/F}\Delta_{\chi_E}|_v^{1/2}}{|\Delta_{\chi}|^{1/2}_v} = [R_{E,v}:R_{v}].\]
  Here, we use the notation $[R_{E,v}:R_{v}] = [\CO_{K,v}:R_{v}]/[\CO_{K,v}:R_{E,v}]$.
\end{lem}
\begin{proof}
    We only need to prove the equality involving $E$, as the first one is the special case that $E=F$.

    By construction we have 
    \[
    \wt{H}(F_v)=\left\{(g_w)\in \prod_{w\mid v}\GL_m(E_w) \right\}.
    \]
In particular by definition we have
$\gamma_E=(\gamma_{w})\in \wt{H}(F_v)$ and
\[
\Delta(\gamma_E)=\prod_{w\mid v,\alpha\in \Phi(\GL_m)}(\alpha(\gamma_{w})-1)\in F_v^\times,
\]
where $\Phi(\GL_m)$ denotes the set of roots for  
$\GL_m$. By our normalization of norms in Section \ref{Haar}, this implies 
\[
|\Delta(\gamma_E)|_v=\prod_{w\mid v,\alpha\in \Phi(\GL_m)}|(\alpha(\gamma_{w})-1)|_w.
\]
On the other hand, let $\chi_w$ be the characteristic polynomial of $\gamma_w$, then we have 
\[
\prod_{\alpha\in \Phi(\GL_m)}|(\alpha(\gamma_{w})-1)|_w=|\Delta(\chi_w)|_w |\det(\gamma_{w})|_w^{-1}
\]
which combines 
$|\det(\gamma)|_v=\prod_{w\mid v}|\det(\gamma_{w})|_w$ 
to give
\[
|\det(\gamma)|_v|\Delta(\gamma_E)|_v=\prod_{w\mid v}|\Delta(\chi_w)|_w.
\]
We also have 
\begin{align*}
|\Delta_{K_w/E_w}|_w=|\Delta(\chi_w)|_w[\CO_{K_w}: R_{E_w}]^{2}.
\end{align*}
Putting the above formulas together, we obtain
\[
|\RN_{E/F}\Delta_{\chi_E}|_v=|\det(\gamma)|_v|\Delta(\gamma_E)|_v=|\RN_{E/F}(\Delta_{K/E})|_v [\CO_{K_v}: R_{E_v}]^{-2}.
\]
Here we use the fact that $E$ is unramified over $F$ hence $[E_w:F_v]$ is independent of $w\mid v$ and 
\begin{align*}
[\CO_{K_v}: R_{E_v}]&=\prod_{w\mid v}[\CO_{K_w}: R_{E_w}],\\
|\RN_{E/F}(\Delta_{K/E})|_v&=\prod_{w\mid v}|\Delta_{K_w/E_w}|_w.
\end{align*}

\end{proof}

\begin{lem}\label{lem-c_G}
	We have 
	\[c_{G,v} = [\CO_v^\times: \RN_{K_v/F_v} \CO_{K,v}^\times]^{-1}, \quad c_{H,v} = [\CO_{E_v}^\times: \RN_{K_v/E_v} \CO_{K,v}^\times]^{-1}.\]
	Moreover, if we denote by $\delta_v(E) = \frac{c_{G,v}}{c_{H,v}}$, then 
	\[\delta_v(E) =  [\CO_{E_v}^1: \CO_{E_v}^1 \cap \RN_{K_v/E_v} \CO_{K_v}^\times]\]
	where $\CO_{E_v}^1$ is the kernel of $\RN_{E_v/F_v}: \CO_{E_v}^\times \ra \CO_v^\times$. In particular, if $E_v = F_v$ or $K_v/E_v$ is unramified, then
	$\delta_v(E) = 1$.
\end{lem}
\begin{proof}
    In this proof, we denote by $\tilde T$ the stabilizer of $\gamma$ in $\tilde G$ and $T$ the stabilizer of $\gamma$ in $G$. Then
    $T(F_v) \bs G(F_v)$ embeds into $\tilde T(F_v) \bs \tilde G(F_v)$ as an open subset.
    Denote by $d\mu$ the quotient measure on $\tilde T(F_v) \bs \tilde G(F_v)$ such that $T(\CO_v) \bs \tilde G(\CO_v)$ has volume $1$. Denote
    by $d\mu_1$ the quotient measure on $T(F_v) \bs G(F_v)$ such that $T(\CO_v) \bs G(\CO_v)$ has volume $1$. Then
    \[d\mu|_{T(F_v) \bs G(F_v)} = c_{G,v} d\mu_1.\]

    The constant $c_{G,v}$ is independent of the embedding of $\tilde T_v$ to $\tilde G_{v}$ (embeddings are unique up to  conjugations of $\tilde G(F_v)$).
    We may assume that $\tilde T(F_v) \cap \tilde G(\CO_v) = \tilde T(\CO_v)$. Under this assumption, we have the following commutative diagram
\[
\xymatrix@C=3pc@R=3pc{
T(F_v) \bs G(F_v) \ar@{^{(}->}[r] & \tilde T(F_v) \bs \tilde G(F_v) \\
T(\CO_v) \bs G(\CO_v) \ar@{^{(}->}[r] \ar@{^{(}->}[u] & \tilde T(\CO_v) \bs \tilde G(\CO_v) \ar@{^{(}->}[u]
}
\]
    By our normalizations on $d\mu$ and $d\mu_1$, $c_{G,v}$ equals the inverse of the cardinality for the following double coset
    \[\tilde T(\CO_v) \bs \tilde G(\CO_v) / G(\CO_v) \stackrel{\sim}{\lra} \RN_{K_v/F_v} \CO_{K_v}^\times \bs \CO_v^\times\]
    where the bijection is given by the determinant map.  We are done.

    The proof of the claim for $c_{H,v}$ is similar. 

    Now, consider the norm map
    \[\RN_{E_v/F_v}: \RN_{K_v/E_v} \CO_{K_v}^\times \bs \CO_{E_v}^\times \lra \RN_{K_v/F_v} \CO_{K_v}^\times \bs \CO_v^\times.\]
    As $E_v/F_v$ is unramified, this map is surjective so that $\delta_v(E)$ equals the cardinality of the kernel of this map. It
    is easy to see that the kernel is
    \[\RN_{K_v/E_v} \CO_{K_v}^\times \bs \left(\RN_{K_v/E_v} \CO_{K_v}^\times \cdot \CO_{E_v}^1 \right) \cong 
    \left(\CO_{E_v}^1 \cap \RN_{K_v/E_v} \CO_{K_v}^\times\right) \bs  \CO_{E_v}^1.\]
    If $K_v/E_v$ is unramified, then $\RN_{K_v/E_v} \CO_{K_v}^\times = \CO_{E_v}^\times$ and the kernel is trivial.

\end{proof}

The proof of the following Lemma is the same as the proof in
Proposition \ref{lem:kappa-inv}.
\begin{lem}\label{lem-SO}
	Consider the following orbital integral on $\wt{H}$
	\[\CO_{\gamma_E}(1_{\wt{\fh}(\CO_v)}) = \int_{\wt{T}(F_v) \bs \wt{H}(F_v)} 1_{\wt{\fh}(\CO_v)}( \gamma_E\cdot g)dg\]
	where the measure is $d^0\mu_H$, that is, the same as the one in the stable orbital integral for $\CS\CO_{\gamma_E}(1_{\wt{\fh}(\CO_v)})$.
	Then
	\[\CO_{\gamma_E}(1_{\wt{\fh}(\CO_v)}) = \CS\CO_{\gamma_E}(1_{\wt{\fh}(\CO_v)}).\]
\end{lem}

Combine Proposition \ref{lem:kappa-inv}, Proposition \ref{prop-vanishing}, Theorem \ref{fl}, Lemma \ref{lem-Delta}, Lemma \ref{lem-c_G} and Lemma \ref{lem-SO} we obtain the following:
\begin{prop}\label{non-arch}
	The nonarchimedean $\kappa$-orbital integral $\CO_\gamma^{\kappa_\fin}(1_{\wt{\fg}(\wh{\CO})}) = 0$ unless $E/F$ is unramified 
	at all the nonarchimedean places. If so
	\[ \CO_\gamma^{\kappa_\fin}(1_{\wt{\fg}(\wh{\CO})}) =  [R_E:R] \delta(E)  \CO_{\gamma_E}(1_{\wt{\fh}(\wh{\CO})}).\]
        Here, $\delta(E) = \prod_{v<\infty} \delta_v(E)$ with $\delta_v(E) =  [\CO_{E_v}^1: \CO_{E_v}^1 \cap \RN_{K_v/E_v} \CO_{K_v}^\times]$.  
\end{prop}

\section{Orbital integrals and Dedekind zeta values}

We review the relation between Dedekind zeta values and orbital integrals.

Let $F$ be a number field. Denote by $\wt{G} = \GL_n$, viewed as an algebraic group over $F$ and $\wt{\fg} = \fg\fl_n$ be its Lie algebra. Let $\gamma \in \wt{G}(F)$ be regular semisimple with the characteristic polynomial $\chi(x) = \prod_{i=0}^r \chi_i(x) \in \CO[x]$. Denote
by $K = F[x]/(\chi(x))$ and $R = \CO[x]/(\chi(x))$ an order of $\CO_K$. Denote by $\zeta_R(s)$ the zeta function associated to
$R$.

\begin{thm}[Yun] \label{zeta-orb} We have 
	\[\frac{\zeta_R(s)}{\zeta_K(s)}\Big|_{s=1} = \frac{\CO_\gamma(1_{\wt{\fg}(\wh{\CO})})}{[\CO_K:R]}.\]
\end{thm}
\begin{proof}
This follows from the first part of 
proof in \cite[Theorem 1.2]{Yun}. In \cite{Yun}, the global base field $F$ is assumed to be $\BQ$ and $K$ is a field. In fact, the proof
is valid for an arbitrary number field $F$ and an \etale algebra
extension of $K$ of $F$. For the convenience of the readers, we produce the main steps.

Let $R=\mathcal{O}[\gamma]$. For each $v<\infty$, Yun introduces the local zeta factor
\[
J_{R_v}(s)=\sum_{j\geq 0}\#\mathrm{Quot}_{R_v^\vee}^j q_v^{-js}
\]
where $q_v$ is the cardinality of the residue field of $\CO_v$, $R_v^\vee=\{x\in K_v| \tr(c_v^{-1}xR_v)\in \CO_v \}$ with $c_v$ a generator of the different ideal of $K_v/F_v$ and $\mathrm{Quot}_{R_v^\vee}^j$
is the set of fractional $R_v$-ideals $M\subseteq R_v^\vee$ such that $R_v^\vee/M$ is of length $j$ over $\CO_v$.

Now it follows from \cite[Theorem 2.5 (2)]{Yun} that 
\[
J_{R_v}(s)=[\CO_{K_v}: R_v]^{-s}\zeta_{K_v}(s)\widetilde{J}_{R_v}(s)
\]
such that $\widetilde{J}_{R_v}(1)=\CO_{\gamma}(1_{\wt{\fg}(\CO_v)})$ (see \cite[Corollary 4.6(2)]{Yun}). (Note that our  
notation $[\CO_{K_v}: R_v] = \#(\CO_{K_v}/R_v)$ 
is different from the one in \cite{Yun}.) 

Globally, the zeta function of $R$ 
\[
\zeta_{R}(s)=\prod_{v< \infty}J_{R_v}(s)=\sum_{M\subseteq R^\vee}(\#R^\vee/M)^{-s}
\]
with $R^\vee=\Hom_{\BZ}(R,\BZ)$.
It follows that we have 
\[
\zeta_R(s)=[\CO_K: R]^{-s}\zeta_K(s)\prod_{v}\widetilde{J}_{R_v}(s)
\]
which implies 
	\[\frac{\zeta_R(s)}{\zeta_K(s)}\Big|_{s=1} = 
    \frac{\CO_\gamma(1_{\wt{\fg}(\wh{\CO})})}{[\CO_K:R]}.\]

\end{proof}

\begin{prop}\label{prop-residue-formula}
We have 
\begin{equation}\label{residue-formula}
\lim_{s \ra 1} (s-1)^{r+1} \zeta_R(s) = \frac{2^{r_{1,K}}(2\pi)^{r_{2,K}}\#\Cl(R) R_K}{w_K \sqrt{D_R}} \sum_{ [M] \in \Cl(R) \bs \ov{\Cl}(R)}
\frac{\#(\CO_K^\times/\Aut(M))}{\#\Stab_{\Cl(R)}([M])}.
\end{equation}
Here, 
\begin{itemize}
\item $r_{1, K},r_{2,K}$ are the sum of the numbers of real and complex places of each simple factor of $K$ respectively, $R_K$ is the regulator of $K$, $w_K$ is the cardinality of
roots of unity in $\CO_K$.
\item   A fractional $R$-ideal is a finitely generated $R$-submodule $M$ in $K$ such that $M \otimes_\CO F = K$.
	The group $K^\times$ acts on the set of fractional $R$-ideals $X_R$ and the set of invertible fractional $R$-ideals $X_R^\circ$. 
	Denote by $\ov{\Cl}(R) = K^\times \bs X_R$ the set of $K^\times$-homothety classes of fractional $R$-ideals and
	$\Cl(R) = K^\times \bs X_R^\circ$ the usual ideal class group of $R$. Then there is an action of $\Cl(R)$ on $\ov{\Cl}(R)$.
\item For a fractional $R$-ideal $M$, denote by $\Aut(M) = \left\{ x \in K^\times | xM = M \right\}$.
\item  $D_R =[\CO_K:R]^2 |\Delta_K|$ 
	is the absolute discriminant of $R$.
\end{itemize}
\end{prop}
\begin{proof}
	The proof is the same as that of \cite[Theorem 1.2]{Yun}. 
Note that in the proof of Theorem \ref{zeta-orb}, we have
	\[\zeta_R(s)=[\CO_K: R]^{-s}\zeta_K(s)\prod_{v}\widetilde{J}_{R_v}(s).\]
   In particular,
   \[\lim_{s \ra 1} (s-1)^{r+1} \zeta_R(s) = [\CO_K:R]^{-1} \lim_{s \ra 1} (s-1)^{r+1} \zeta_K(s) \prod_{v<\infty} \wt{J}_{R_v}(1).\]
   We have
   \[\lim_{s \ra 1} (s-1)^{r+1} \zeta_K(s) = \frac{2^{r_{1,K}}(2\pi)^{r_{2,K}}\#\Cl(\CO_K) R_K}{w_K \sqrt{|\Delta_K|}}.\]
  Applying \cite[Theorem 2.5(2) and Equation (3.4)]{Yun}, we have
   \[\#\Cl(\CO_K)\prod_{v<\infty} \wt{J}_{R_v}(1) = \#\Cl(R) \sum_{ [M] \in \Cl(R) \bs \ov{\Cl}(R)}
   	\frac{\#(\CO_K^\times/\Aut(M))}{\#\Stab_{\Cl(R)}([M])}.\]
   Equation \ref{residue-formula} now holds.
\end{proof}

\section{The proof of the main result}

We now combine the above results to give the proof of Theorem \ref{thm-main}. Denote by $d = n(n-1)/2$ 
and $n = r_1 + 2r_2$  with $r_1$ (resp. $r_2$) the number of real (resp. complex) places of $F$.

Let $\chi(x) \in \CO[x]$ be a regular monic polynomial with coefficients in $\CO$. Let $X$ be the variety representing the
set of $n \times n$ matrices with characteristic polynomial $\chi(x)$. 

\begin{lem}\label{lem:chi(0)}
    We may assume that $\chi(0) \not= 0$.
\end{lem}
\begin{proof}
    Let $a \in \CO$ and consider the following twist of $\chi(x)$
    \[\chi_1(x) = \chi(x-a).\]
    Denote by $X_1$ the variety representing the set of matrices with characteristic polynomial $\chi_1(x)$. There is a bijection
    \[X(\CO) \stackrel{\sim}{\lra} X_1(\CO), \quad A \mapsto A +aI_n\].
    In particular, 
    \[N(X,\CT) \sim N(X_1,T).\]
    
    Moreover, we have
    \[R_1 = \CO[x]/(\chi_1(x)) = \CO[x]/(\chi(x)) = R.\]
    In fact, all the terms in the right hand side of
    the equation for $C(\chi)$ in the main result (Theorem \ref{thm-main}) are invariant when we modify $\chi$ to
    $\chi_1$. 

    Therefore, the main result for $\chi$ and for $\chi_1$
    are equivalent. Finally, note that, if $\chi(0) = 0$, we may
    choose $a \in \CO$ such that $\chi_1(0) \not=0$.
    
\end{proof}

Let $\gamma \in X(\CO)$. By the above Lemma, we may assume 
$\det \gamma \not= 0$. 

As before, we denote by $G = \SL_n$, $\wt{G} = \GL_n$, 
$T$ the stabilizer of $\gamma$ in $G$ and $\wt{T}$ the stabilizer
of $\gamma$ in $\wt{G}$.

By Corollary \ref{cor:kappa}, each $\kappa \in H^1(F,T(\BA_{\bar{F}})/T(\bar{F}))^\vee$
	is in one-to-one correspondence with pairs $(E,\nu)$ where 
	\begin{itemize}
	\item $E$ is a field such that $F \subset E \subset K_i$ with $E/F$ a cyclic extension.
	\item $\nu$ is a faithful character on $\Gal(E/F)$.
	\end{itemize}
	Moreover, in this case, by Proposition \ref{ell-endo}, the endoscopic subgroup $H = \Res_{E/F}^1 \GL_m$ with $m[E:F] = n$.

	By Proposition \ref{arch} and Proposition \ref{non-arch}, all the subfields $E$ of $K$ contributes to $N(X,\CT)$ satisfying that
	$E/F$ is unramified everywhere (including the archimedean places). Moreover, for each $\kappa$ corresponding to $(E,\nu)$ with
	$E/F$ unramified everywhere, the $\kappa$-orbital integral
	\[
	\begin{aligned}
		\CO_\gamma^\kappa(f^\CT)  &=
		\CO_\gamma^{\kappa_\infty}(f_\infty^\CT)\CO_\gamma^{\kappa_\fin}(f_\fin) \\
		&\sim \frac{w_n \Vol(U_\infty)}{|\Delta_\chi|_\infty^{1/2}}
		[R_E:R]\delta(E)\CO_{\gamma_E}(1_{\wt{\fh}(\wh{\CO})}) \cdot \CT^{[F:\BQ] d} 
         \end{aligned}\]
	Now, we apply Theorem \ref{zeta-orb} to the order $R_E = \CO_E[\gamma_E] \cong \CO_E[\gamma]$ (with base field $E$)
	\[\CO_\gamma^\kappa(f^\CT)
		\sim \frac{w_n \Vol(U_\infty)[\CO_K:R]\delta(E)}{|\Delta_\chi|_\infty^{1/2}}
		\frac{ \zeta_{R_E}(s)}{ \zeta_K(s)}\Big|_{s=1} \cdot \CT^{[F:\BQ] d}.\]
	Note that for each $\kappa$ corresponding to $(E,\nu)$, the asymptotic behavior of $\CO_\gamma^\kappa(f^\CT)$
	depends only on $E$ and is independent of $\nu$. For each $E$, the cardinality of such $\nu$ is  $\phi([E:F])$. This gives
	\[N(X,\CT) \sim  \frac{v_\chi w_n \Vol(U_\infty) C_0^\Tam}{|\Delta_\chi|_\infty^{1/2}} 
	\sum_E \phi([E:F])[\CO_K:R]\delta(E)\frac{\zeta_{R_E}(s)}{\zeta_K(s)}\Big|_{s=1}
         \cdot \CT^{[F:\BQ] d}(\log \CT)^r.\]
	 Note that for any finite place $v$ (see Lemma \ref{lem-Delta})
	\[ |\Delta_{K/F}|_v^{1/2} = |\Delta_\chi|_v^{1/2} [\CO_{K_v}:R_v].\] Applying Proposition \ref{meas-comp}, we obtain that
	\[\lim_{\CT \ra +\infty} \frac{N(X,\CT)}{\CT^{[F:\BQ] d}(\log \CT)^r}\]
	is equal to
	\[\frac{1}{|\Delta_F|^{d}} \cdot \frac{\lim_{s\ra 1}(s-1)^{r+1} \zeta_K(s)}{\Res_{s=1}\zeta_F(s) \prod_{i=2}^n \zeta_F(i)}
	\cdot \frac{1}{|\Delta_\chi|_\fin^{1/2}[\CO_K:R]}\]
	times
	\[\frac{v_\chi w_n \Vol(U_\infty)}{|\Delta_\chi|_\infty^{1/2}} 
	\sum_E \phi([E:F])[\CO_K:R]\delta(E) \frac{\zeta_{R_E}(s)}{\zeta_K(s)}\Big|_{s=1}.\]
	We obtain 
        \[N(X,\CT) \sim \frac{v_\chi w_n \Vol(U_\infty)}{|\Delta_F|^{d}\Res_{s=1} \zeta_F(s) \prod_{i=2}^n \zeta_F(i)} 
	\sum_E \phi([E:F])\delta(E) \lim_{s\ra 1} (s-1)^{r+1}\zeta_{R_E}(s)
         \cdot \CT^{[F:\BQ] d}(\log \CT)^r.\]
	Now by Proposition \ref{prop-vol-arch}, 
	\[\Vol(U_\infty) = \frac{1}{\pi^{nr_2} \prod_{i=1}^n \prod_{v|\infty} \Gamma_v(i)}.\]
	We obtain Theorem \ref{thm-main}.

	Now Theorem \ref{thm-main-Q} is an immediate corollary of Theorem \ref{thm-main}.

\appendix

\section{The proof of Proposition \ref{prop: c_T} (by Taiwang Deng and Runlin Zhang)}

The goal of this section is to give a proof of Proposition \ref{prop: c_T}. We will deduce \eqref{eq:key24} by following \cite[\S 11]{SZ19}.

By definition 
\[
   B( \gamma_0 ,\CT) = \prod_{v\mid \infty} B_v( \gamma_0 ,\CT) := 
    \prod_{v\mid \infty} \left\{ x\in \gamma_0 \cdot G(F_v) \;\middle\vert\;
    \norm{x}_v \leq \CT
    \right\}.
\]
We are going to show that, upon removing a negligible portion from $B_v( \gamma_0 ,\CT)$ for each $v$, 
we have a good estimate of $\norm{w_J \cdot g}_v$ and hence of $d_J(g)$. Each time we will focus on a single place $v$.
When $F_v$ is complex our argument is the same  as that
in \cite[Proposition 11.8]{SZ19}.
Hence we are left with the case 
when $F_v$ is real. In what follows, the subscript 
$v$ may be omitted when no confusion can arise.
In what follows for the unexplained notations, we refer the readers to Section \ref{sec: equi-distribution-refinement}.

\subsection{Diagonalization over $\R$}

Take $g_0 \in G(F_v)$ such that $x_0:= \gamma_0 \cdot g_0$ is block diagonal:
\[
x_0=\diag(d_1,\ldots, d_{r_1+r_2}),
\]
where $d_i\in F_v$ for $1\leq i\leq r_{1}$ and 
\[
d_{i}=\begin{pmatrix}
a_i& b_i\\
-b_i & a_i
\end{pmatrix}, \qquad r_1<i\leq r_1+r_2
\]
where $a_i+\sqrt{-1}b_i\in \C$ is an eigenvalue of $\gamma_0$.
Recall that $(e_1,...,e_n)$ denotes the standard basis of $F^n$.
Each $(V_j \otimes F_v)\cdot g_0$ is $F_v$-spanned by a subcollection of $e_i$'s.
We may and do further require that $g_0$ preserves the order, namely,
\[
   (V_0 \otimes F_v) g_0 = F_v e_1 \oplus ... \oplus F_v e_{n_0},\;
   \cdots,
   (V_r \otimes F_v) g_0 = F_v e_{n_0+...+n_{r-1}+1}  \oplus ... \oplus F_v  e_{n}.
\] 
Therefore, for $J\in \CJ$,  
\[
  \big( \bigoplus_{j\in J } V_j \otimes F_v \big )\cdot g_0 =  \bigoplus_{t=1}^l \mathbb{R} e_{i_t}
\]
with $(i_t)$ identical to that in Eq.(\ref{eq:degree}). Therefore, for some constant $c_J \neq 0$, we have
\[
    w_J \cdot g_0 = c_J \wedge_{t=1}^l e_{i_t} =: c_J w_{J,v}.
\]

In the discussions below, we will often decompose an $n$-by-$n$ matrix $M$ into $(r_1+r_2)^2$-many blocks, which is compatible with the conjugate action of $x_0$. To avoid possible confusion, we write $M_{pq}$ for such a block whereas $M_{ij}$ is reserved for the usual $(i,j)$-th entry of a matrix. For instance $(x_0)_{pq}=d_{r_1}$ if $p=q=r_1+1$, but $(x_0)_{ij}= a_i$ if $i=j=r_1+1$.

\subsection{Iwasawa-like decomposition}
To proceed, we record the Iwasawa-like decomposition $G(F_v)= S N_0 N U$ from \cite{Sha}. Here $U$ is just $\SO(n)$. It remains to introduce the other players.

Let $T_{x_0}$ be the stabilizer of $x_0$ in $G(F_v)$, then $S$ is its identity connected component 
\[
  S = \left\{
      \diag \left(s_1, \ldots, s_{r_1}, s_{r_1+1}\begin{pmatrix}\cos \theta_1& \sin \theta_1\\ -\sin\theta_1& \cos\theta_1 \end{pmatrix}, \ldots, s_{r_1+r_2} \begin{pmatrix}\cos \theta_{r_2}& \sin \theta_{r_2}\\ -\sin\theta_{r_2}& \cos\theta_{r_2} \end{pmatrix} \right) 
      \right\}
\]
where $ s_i>0,\, \prod_{i\le r_1}s_i\prod_{i>r_1}s_i^2=1$ and $\theta_i \in \R/ 2\pi \mathbb{Z}$.

For $t \geq 0$, let $h(t):=\begin{pmatrix}
1&\sqrt{t}\\
0& 1
\end{pmatrix}$. And
\[
N_0:=\left\{
     h(\mathbf t)= \diag (1,\ldots, 1, h(t_1), \ldots, h(t_{r_2})) \;\middle\vert\; \mathbf t=(t_i)\in \R_{\geq 0}^{r_2} 
\right\}.
\]

For $p,q=1,\ldots, r_1+r_2$, let
\[
N_{pq}=\begin{cases}
\R,  &\text{\,  if \, }  p \leq r_1, \, q\leq r_1\\
\RM_{1\times 2}(\R), &\text{\, if \, } p\leq r_1, \, q>r_1\\
\RM_{2\times 1}(\R), & \text{\, if \, } p > r_1, \, q\leq r_1\\
\RM_{2\times 2}(\R), &\text{\, if \, }  p> r_1, \, q>r_1
\end{cases}
\]
and
\[
   \CN=\prod_{1\leq p < q \leq r_1+r_2} N_{pq}\cong \R^{\frac{n(n-1)}{2}-r_2}.
\]
Now
$N:= \left\{  n(\mathbf {x} ) \,\middle\vert\, \mathbf{x}  \in \CN \right\} $ is parametrized by $\CN$. If $\mathbf x=(x_{pq})\in \CN $, then
\[
   n(\mathbf x) :=(n_{pq})\;
   \text{ with }\;
   n_{pq} :=\begin{cases}
0,  &\text{\, if \, }p>q\\
1,  &\text{\, if \, }p=q\\
x_{pq}, &\text{\, if \, } p<q
\end{cases}
\]
where $1$ is the identity matrix of size $2$ if $p=q>r_1$.

Let $\mu_S, \mu_N$ and $\mu_U$ be Haar measures on $S,N$ and $U$. Let  $\mu_{N_0}:=\mathrm{d} t_1...\mathrm{d} t_{r_2}$ be a measure on $N_0$. 

\begin{lem}\label{eqn: Iwasawa-decomposition}
Taking the product yields a surjection $\Phi^{\mathrm{Iwa}}:S\times N_0 \times N \times U \to G(F_v)$. 
 And for some constant $c>0$, 
  \[
  \Phi^{\mathrm{Iwa}}_* \left(  \mu_S \otimes \mu_{N_0} \otimes \mu_N\otimes \mu_U
  \right) = c\cdot \mu_{G(F_v)}.
  \]
\end{lem}

\subsection{The ball $B_v(\gamma_0,\CT)$ using new coordinates}\label{section:ball_in_t_x_coordinates}

Set 
\[
D_{\CT}=\{(\mathbf t, \mathbf x): \norm{ n(\mathbf x)^{-1} h(\mathbf t)^{-1}x_0 h(\mathbf t) n(\mathbf x) }_v\leq \CT\}.
\]
Define
\begin{equation*}
\begin{aligned}
        \Psi: \R_{\geq}^{r_2} \times \CN \times U &  \to \gamma_0 \cdot G(F_{v})
        \\
         (\mathbf{t}, \mathbf{x} , u) & \mapsto \gamma_0 \cdot g_0 h(\mathbf{t}) n(\mathbf{x}) u.
\end{aligned}
\end{equation*}
Then by Lemma \ref{eqn: Iwasawa-decomposition}, $\Psi_*(\mathrm{d} \mathbf{t} \otimes \mathrm{d} \mathbf{x} \otimes \mu_U) = c\cdot \mu_{X(F_v)}$ for some $c>0$. Moreover,
\[
   \Psi^{-1}(B_v(\gamma_0,\CT)) = D_{\CT} \times U.
\]

\subsection{Diagonization over $\mathbb{C}$}

Consider
 $P_0=\frac{1}{\sqrt 2}\begin{pmatrix} 1&1\\ \sqrt{-1}&-\sqrt{-1}\end{pmatrix}\in \GL_2(\C)$ and 
\[
P := \diag(\underbrace{1,\ldots, 1}_{r_1}, \underbrace{P_0,\ldots, P_0}_{r_2}).
\]
Since $P_0^{-1}d_i P_0=\diag(a_i+b_i \sqrt{-1}, a_i-b_i\sqrt{-1})$ for $i>r_1$, we have 
$y_0:= P^{-1}x_0P$ is a diagonal matrix.
Now
\begin{equation*}
  \begin{aligned}
      \omega  :=   n(\mathbf x)^{-1} h(\mathbf t)^{-1}Py_0P^{-1} h(\mathbf t) n(\mathbf x)
  \end{aligned}
\end{equation*}
is block-upper-triangular. We will make it truly upper triangular. 

\subsection{Iwasawa decomposition for $\SL_2(\mathbb{C})$}

Let $\mathbf{x}' \in \mathcal{N}(\mathbb{C})$  be such that 
\[
n(\mathbf{x}') = P^{-1} n(\mathbf{x} ) P.
\]
 And let 
\[
     h'(\mathbf t)=\diag(\underbrace{1,\ldots, 1}_{r_1}, \underbrace{h'(t_1),\ldots, h'(t_{r_2})}_{r_2})
\]
where
\[
    h'( t)=P^{-1}h( t)P=
    \begin{pmatrix}1+\dfrac{i \sqrt t}{2} & -\dfrac{i \sqrt t}{2}\\[4pt]\dfrac{i \sqrt t}{2} & 1-\dfrac{i \sqrt t}{2}\end{pmatrix}\in \SL_2(\C).
\]
By Iwasawa decomposition for $\SL_2(\C)$, write $h'( t)=s(t)h''( t)k(t)$ (for $t\in \R_{\geq 0}$) with
\[
s(t):=\begin{pmatrix}\dfrac{\sqrt2}{\sqrt{t+2}}&0\\[6pt]0&\dfrac{\sqrt{t+2}}{\sqrt2}\end{pmatrix},
\quad
h''( t):=\begin{pmatrix}1&\frac{t}{2}-i \sqrt t\\0&1\end{pmatrix}, 
\quad
k(t):=
\dfrac{\sqrt2}{\sqrt{t+2}}
\begin{pmatrix}
1+\dfrac{i \sqrt t}{2}&\dfrac{i \sqrt t}{2}\\[6pt]
\dfrac{i \sqrt t}{2}&1-\dfrac{i \sqrt t}{2}
\end{pmatrix}\in \SU(2).
\]
Let 
\begin{align*}
k(\mathbf t)& :=\diag(\underbrace{1,\ldots, 1}_{r_1}, \underbrace{k(t_1),\ldots, k(t_{r_2})}_{r_2})\in \SU(n),\\
 s(\mathbf t)& :=\diag(\underbrace{1,\ldots, 1}_{r_1}, \underbrace{s(t_1),\ldots, s(t_{r_2})}_{r_2}),\\
h''(\mathbf t)& :=\diag(\underbrace{1,\ldots, 1}_{r_1}, \underbrace{h''(t_1),\ldots, h''(t_{r_2})}_{r_2}),\\
n(\mathbf t, \mathbf x') & :=k(\mathbf t) n(\mathbf x')k(\mathbf t)^{-1}.
\end{align*}
Set 
\[
\omega'':=k(\mathbf t)P^{-1}\omega P k(\mathbf t)^{-1}
= n(\mathbf t, {\mathbf x'})^{-1}h''(\mathbf t)^{-1} y_0 h''(\mathbf t) n(\mathbf t, {\mathbf x'}).
\] 
Let us remark that $y_0 $ is diagonal, $h''(\mathbf t) n(\mathbf t, {\mathbf x'})$ is upper triangular unipotent and $\omega''$ is upper triangular and conjugate to $y_0$.

\subsection{Definition of $\mathcal{D}_{\CT}^{\varepsilon}$}
For $\varepsilon >0$, let
\[
D_{\CT}^\varepsilon:=\left\{
(\mathbf t, \mathbf x)\in D_{\CT} \;\middle\vert\; \|\omega''_{i, i+1}\|\geq\varepsilon \CT, \;\forall \, i=1,...,n-1
\right\}.
\]

\begin{prop}\label{prop:generaic_ball}
There exists $T_0\ge 1$ such that for every $\CT\ge T_0$ and every $0<\varepsilon\le 1$,
\[
\frac{\operatorname{vol}(D_\CT\setminus D_\CT^{\varepsilon})}{\operatorname{vol}(D_\CT)}\ll \varepsilon,
\]
where the implied constant depends only on $x_0$ and on the chosen norms. In particular,
\[
\liminf_{\CT\to\infty}\frac{\operatorname{vol}(D_\CT^{\varepsilon})}{\operatorname{vol}(D_\CT)}\ge 1-C\varepsilon.
\]
Consequently, if $\varepsilon=\varepsilon(\CT)\to 0$ as $\CT\to\infty$, then
\[
\lim_{\CT\to\infty}\frac{\operatorname{vol}(D_\CT^{\varepsilon(\CT)})}{\operatorname{vol}(D_\CT)}=1.
\]
\end{prop}

Let $h(\mathbf t)^{-1}x_0 h(\mathbf t)=(z_1,\ldots, z_{r_1+r_2})$ where 
\[
z_p=\begin{cases}
d_p, \text{\,  if \, } p\leq r_1\\
h(t_{p-r_1})^{-1}d_p h(t_{p-r_1}) \text{\, if \, }p>r_1.
\end{cases}
\]
Let $n(\mathbf x)^{-1}z(\mathbf t) n(\mathbf x)=(\omega_{pq})_{p,q=1}^{r_1+r_2}:=\omega$ with $\omega_{pq}\in N_{pq}$ such that
\[
\omega_{pq}=\begin{cases}
z_p, &\text{\ if \ } p=q\\
L_{pq}(x_{pq})+Q_{pq}((x_{k\ell})_{0<\ell-k<q-p}), & \text{\ if \ } p<q
\end{cases}
\]
where $L_{pq}: N_{pq}\rightarrow N_{pq} (p<q)$ is defined by 
$$L_{pq}(x)=z_p x-x z_q $$
 and
\[
Q_{pq}: \prod_{0<k-\ell<q-p}N_{k\ell}\longrightarrow N_{pq}
\]
is a polynomial map for $p<q-1$ and $Q_{pq}\equiv 0$ if $q-p=1$.

\begin{proof}
Let $M=r_1+r_2$. 
For $1\le p<q\le M$ consider the  linear map
\[
L_{pq}:N_{pq}\longrightarrow N_{pq},\qquad L_{pq}(u)=d_p u-u d_q.
\]
Because the spectra of $d_p$ and $d_q$ are disjoint, $L_{pq}$ is invertible. Also, if
\[
C_{pq,\mathbf t}:N_{pq}\longrightarrow N_{pq},\qquad C_{pq,t}(u)=h(t_p)^{-1}uh(t_q),
\]
then, since $\det h(t)=1$, one has $|\det C_{pq,\mathbf t}|=1$ and
\[
S_{pq,\mathbf t}=C_{pq,\mathbf t}^{-1}\circ L_{pq}\circ C_{pq,\mathbf t}.
\]
Hence
\[
|\det S_{pq,\mathbf t}|=|\det L_{pq}|,
\]
which is independent of $\mathbf t$.

Note that it follows from the definition
\begin{equation}\label{eqn: ball-unit}
D_{\CT}=\{(\mathbf t, \mathbf x):\sum_{p=1}^{r_1+r_2} \| z_p\|^2+\sum_{1\leq p<q\leq r_1+r_2}\|\omega_{pq}\|^2\leq \CT^2\}.
\end{equation}
For every fixed $\mathbf t$ the map
\[
\mathbf x\longmapsto \omega=(\omega_{pq})_{1\le p<q\le M}
\]
is a polynomial diffeomorphism, and its Jacobian is the constant
\[
J_0:=\prod_{1\le p<q\le M}|\det L_{pq}|^{-1}.
\]
Therefore the volume form $\mathrm{d}\mathbf t\,\mathrm{d}\mathbf x$ becomes $J_0\,\mathrm{d} t\,\mathrm{d}\omega$, where
\[
   \mathrm{d}\omega:=\prod_{1\le p<q\le M}\mathrm{d}\omega_{pq}
\]
is the Euclidean product measure on $\prod_{p<q}N_{pq}$.

Now  \eqref{eqn: ball-unit} shows that there exist constants $A_0\ge 0$ and $c_a>0$
($1\le a\le q$), depending only on $x_0$ and the chosen norms, such that under the coordinates $(\mathbf t,\omega)$,
\[
D_\CT\simeq \mathcal D_{\CT}:=\left\{(\mathbf t,\omega)\in \mathbb R_{\ge 0}^q\times \prod_{p<q}N_{pq}:
A_0+\sum_{a=1}^q c_a(t_a^2+4t_a)+\sum_{p<q}\|\omega_{pq}\|^2\le \CT^2\right\},
\]
and
\[
\operatorname{vol}(D_\CT)=J_0\,\operatorname{vol}(\mathcal D_\CT).
\]
Set
\[
N:=r_2+\sum_{1\le p<q\le M}\dim N_{pq}=\frac{n(n-1)}2.
\]

We first record the crude growth of $\operatorname{vol}(D_\CT)$. There exists $T_0\ge 1$ such that for all $\CT\ge T_0$,
\[
\operatorname{vol}(D_\CT)\asymp \CT^N.
\]
Indeed, the upper bound follows immediately from the inequalities $t_a\ll \CT$ and $\|\omega_{pq}\|\le \CT$ on $\mathcal D_\CT$. For the lower bound, choose $\kappa>0$ so small that
\[
\sum_{a=1}^q c_a(\kappa^2+4\kappa)+\kappa^2<\frac14.
\]
If $\CT\ge T_0$ is large enough so that $A_0\le \CT^2/4$, then the region
\[
0\le t_a\le \kappa \CT\quad (1\le a\le q),
\qquad
\sum_{p<q}\|\omega_{pq}\|^2\le \kappa^2\CT^2
\]
is contained in $\mathcal D_\CT$, which gives the required lower bound.

For $1\le i\le n-1$, define
\[
E_i(\CT,\varepsilon):=\{(\mathbf t,\mathbf x)\in D_\CT:\ \|\omega''_{i,i+1}\|<\varepsilon \CT\}.
\]
Then
\[
D_\CT\setminus D_\CT^{\varepsilon}=\bigcup_{i=1}^{n-1}E_i(\CT,\varepsilon).
\]
It is therefore enough to prove that, uniformly for $\CT\ge T_0$,
\[
\operatorname{vol}(E_i(\CT,\varepsilon))\ll \varepsilon \CT^N
\qquad (1\le i\le n-1).
\]
We distinguish two cases.

\smallskip
\noindent\emph{Case 1: $i=r_1+2a-1$ for some $1\le a\le r_2$.}
Then $i,i+1$ are the two coordinates belonging to the $a$-th complex $2\times 2$ block. Since $n(\mathbf x')$ has no off-diagonal entry inside a block and $k(\mathbf t)$ is block diagonal, the $(i,i+1)$-entry of $h''(\mathbf t)n(\mathbf t,\mathbf x')$ is exactly
\[
\eta_a(t_a):=\frac{t_a}{2}-i\sqrt{t_a}.
\]
Hence
\[
\omega''_{i,i+1}=(y_i-y_{i+1})\eta_a(t_a),
\]
so
\[
|\omega''_{i,i+1}|^2=|y_i-y_{i+1}|^2\left(\frac{t_a^2}{4}+t_a\right).
\]
Thus $|\omega''_{i,i+1}|<\varepsilon \CT$ implies $|\eta_a(t_a)|\ll \varepsilon \CT$, and therefore
\[
t_a=2\bigl(\sqrt{1+|\eta_a(t_a)|^2}-1\bigr)\le 2|\eta_a(t_a)|\ll \varepsilon \CT.
\]
So, after fixing all variables except $t_a$, the bad set in the $t_a$-direction has length $O(\varepsilon \CT)$. Since all remaining $r_2-1$ parameters $t_b$ range over intervals of length $O(\CT)$ and the $\omega$-variables lie in a Euclidean ball of radius $O(\CT)$, Fubini gives
\[
\operatorname{vol}(E_i(\CT,\varepsilon))\ll (\varepsilon \CT)\,\CT^{r_2-1}\,\CT^{N-r_2}=\varepsilon \CT^N.
\]

\smallskip
\noindent\emph{Case 2: $i\notin\{r_1+1,r_1+3,\dots,r_1+2r_2-1\}$.}
Then $i$ and $i+1$ lie in two consecutive blocks. Let these block indices be $p=\delta(i)$ and $q=\delta(i+1)=p+1$. Because $P$ and $k(\mathbf t)$ are block diagonal and unitary on each block, the scalar entry $\omega''_{i,i+1}$ is obtained from the block $\omega_{pq}\in N_{pq}$ by first applying an isometry of Euclidean spaces and then taking one real or complex coordinate. Equivalently, there exist $\rho_i\in\{1,2\}$ and a surjective linear map
\[
\ell_{i,t}:N_{pq}\longrightarrow \mathbb R^{\rho_i},\qquad \|\ell_{i,t}\|=1,
\]
such that
\[
\|\omega''_{i,i+1}\|=\|\ell_{i,t}(\omega_{pq})\|.
\]
Fix all variables except $\omega_{pq}$. Then $\omega_{pq}$ ranges inside a Euclidean ball $B_{pq}(R)\subset N_{pq}$ with $R\le \CT$. After an orthogonal change of coordinates on $N_{pq}$, the condition $\|\ell_{i,t}(\omega_{pq})\|<\varepsilon \CT$ becomes the condition that the first $\rho_i$ coordinates lie in a ball of radius $\varepsilon \CT$. Hence
\[
\operatorname{vol}_{N_{pq}}\{u\in B_{pq}(R):\ \|\ell_{i,t}(u)\|<\varepsilon \CT\}
\ll (\varepsilon \CT)^{\rho_i}R^{\dim N_{pq}-\rho_i}
\le \varepsilon \CT^{\dim N_{pq}},
\]
because $R\le T$, $\rho_i\in\{1,2\}$, and $0<\varepsilon\le 1$. Integrating over the remaining variables again yields
\[
\operatorname{vol}(E_i(\CT,\varepsilon))\ll \varepsilon \CT^N.
\]

Combining the two cases and summing over $i=1,\dots,n-1$, we obtain
\[
\operatorname{vol}(D_\CT\setminus D_\CT^{\varepsilon})
\le \sum_{i=1}^{n-1}\operatorname{vol}(E_i(\CT,\varepsilon))
\ll \varepsilon \CT^N
\ll \varepsilon\,\operatorname{vol}(D_\CT),
\]
for all $\CT\ge T_0$. This proves the first claim.

The remaining assertions are immediate:
\[
\frac{\operatorname{vol}(D_\CT^{\varepsilon})}{\operatorname{vol}(D_\CT)}
=1-\frac{\operatorname{vol}(D_\CT\setminus D_\CT^{\varepsilon})}{\operatorname{vol}(D_\CT)}
\ge 1-C\varepsilon,
\]
so
\[
\liminf_{\CT\to\infty}\frac{\operatorname{vol}(D_\CT^{\varepsilon})}{\operatorname{vol}(D_\CT)}\ge 1-C\varepsilon.
\]
Finally, if $\varepsilon=\varepsilon(\CT)\to 0$, then
\[
0\le 1-\frac{\operatorname{vol}(D_\CT^{\varepsilon(\CT)})}{\operatorname{vol}(D_\CT)}\le C\varepsilon(\CT)\xrightarrow[\CT\to\infty]{}0,
\]
and therefore
\[
\lim_{\CT\to\infty}\frac{\operatorname{vol}(D_\CT^{\varepsilon(\CT)})}{\operatorname{vol}(D_\CT)}=1.
\]
This completes the proof.
\end{proof}

\subsection{Estimate of $\norm{ w_{J} \cdot g_0 h(\mathbf{t}) n(\mathbf{x})}$}
\begin{prop}\label{prop: top-degree-term}
Let $J\in \CJ$ and recall $w_{J,v}=\wedge_{t=1}^l e_{i_t}$. We have for some  constant $\normm{ c_{J,v}}=1$,
\[
    w_{J,v} \cdot h(\mathbf t)n(\mathbf x) k(\mathbf{t}) ^{-1}
    = c_{J,v} c(i_1,\ldots, i_\ell)\prod_{j=1}^{\ell}\prod_{p=j}^{i_j-1}\omega''_{p, p+1}(e_1\wedge e_2\wedge\cdots e_\ell ) +  \ldots
\]
where $c(i_1,\ldots, i_\ell)$ is the constant in \cite[Lemma 11.4]{SZ19} and we omit the terms of polynomials in variables $\omega''_{p,q}(p<q)$ of degrees 
smaller than $\deg J= \sum_{t =1}^\ell (i_t  - t)$.
\end{prop}

\begin{proof}
It follows from \cite[Proposition 11.5]{SZ19} that  
\[
w_{J,v} \cdot h''(\mathbf t)n(\mathbf t, \mathbf x')=c(i_1,\ldots, i_\ell)\prod_{j=1}^{\ell}\prod_{p=j}^{i_j-1}\omega''_{p, p+1}(e_1\wedge e_2\wedge\cdots e_\ell )+\ldots
\]
where $c(i_1,\ldots, i_\ell)$ is the constant in \cite[Lemma 11.4]{SZ19} and we omit the terms of polynomials in variables $\omega''_{p,q}(p<q)$ of degrees 
smaller than $\sum_{j=1}^\ell (i_j-j)$. 
Note that $w_{J,v}$ is fixed by $s(\mathbf{t})$ and $w_{J,v}\cdot P^{-1} = c^{-1}_{J,v} w_{J,v}$ with $\normm{c_{J,v}}_v=1$. 
Therefore, by unwrapping definitions, 
\begin{equation*}
   \begin{aligned}
          w_{J,v} \cdot h''(\mathbf{t} )  n(\mathbf{t}, \mathbf{x'}) 
          & =
          w_{J,v} \cdot \big( s(\mathbf{t}) h''(\mathbf{t}) k(\mathbf{t})  k(\mathbf{t}) ^{-1}  \big)
           \big( k(\mathbf{t}) P^{-1} n(\mathbf{x}) P  k(\mathbf{t})^{-1} \big)
           \\
           &=
           w_{J,v} \cdot P^{-1} h(\mathbf{t}) P P^{-1} n(\mathbf{x}) P  k(\mathbf{t})^{-1}
           \\
           &=
           c_{J,v} ^{-1}
            w_{J,v} \cdot h(\mathbf{t}) n (\mathbf{x}) k(\mathbf{t})^{-1}.
   \end{aligned}
\end{equation*}
\end{proof}

Note that $(\mathbf{t},\mathbf{x}) \in D_{\CT} \implies \normm{\omega''_{ij}}_v \leq T$ for every $(i,j)$. 
Also recall that $w_J \cdot g_0 = c_J w_{J,v}$.
The above proposition leads to the following corollary
\begin{cor}\label{cor: top-degree-term}
There exists $0<c_{\varepsilon}<1<C$ such that for every $\CT>1$ and $(\mathbf{t},\mathbf{x}) \in D_{\CT}$,
\[
    \norm{   w_{J} \cdot g_0 h(\mathbf{t}) n(\mathbf{x}) 
    }_v \leq C \cdot \CT^{\deg J}.
\]
And if additionally $(\mathbf{t},\mathbf{x}) \in D_{\CT}^{\varepsilon}$, then
\[
   \norm{   w_{J} \cdot g_0 h(\mathbf{t}) n(\mathbf{x}) 
    }_v \geq c_{\varepsilon} \cdot \CT^{\deg J}.
\]
\end{cor}
The same corollary also holds for a complex place $v$, replacing $\CT^{\deg J}$ by $\CT^{2\deg J}$, and the proof is omitted as it is easier and essentially contained in \cite{SZ19}.

\subsection{Proof of  Proposition \ref{prop: c_T} }
\begin{proof}
By Lemma \ref{lem: modified-def-Omega}, we have 
\[
   \Omega_{g,\varepsilon}\cong 
   \Big\{
  \mathbf{y} \in H \;\Big \vert\;
    \sum_{j\in J}y_j\geq \log(\varepsilon)-  \sum_{v\mid \infty } \log(\| w_J\cdot g_v\|_v), \;\forall\, J \in \CJ
    \Big\}.
\]
Let $D_{\CT}:= \prod_{v\mid \infty} D_{\CT,v}$ and write $\mathbf{t}= \prod_{v\mid \infty} \mathbf{t}_v$, $\mathbf{x}= \prod_{v\mid \infty} \mathbf{x}_v$.
By Section \ref{section:ball_in_t_x_coordinates}, Equation (\ref{eq:key24}) is equivalent to
\begin{equation}\label{eq:key24again}
     \int_{ (\mathbf{t}, \mathbf{x}) \in D_{\CT} } 
     \mu( \Omega_{g_0 h(\mathbf{t}) n(\mathbf{x})  ,\varepsilon}) \, \mathrm{d}\mathbf{t} \,\mathrm{d}\mathbf{x}
     \sim
      v_{\chi}(\log \CT)^r \int_{D_{\CT}}  \mathrm{d}\mathbf{t} \,\mathrm{d}\mathbf{x}.
\end{equation}
For simplicity we write $L_{\CT}$ and $R_{\CT}$ for them.

By Corollary \ref{cor: top-degree-term}, for $( \mathbf{t}, \mathbf{x} ) \in D_{\CT}$,
\begin{equation*}
\begin{aligned}
           \mu(  \Omega_{g_0 h(\mathbf{t}) n(\mathbf{x})  ,\varepsilon} )
           &\leq \mu
           \Big(
           \Big\{
               \mathbf{y} \in H \;\Big \vert \;
               \sum_{j\in J} y_{j} \geq \log \varepsilon - \log (C \CT^{\deg(F/\Q) \deg J })
           \Big\}
           \Big)
           \\
          & = (\log \CT)^r \mu \Big(
           \Big\{
                \mathbf{y} \in H \;  \Big \vert \;
               \sum_{j\in J} y_{j} \geq \frac{\log \varepsilon - \log C}{\log \CT} - {\deg(F/\Q) \deg J }
            \Big\}
           \Big)
           \\
           &=  (\log \CT)^r v_{\chi} + o(1).
\end{aligned}
\end{equation*}
Similarly, for $\eta>0$ and $( \mathbf{t}, \mathbf{x} ) \in D_{\CT}^{\eta}$,
\[
   \mu(  \Omega_{g_0 h(\mathbf{t}) n(\mathbf{x})  ,\varepsilon} )
   \geq  (\log \CT)^r v_{\chi} + o_{\eta}(1)
\]
Invoking Proposition \ref{prop:generaic_ball}, one has 
\[
   1- C\eta  \leq 
   \liminf_{\CT \to \infty}  \frac{L_{\CT}}{R_{\CT}} \leq  \limsup_{\CT \to \infty} \frac{L_{\CT}}{R_{\CT}} \leq  1.
\]
Letting $\eta \to 0$ completes the proof.
\end{proof}

\bibliographystyle{plain}
\bibliography{biblio}

\end{document}